\documentclass[12pt,reqno]{amsart}
\usepackage{graphicx}
\usepackage{amssymb,amsmath}
\usepackage{amsthm}
\usepackage{hyperref}
\usepackage{color}
\usepackage{subcaption}

\numberwithin{equation}{section}
\numberwithin{figure}{section}

\newtheorem{theorem}{Theorem}[section]
\newtheorem{proposition}[theorem]{Proposition}
\newtheorem{remark}[theorem]{Remark}
\newtheorem{lemma}[theorem]{Lemma}

\newtheorem{assumption}[theorem]{Assumption}

\DeclareMathOperator{\C}{\mathbb{C}}

\DeclareMathOperator{\R}{\mathbb{R}}

\DeclareMathOperator{\re}{\mathrm{e}}
\newcommand{\de}{\mathrm{d}}

\begin{document}
		
\title[Extinction of multiple shocks]{\bf On the extinction of multiple shocks \\  in  scalar viscous conservation laws}

\author{Jeanne Lin}
\address[J.~Lin]{Department of Mathematics and Statistics, McMaster University, Hamilton, Ontario, Canada, L8S 4K1}
\email{lin64@mcmaster.ca}

\author{Dmitry E. Pelinovsky}
\address[D.~E.~Pelinovsky]{Department of Mathematics and Statistics, McMaster University, Hamilton, Ontario, Canada, L8S 4K1 and 
	Department of Applied Mathematics, Nizhny Novgorod State Technical University, 24 Minin street, 603950 Nizhny Novgorod, Russia}
\email{dmpeli@math.mcmaster.ca}

\author{Bj\"{o}rn de Rijk}
\address[B.~de Rijk]{Department of Mathematics, Karlsruhe Institute for Technology, Englerstra\ss e 2, 76131 Karlsruhe, Germany}
\email{bjoern.rijk@kit.edu} 

\begin{abstract} 
We are interested in the dynamics of interfaces, or zeros, of shock waves in general scalar viscous conservation laws with a locally Lipschitz continuous flux function, such as the modular Burgers' equation. We prove that all interfaces coalesce within finite time, leaving behind either a single interface or no interface at all. Our proof relies on mass and energy estimates, regularization of the flux function, and an application of the Sturm theorems on the number of zeros of solutions of parabolic problems. Our analysis yields an explicit upper bound on the time of extinction in terms of the initial condition and the flux function. Moreover, in the case of a smooth flux function, we characterize the generic bifurcations arising at a coalescence event with and without the presence of odd symmetry. We identify associated scaling laws describing the local interface dynamics near collision. Finally, we present an extension of these results to the case of anti-shock waves converging to asymptotic limits of opposite signs. Our analysis is corroborated by numerical simulations in the modular Burgers' equation and its regularizations.
\end{abstract}

\date{\today}
\maketitle


\section{Introduction}

We consider shock and anti-shock waves with multiple interfaces in the scalar viscous conservation law
\begin{align} 
u_t = u_{xx} + f(u)_x , \qquad t \geq 0, \quad x \in \R, \quad u(t,x) \in \R, \label{modBurgers}
\end{align}
where $f \colon \R \to \R$ is a locally Lipschitz continuous flux function. A classical example is the viscous Burgers' equation with $f(u) = u^2$. Our regularity assumption on $f$ allows for nonsmooth choices such as $f(u) = |u|$, yielding the modular Burgers' equation which has been used to model inelastic dynamics of particles with piecewise interaction potentials~\cite{Hedberg1,RH18} and whose behavior has been studied analytically and numerically in~\cite{PELI,PEDR,Rad}.

\emph{Shock waves} are solutions of~\eqref{modBurgers} with initial data $u_0(x)$ converging to {nonzero} asymptotic limits $\phi_\pm$ as $x \to \pm \infty$, which satisfy $\phi_+ \neq 
\phi_-$ and obey the Gel'fand-Oleinik entropy condition
\begin{align} \frac{f(\phi_{\min}) - f(z)}{z - \phi_{\min}} > \frac{f(\phi_-) - f(\phi_+)}{\phi_+ - \phi_-}, \qquad z \in (\phi_{\min},\phi_{\max}),  \label{Gelfand} \end{align}
where we denote $\phi_{\min} = \min\{\phi_-,\phi_+\}$ and $\phi_{\max} =  \max\{\phi_-,\phi_+\}$. On the other hand, \emph{anti-shock waves} are solutions of~\eqref{modBurgers} with initial data $u_0(x)$ converging to {nonzero} asymptotic limits $\phi_\pm$, which satisfy $\phi_+ \neq \phi_-$ and do not fulfill the entropy condition~\eqref{Gelfand}. 

The Gel'fand-Oleinik entropy condition~\eqref{Gelfand} is consistent with the existence of \emph{traveling shock waves}, which are solutions of~\eqref{modBurgers} of the form $u(t,x) = \phi(x-ct)$, where $c \in \R$ denotes the propagation speed and the profile $\phi \colon \R \to \R$ solves the scalar problem
\begin{align*}
0 = \phi_\xi + c\left(\phi - \phi_-\right) + f(\phi) - f(\phi_-).
\end{align*}
Here, the profile $\phi(\xi)$ converges to the asymptotic limits $\phi_\pm$ as $\xi \to \pm \infty$ and the speed is given by the Rankine-Hugoniot condition
\begin{align} \label{rankine} c = \frac{f(\phi_-) - f(\phi_+)}{\phi_+ - \phi_-}.\end{align}
The traveling shock-wave solution $u(t,x) = \phi(x-ct)$ defined for $\phi_- \neq \phi_+$ exists if and only if the entropy condition~\eqref{Gelfand} is fulfilled. 

Traveling shock waves form an important class of asymptotic solutions of~\eqref{modBurgers} in the sense that they serve as global attractors for shock waves. More precisely, for twice continuously differentiable flux functions $f$, it has been proven in~\cite{FRSE,ILOL} that any shock-wave solution of the viscous conservation law~\eqref{modBurgers} converges as $t \to \infty$ in both $L^1$- and $L^\infty$-norm to a traveling shock wave, which necessarily possesses the same asymptotic limits $\phi_\pm$ at $\pm \infty$.

We are interested in the temporal dynamics of zeros, so-called \emph{interfaces}, for shock- and anti-shock wave solutions of the viscous conservation law~\eqref{modBurgers}. In our analysis we distinguish between three classes of initial data $u_0$, where both $\phi_+$ and $\phi_-$ are nonzero:
\begin{itemize}
    \item {\bf Class I:} $u_0(x)$ converges to asymptotic limits $\phi_\pm$ of \emph{opposite signs} as $x \to \pm \infty$, which obey the Gel'fand-Oleinik entropy condition~\eqref{Gelfand};
    \medskip
    \item {\bf Class II:} $u_0(x)$ converges to asymptotic limits $\phi_\pm$ of the \emph{same sign} as $x \to \pm \infty$;
    \medskip
    \item {\bf Class III:} $u_0(x)$ converges to asymptotic limits $\phi_\pm$ of \emph{opposite signs} as $x \to \pm \infty$, which do \emph{not} satisfy the entropy condition~\eqref{Gelfand}. 
\end{itemize}
We note that solutions of~\eqref{modBurgers} with initial data of the class I are shock waves, whereas solutions of~\eqref{modBurgers} with initial data of class III are anti-shock waves. Although solutions of~\eqref{modBurgers} with initial data of class II can be either shock or anti-shock waves, it is not necessary to distinguish between them in our analysis. 

In addition to the above assumptions, we require that our initial datum $u_0$ is uniformly continuous and bounded, and that $u_0 - \phi_\pm$ is $L^1$-integrable on $\R_\pm$. Then, by the comparison principle and standard parabolic regularity theory~\cite{LUN}, the solution of~\eqref{modBurgers} with initial condition $u_0$ stays bounded and is continuously differentiable for all positive times, while maintaining its asymptotic limits $\phi_\pm$ at $\pm \infty$. Nevertheless, if the flux function $f$ is not continuously differentiable, as in the case of the modular Burgers' equation, the second derivative of  the solution of~\eqref{modBurgers} may be discontinuous~\cite{PELI}. 

The classical Sturm Theorems yield that in parabolic semilinear equations the number of zeros of solutions is nonincreasing over time. Moreover, if at some time $t_0$ the solution has an (isolated) multiple zero $x_0$, then, in a sufficiently small neighborhood of $x_0$, the number of zeros strictly decreases when $t$ passes through $t_0$. We refer to~\cite{GAHA} for a survey on Sturm's Theorems and their applications.

In this paper we study \emph{finite-time coalescence} of interfaces. In preliminary work~\cite{PEDR}   we showed that the evolution of odd shock waves with three symmetric interfaces in the modular Burgers' equation leads to a finite-time coalescence of these interfaces to a single interface and we conjectured a scaling law for the local interface dynamics near the collision event based on data fitting. In this work we extend these results to general viscous conservation laws of the form~\eqref{modBurgers} and establish finite-time coalescence of interfaces for all solutions with initial data of class I or II and thus, of all shock-wave solutions. Moreover, we show that in the specific case of the modular Burgers' equation, solutions with initial data of class III, i.e.~anti-shock waves, can also exhibit finite-time coalescence of interfaces.

For solutions of~\eqref{modBurgers} with initial data of class I, we establish that all interfaces must coalesce to a single interface within finite time. The argument generalizes the idea from~\cite{PEDR} and relies on a differential inequalities for the masses of $u(t,x) - \phi_+$ and  $u(t,x) - \phi_-$  measured with respect to the position of the interface, in combination with  smooth approximation of the flux function and an application of the Sturm Theorem from~\cite{ANG}. Our analysis yields an explicit upper bound on the time at which all interfaces have collapsed to a single interface. We emphasize that although the results in~\cite{FRSE,ILOL} imply that solutions with initial data of class I converge in $L^1$- and $L^\infty$-norms to a traveling shock wave, which must necessarily be strictly monotone and thus, has precisely a single interface, this is not sufficient to conclude finite-time coalescence to a single interface because interfaces of the solution might accumulate close to the interface of the associated traveling shock wave. 

Initial data $u_0$ of class II can always be bounded from above or below by a smooth function $\tilde{u}_0$, which satisfies $\tilde{u}_0(x) \to \tilde{u}_\infty$ as $x \to \pm \infty$, where $\tilde{u}_\infty \neq 0$ has the same sign as $\phi_\pm$. For twice continuously differentiable flux functions $f$, the finite-time extinction of all interfaces of the solution $\tilde{u}(t,\cdot)$ of~\eqref{modBurgers} with initial condition $\tilde{u}_0$ follows by evoking the result from~\cite{FRSE} that $\tilde{u}(t,\cdot)$ converges in $L^\infty$-norm to the constant state $\tilde{u}_\infty$ as $t \to \infty$.  Consequently, the comparison principle yields the finite-time extinction of interfaces of the solution $u(t,\cdot)$ of~\eqref{modBurgers} with initial condition $u_0$. Yet, the result in~\cite{FRSE} does not provide an explicit upper bound on the extinction time and does not readily apply to the current setting of locally Lipschitz continuous flux functions. To extend the conclusion to our setting, we apply a softer argument based on energy estimates, smooth approximation of the flux function, conservation of mass and the Gagliardo-Nirenberg inequality to yield an explicit upper bound on the time at which all interfaces of $\tilde{u}(t,\cdot)$, and thus, also of $u(t,\cdot)$, have gone extinct, cf.~Remark~\ref{explicittime}.

Whether solutions of~\eqref{modBurgers} with initial data of class III do exhibit finite-time coalescence of interfaces to a single interface is currently an open problem. Since the entropy condition~\eqref{Gelfand} is not fulfilled, there exists no traveling shock to which $u(t,\cdot)$ can converge in norm as $t \to \infty$. To shed some light on this open question, we consider anti-shock waves with initial data of class III in the modular Burgers' equation with flux function $f(u) = |u|$. Our analysis indicates that, although all interfaces coalesce to a single interface in this case, the anti-shock wave converges \emph{locally uniformly} to $0$ as $t \to \infty$ suggesting that obtaining a result in general might be subtle or even false. Colloquially speaking, since the solution profile can converge to $0$ uniformly, locally near interfaces, diffusion might be too weak to enforce coalescence of interfaces. In fact, recent results~\cite{GASE} imply that the $\omega$-limit set (in the locally uniform topology induced by $L^\infty_{\mathrm{loc}}(\R)$) of bounded solutions of scalar viscous conservation laws~\eqref{modBurgers} can be complicated in the sense that it can contain a solution that is neither a traveling shock nor a constant, underlining a fundamental difference between shock waves and general bounded solutions of~\eqref{modBurgers}.

In addition to establishing finite-time coalescence of interfaces of shock and anti-shock waves, we study the interface dynamics about a coalescence event in the case of a smooth flux function $f$. If a coalescence event occurs for a solution $u(t,x)$ of~\eqref{modBurgers} at some time $t = t_0$ and point $x = \xi_0$, it must hold that $u_x(t_0,\xi_0) = 0$ and it follows from one of the classical Sturm Theorems~\cite{ANG} that there exist $\delta > 0$ and a neighborhood $U \subset \R$ of $\xi_0$ such that for $t \in (t_0-\delta,t_0)$, there are at least two interfaces in $U$ and for $t \in (t_0,t_0+\delta)$, there is at most one interface in $U$. Without the presence of additional symmetries, one generically has $u_{xx}(t_0,\xi_0) \neq 0$. We show that in this situation {\em a fold bifurcation} occurs. That is, there are precisely two interfaces $\xi_1(t) < \xi_2(t)$ in $U$ for $t \in (t_0-\delta,t_0)$ and no interfaces in $U$ for $t \in (t_0,t_0+\delta)$. Moreover, we obtain the scaling law 
\begin{equation}
\label{law-fold}
\xi_{1,2}(t) - \xi_0 \sim \pm\sqrt{2(t_0-t)} \quad \mbox{\rm as} \;\; t \to t_0^-. 
\end{equation}
In the case of an odd reflection symmetry, we generically have $u_{xx}(t_0,\xi_0) = 0$ and $u_{xxx}(t_0,\xi_0) \neq 0$. This leads to {\em a pitchfork bifurcation}, for which there are precisely three interfaces $\xi_1(t) < \xi(t) < \xi_2(t)$ in $U$ for $t \in (t_0-\delta,t_0)$ and exactly one interface $\xi(t)$ remains in $U$ for $t \in (t_0,t_0+\delta)$. We also identify the associated scaling laws 
\begin{equation}
\label{law-pitchfork1}
\xi_{1,2}(t) - \xi_0 \sim \pm\sqrt{6(t_0-t)} \quad \mbox{\rm as} \;\; t \to t_0^-
\end{equation}
 and 
\begin{equation}
\label{law-pitchfork2}
 \xi(t) - \xi_0 \sim \alpha(t_0-t) \quad \mbox{\rm as} \;\; t \to t_0
 \end{equation}
for some $\alpha \in \R$. We show that the conditions for a pitchfork bifurcation are satisfied in the classical Burgers' equation with flux function $f(u) = u^2$ for odd shock waves with a single zero on $(0,\infty)$. We note that the above results yield that the lower and upper bounds in the Sturm Theorem~\cite[Theorem~B]{ANG} on the number of interfaces before and after a coalescence event are sharp.

Finally, we corroborate our results with numerical simulations of the modular Burgers' equation. Our numerical approximations rely on a regularization of the modular nonlinearity and employ an elementary finite-difference scheme. These numerical approximations are different from those used in~\cite{PEDR}, where the modular Burgers' equation was solved on a partition of a real line complemented with additional boundary  conditions at the interfaces. We study odd shock and anti-shock waves and observe finite-time coalescence of interfaces through a pitchfork bifurcation. In addition, the numerics confirms the same scaling law 
(\ref{law-pitchfork1}) for the interface extinction. 

The derivation of scaling laws describing the interface dynamics near coalescence has been addressed in other contexts as well and appeared to be challenging. In~\cite{CG72} a linear inhomogeneous heat equation was considered as a simple model for oxygen diffusion. It was suggested 
that the oxygen front (the interface) collapses according to the scaling law $(t_0-t)^{1/2}$. However, a more recent study in~\cite{MV15} based on new numerical algorithms for the time-dependent Stefan problem showed that the scaling law $(t_0-t)^{1/2}$ is not accurate due to an additional singularity as $t \to t_0^-$. Other interface models were studied in~\cite{Needham2,Needham3} by means of matched asymptotic expansions in the context of a KPP equation with a discontinuous cut-off in the reaction function.

We conjecture that the scaling laws~(\ref{law-fold}),~(\ref{law-pitchfork1}), and~(\ref{law-pitchfork2}) proven for smooth flux functions remain true for locally Lipschitz continuous flux functions such as the modular Burgers' equation. However, this question remains open for future research. 

This paper is organized as follows. In Section~\ref{s:wellposedness} we state well-posedness and approximation results for
solutions of the viscous conservation law~(\ref{modBurgers}). Section~\ref{s:finitetime} is devoted to the analysis of finite-time coalescence of interfaces for solutions with initial data of class I, II, and III. In Section~\ref{s:interfacedynamics} we analyze the fold and pitchfork bifurcations describing the interface dynamics near coalescence events and derive associated scaling laws. Section~\ref{sec:numerics} presents numerical simulations illustrating the pitchfork bifurcation for both shock and anti-shock waves in a regularized version of the modular Burgers' equation. Appendix~\ref{app:A} contains the proofs of the well-posedness and approximation results of Section~\ref{s:wellposedness}.

\medskip

{\bf Acknowledgements.} J.~Lin was supported by Stewart Research Scholarship of McMaster University. 
D.~E.~Pelinovsky acknowledges the funding of this study provided by the grant No.~FSWE-2023-0004 
and grant No.~NSH-70.2022.1.5.

\section{Global well-posedness and approximation} 
\label{s:wellposedness}

In this section we establish global well-posedness of uniformly continuous and bounded solutions of the viscous conservation law~\eqref{modBurgers}. We first consider smooth flux functions $f$ before studying the general case of a locally Lipschitz continuous flux function. We show that by locally approximating the flux function $f$ by a smooth function $\tilde{f}$, one can approximate solutions $u(t,\cdot)$ of~\eqref{modBurgers} on any finite time interval by a solution $\tilde u(t,\cdot)$ of the regularized problem  
\begin{align} \label{modBurgersApprox}
\tilde u_t = \tilde u_{xx} + \tilde{f}(\tilde u)_x.
\end{align}
Proofs of all results formulated in this section can be found in Appendix~\ref{app:A}.

For smooth flux functions $f \in C^\infty(\R)$ local existence and uniqueness of classical solutions of~\eqref{modBurgers} follow readily by standard regularity theory for parabolic semilinear equations~\cite{LUN}. The fact that~\eqref{modBurgers} obeys a comparison principle~\cite{PRWE,SERR} then yields global well-posedness. All in all, we establish the following result.

\begin{lemma} \label{p:globalsmooth}
Let $f \in C^\infty(\R)$ and $u_0 \in C_{\mathrm{ub}}^1(\R)$.  Let $M_0 = \sup\{u_0(x) : x \in \R\}$ and $m_0 = \inf\{u_0(x) : x \in \R\}$. There exists a unique smooth global classical solution 
\begin{align*}
u \in C\big([0,\infty),C_{\mathrm{ub}}^1(\R)\big) \cap C\big((0,\infty),C_{\mathrm{ub}}^2(\R)\big) \cap C^1\big((0,\infty),C_{\mathrm{ub}}(\R)\big),
\end{align*}
of~\eqref{modBurgers}  with initial condition $u(0,\cdot) = u_0$ such that $m_0 \leq u(t,x) \leq M_0$ for all $t \geq 0$ and $x \in \R$. Moreover, we have $u \in C^\infty\big((0,\infty) \times \R,\R\!\big)$ with $\partial_t^k u(t,\cdot) \in C_{\mathrm{ub}}^l(\R)$ for $t \geq 0$ and $k,l \in \mathbb{N}_0$.
\end{lemma}

Next, we establish global well-posedness of solutions of~\eqref{modBurgers} for locally Lipschitz continuous flux functions $f$. In this case, classical solutions in the sense of Lemma~\ref{p:globalsmooth} cannot always be expected. For instance, the modular Burgers' equation with flux function $f(u) = |u|$ admits for any $\phi_\pm \in \R$ with $\phi_- < 0 < \phi_+$ a traveling shock-wave solution $u(t,x) = \phi(x-ct)$ converging to asymptotic limits $\phi_\pm$ and propagating with speed 
\begin{align*} 
c = \frac{\phi_+ + \phi_-}{\phi_- - \phi_+}, 
\end{align*}
whose profile 
\begin{align*}
\phi(\pm \xi) = \phi_\pm \left(1-\re^{- (1+c)\xi}\right), \qquad \xi \geq 0,
\end{align*}
does lie in $C_{\mathrm{ub}}^1(\R)$, but not in $C_{\mathrm{ub}}^2(\R)$.
Therefore, we consider \emph{mild} solutions of~\eqref{modBurgers}, which solve the associated integral equation
\begin{align} \label{mildform2}
u(t,\cdot) = \re^{\partial_x^2 t} u_0 + \int_0^t \partial_x \re^{\partial_x^2(t-s)} f(u(s,\cdot)) \de s,
\end{align}
where $u(0,\cdot) = u_0 \in C_{\mathrm{ub}}^1(\R)$ denotes the initial condition. 

Standard analytic semigroup theory in combination with the fact that $f$ is locally Lipschitz continuous yields local existence and uniqueness of solutions of~\eqref{mildform2} in $C_{\mathrm{ub}}(\R)$. We note that it is important here to compose the derivative in~\eqref{mildform2} with the semigroup $\re^{\partial_x^2(t-s)}$, rather than applying it to the flux function $f$, since $f'$ is not necessarily locally Lipschitz continuous. Global well-posedness follows by approximating the solution $u(t,\cdot)$ of~\eqref{mildform2} by the global classical solution $\tilde{u}(t,\cdot)$ of the regularized problem~\eqref{modBurgersApprox}, where $\tilde{f} \in C^\infty(\R)$ is a smooth local approximation of $f$. This leads to the following result.

\begin{lemma} \label{p:globalcub}
Let $f \colon \R \to \R$ be locally Lipschitz continuous and $u_0 \in C_{\mathrm{ub}}^1(\R)$. Let $M_0 = \sup\{u_0(x) : x \in \R\}$ and $m_0 = \inf\{u_0(x) : x \in \R\}$. There exists a unique global solution $u \in C([0,\infty),C_{\mathrm{ub}}(\R))$ of~\eqref{mildform2} such that $m_0 \leq u(t,x) \leq M_0$ for all $t \geq 0$ and $x\in \R$. Moreover, there exist constants $\delta_0,C_0 > 0$ such that for each $\delta \in (0,\delta_0)$ and $\tilde{f} \in C^\infty(\R)$ satisfying
\begin{align*}
 \sup\left\{\big|f(v)-\tilde{f}(v)\big| : v \in \left[m_0,M_0\right]\right\} < \delta,
\end{align*}
the global classical solution
\begin{align} \label{regularity2}
\tilde u \in C\big([0,\infty),C_{\mathrm{ub}}^1(\R)\big) \cap C\big((0,\infty),C_{\mathrm{ub}}^2(\R)\big) \cap C^1\big((0,\infty),C_{\mathrm{ub}}(\R)\big),
\end{align}
of the regularized equation~\eqref{modBurgersApprox} with $\tilde u(0,\cdot) = u_0$, established in Lemma~\ref{p:globalsmooth}, obeys the estimates
\begin{align} \label{approxub}
m_0 \leq \tilde u(t,x) \leq M_0, \qquad \left\|u(t,\cdot) - \tilde u(t,\cdot)\right\|_\infty \leq C_0 \delta \sqrt{t},
\end{align}
for all $t \geq 0$ and $x \in \R$.
\end{lemma} 

Next, we approximate mild solutions of~\eqref{modBurgers} by solutions of the regularized equation~\eqref{modBurgersApprox} in $C_{\mathrm{ub}}^1$-norm rather than in $C_{\mathrm{ub}}$-norm. The approximation in $C_{\mathrm{ub}}^1$-norm will be used in the upcoming analysis to conclude that a single interface of the approximate solution also yields a single interface of the original solution. 

\begin{lemma} \label{p:approxcub1}
Let $f \colon \R \to \R$ be locally Lipschitz continuous. Let $u \in C\big([0,\infty),C_{\mathrm{ub}}^1(\R)\big)$ be a global solution of~\eqref{mildform2} with initial condition $u(0) = u_0 \in C_{\mathrm{ub}}^1(\R)$. Set $M_0 = \sup\{u_0(x) : x \in \R\}$ and $m_0 = \inf\{u_0(x) : x \in \R\}$. Let $R,\tau,\varepsilon > 0$.
There exists $\delta_0 > 0$ such that for each $\delta \in (0,\delta_0)$ and $\tilde f \in C^\infty(\R)$ satisfying
\begin{align*}
 \sup\left\{\big|f(v)-\tilde{f}(v)\big| : v \in \left[m_0,M_0\right]\right\} < \delta, \qquad 
 \sup\left\{\big|\tilde{f}'(v)\big| : v \in \left[m_0,M_0\right]\right\} \leq R,
\end{align*}
the global classical solution~\eqref{regularity2} of the regularized equation~\eqref{modBurgersApprox} with initial condition $\tilde u(0,\cdot) = u_0$, established in Lemma~\ref{p:globalsmooth}, obeys the estimates
\begin{align} \label{approxub2}
m_0 \leq \tilde u(t,x) \leq M_0 \qquad \sup_{0 \leq s \leq \tau}\left\|u(s,\cdot) - \tilde u(s,\cdot)\right\|_{W^{1,\infty}} < \varepsilon,
\end{align}
for $x\in\R$ and $t\geq 0$.
\end{lemma}

We emphasize that Lemma~\ref{p:approxcub1}, in contrast to Lemma~\ref{p:globalcub}, is merely an approximation result and does not imply the existence of a global mild solution in $C_{\mathrm{ub}}^1(\R)$. This suffices for our purposes because we only apply Lemma~\ref{p:approxcub1} to establish finite-time coalescence of interfaces for solutions of~\eqref{modBurgers} with initial data of class I, for which global existence of a mild solution in $C_{\mathrm{ub}}^1(\R)$ follows from a separate well-posedness result, which we will formulate next.

In case of initial data of class I the entropy condition~\eqref{Gelfand} yields the existence of a traveling shock wave with the same limits at $\pm \infty$. We require that the difference between the initial condition and the traveling shock wave is $L^1$-integrable and show that this integrability is maintained over time, which will be important for the mass and energy estimates in the upcoming proofs establishing finite-time coalescence of interfaces in~\S\ref{s:finitetime}. Moreover, by integrating the viscous conservation law~\eqref{modBurgers} we obtain global well-posedness of mild solutions in $C_{\mathrm{ub}}^1(\R)$ rather than in $C_{\mathrm{ub}}(\R)$. 

\begin{lemma} \label{p:globalcub1}
Let $f \colon \R \to \R$ be locally Lipschitz continuous and let $u_0 \in C_{\mathrm{ub}}^1(\R)$. Suppose that there exist $c, C \in \R$ and a solution $\phi \in C_{\mathrm{ub}}^1(\R)$ of the profile equation
\begin{align*}
0 = \phi_\xi + c\phi + f(\phi) + C.
\end{align*}
Suppose $u_0 - \phi$ is $L^1$-integrable. Then, there exists a unique solution $u \in C\big([0,\infty),C_{\mathrm{ub}}^1(\R)\big)$ of~\eqref{mildform2} such that $u(t,\cdot) - \phi$ is $L^1$-integrable for all $t \geq 0$.
\end{lemma}

\section{Finite-time coalescence of interfaces} \label{s:finitetime}

Here we establish finite-time coalescence of interfaces for solutions $u(t,\cdot)$ of~\eqref{modBurgers} with initial data $u(0,\cdot) = u_0 \in C_{\mathrm{ub}}^1(\R)$ of class I or II. We emphasize that solutions with such initial data include all shock waves. On the other hand, anti-shock waves converging to asymptotic limits of opposite signs are not included. We study finite-time coalescence of interfaces of this type of anti-shock waves at the end of this section in the specific setting of the modular Burgers' equation.

\subsection{Solutions with initial data of class I}

Observing that solutions $u(t,x)$ of~\eqref{modBurgers} with initial data $u(0,\cdot) = u_0 \in C_{\mathrm{ub}}^1(\R)$ of type I maintain their asymptotic limits $\phi_\pm$ as $x \to \pm \infty$ for every $t > 0$ by~Lemma~\ref{p:globalcub1}, it readily follows that the solution possesses at least one interface for all $t \geq 0$ since $\phi_+$ and $\phi_-$ have opposite signs. We establish that all interfaces coalesce to a single one within finite time in this case. 

\begin{theorem} \label{t:shockopposite}
Let $f \colon \R \to \R$ be locally Lipschitz continuous and $u_0 \in C_{\mathrm{ub}}^1(\R)$. Suppose $u_0(x)$ converges to asymptotic limits $\phi_\pm$ as $x \to \pm \infty$ such that $\phi_+$ and $\phi_-$ have opposite signs and the Gel'fand-Oleinik entropy condition~\eqref{Gelfand} holds. Moreover, assume that $u_0 - \phi_\pm$ is $L^1$-integrable on $\R_\pm$ and we have $u_0(x) \in [\min\{\phi_-,\phi_+\},\max\{\phi_-,\phi_+\}]$ for all $x \in \R$. Let $u \in C\big([0,\infty),C_{\mathrm{ub}}^1(\R)\big)$ be the global mild solution of~\eqref{modBurgers}, established in Lemma~\ref{p:globalcub1}. Then, there exists a time $T > 0$ such that for all $t > T$ the solution $u(t,\cdot)$ possesses precisely one zero.
\end{theorem}

The proof of Theorem~\ref{t:shockopposite} is based on ideas developed in~\cite{PEDR}, where it is shown that the interfaces of odd shock waves in the modular Burgers' equation coalesce to a single one within finite time. The analysis in~\cite{PEDR} relies on a differential inequality for the mass measured with respect to the fixed interface at $0$. Indeed, due to odd symmetry, $0$ is necessarily an interface of the shock wave for all time and must be the middle interface.

In the general setting considered here, without the presence of an odd symmetry, interfaces are a priori not fixed, which suggests mass functions of the form
\begin{align} \label{e:massfunctions}
\mathcal M_1(t) = \int_{-\infty}^{\xi_2(t)} \left(u(t,x) - \phi_-\right) \de x, \qquad \mathcal M_2(t) = \int_{\xi_2(t)}^\infty \left(\phi_+ - u(t,x)\right) \de x,
\end{align}
where $\xi_2(t)$ is an interface of $u(t,\cdot)$, which now depends on time. As in~\cite{PEDR} we aim to show that the assumption that $\xi_2(t)$ is an interface lying strictly in between two other interfaces $\xi_1(t), \xi_3(t)$ leads to a contradiction with certain inequalities obeyed by the mass functions $\mathcal M_1(t)$ and $\mathcal M_2(t)$. This then yields an explicit time $T > 0$ such that $\xi_1(t) < \xi_2(t) < \xi_3(t)$ cannot hold for $t > T$. 

To derive the desired inequalities for $\mathcal M_1(t)$ and $\mathcal M_2(t)$, a standard strategy is to differentiate with respect to time (using the Leibniz' integral rule) and use the equation~\eqref{modBurgers} to express temporal derivatives of $u(t,x)$. Yet, as mentioned in~\S\ref{s:wellposedness}, it cannot be expected in the case of a locally Lipschitz continuous flux function $f$ that $u(t,x)$ is a classical solution of~\eqref{modBurgers}, which is differentiable with respect to time and twice differentiable with respect to space. In addition, even if the flux function $f$ were smooth, the interface $\xi_2(t)$, being a root of the $C^1$-function $u(t,x)$, is not necessarily differentiable. In fact, the upcoming analysis in~\S\ref{s:interfacedynamics} shows that $\xi_2(t)$ may fail to be differentiable if two interfaces collide.

To address the first challenge we approximate the solution $u(t,x)$ of~\eqref{modBurgers} by a classical solution $\tilde{u}(t,x)$ of the regularized problem~\eqref{modBurgersApprox}, where $\tilde{f}$ is a smooth approximation of $f$ and $\tilde{u}(t,\cdot)$ has the same initial condition as $u(t,\cdot)$. We then aim to show that any three interfaces $\tilde{\xi}_1(t) \leq \tilde{\xi}_2(t) \leq \tilde{\xi}_3(t)$ of $\tilde{u}(t,\cdot)$ coalesce to a single interface within finite time. We address the second challenge by approximating $\tilde{\xi}_2(t)$ on a compact time interval by a sequence of smooth approximations $\tilde{\xi}_{2,n}(t)$. Thus, the mass functions~\eqref{e:massfunctions} with $u(t,x)$ replaced by $\tilde{u}(t,x)$ and $\xi_2(t)$ by $\tilde{\xi}_{2,n}(t)$ are differentiable with respect to $t$ and we can obtain the desired inequalities, which then yield that the interfaces $\tilde{\xi}_1(t),\tilde{\xi}_2(t)$ and $\tilde{\xi}_3(t)$ of $\tilde{u}(t,x)$ coalesce to a single interface before an explicit time $T > 0$, which is independent of the approximation function $\tilde{f}$.

The approximation of the flux $f$ by a smooth function $\tilde{f}$ introduces an additional difficulty. Even with control on the norm $\|u(t,\cdot) - \tilde{u}(t,\cdot)\|_{W^{1,\infty}}$ through Lemma~\ref{p:approxcub1}, the fact that $\tilde{u}(t,\cdot)$ possesses a single interface is not sufficient to conclude that $u(t,\cdot)$ has a single interface because interfaces of $u(t,\cdot)$ might accumulate close to the single interface of $\tilde{u}(t,\cdot)$. We address this issue by bounding the derivative $\partial_x \tilde{u}(t,\cdot)$ at the interface away from $0$, precluding the accumulation of multiple interfaces of $u(t,\cdot)$ close to the single interface of $\tilde{u}(t,\cdot)$. 

We bound the derivative of $\tilde{u}(t,\cdot)$ away from $0$ by considering a traveling shock-wave solution $\tilde u_{\mathrm{tw}}(t,x) = \psi(x - ct)$ of~\eqref{modBurgersApprox}, which propagates at some speed $c \in \R$ and connects asymptotic limits $\psi_{\pm}$ of opposite signs satisfying $|\psi_\pm| < |\phi_\pm|$. Upon switching to a co-moving frame, we may without loss of generality assume that $c = 0$. We then show, with the same methods as before, that all interfaces of the difference $v(t,\cdot) = \tilde{u}(t,\cdot) - \psi$ converge to a single interface within finite time, see Figure~\ref{fig:interfaces}. This then yields the desired lower bound on $\|\partial_x \tilde{u}(t,\cdot)\|_{L^\infty}$. Using that $\|u(t,\cdot) - \tilde{u}(t,\cdot)\|_{W^{1,\infty}}$ can be taken sufficiently small by taking a better approximation $\tilde{f}$ of $f$ if necessary, we thus conclude that the solution $u(t,\cdot)$ must have a single interface for $t > T$, since the same holds for the approximation $\tilde{u}(t,\cdot)$.

\begin{figure}[htb]
	\centering
 \begin{subfigure}{0.49\textwidth}
    \includegraphics[trim = 1 1 1 1, clip, scale=0.6]{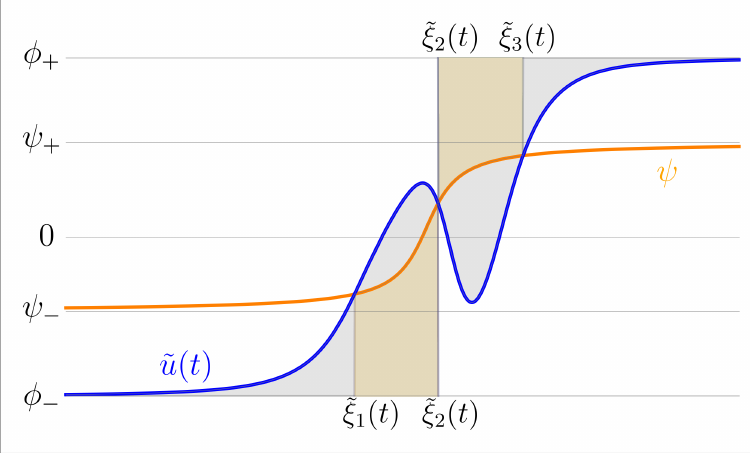}
\end{subfigure}
    \hfill
\begin{subfigure}{0.49\textwidth}
    \includegraphics[scale=0.6]{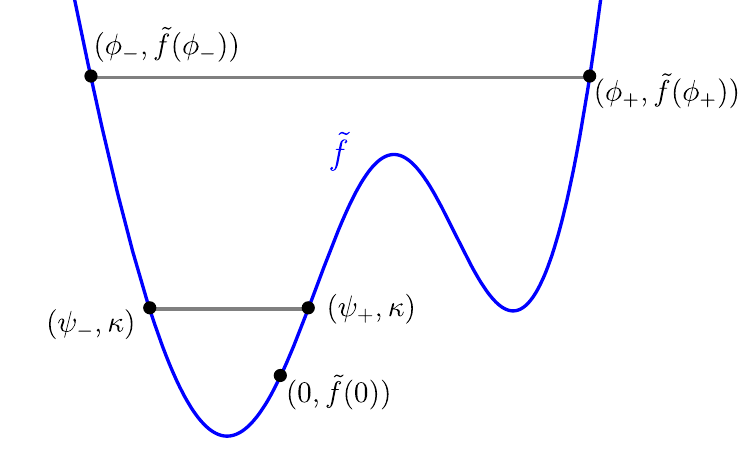}
\end{subfigure}
	\caption{Left: the approximate shock-wave solution $\tilde{u}(t,\cdot)$ of~\eqref{modBurgersApprox} with asymptotic limits $\phi_\pm$, the traveling shock wave $\psi$ with asymptotic limits $\psi_\pm$ and the interfaces $\tilde{\xi}_1(t),\tilde{\xi}_2(t)$ and $\tilde{\xi}_3(t)$ of the difference $v(t,\cdot) = \tilde{u}(t,\cdot) - \psi$. In the proof of Theorem~\ref{t:shockopposite} we bound the shaded areas above and below the graph of $\tilde{u}(t,\cdot)$ from below by the orange subareas. Right: the smooth approximation $\tilde{f}$ of the flux function $f$, established in Lemma~\ref{lem:tech0}. One observes that the regularized problem~\eqref{modBurgersApprox} admits a standing shock-wave solution connecting the asymptotic limits $\phi_\pm$ and one connecting the asymptotic states $\psi_\pm$, where $\phi_- < \psi_- < 0 < \psi_+ < \phi_+$.}
	\label{fig:interfaces}
\end{figure}

Before we proceed with the proof of Theorem~\ref{t:shockopposite}, we first state the following technical lemma, which establishes a suitable smooth approximation $\tilde{f}$ of the flux function $f$ in~\eqref{modBurgers}. Naturally, we require that $\tilde{f}$ lies sufficiently close to $f$ and its derivative is well-behaved. Moreover, we wish that the regularized problem~\eqref{modBurgersApprox} admits a traveling shock-wave solution connecting the asymptotic states $\phi_\pm$, but also a traveling shock wave with asymptotic limits $\psi_\pm$ of opposite signs lying in between $\phi_-$ and $\phi_+$, see also Figure~\ref{fig:interfaces}. Without loss of generality, we can restrict to the case $\phi_- < 0 < \phi_+$ and we may assume $f(\phi_+) = f(\phi_-)$ by replacing $f(u)$ by $f(u)+cu$, where $c$ is given by the Rankine-Hugoniot condition~\eqref{rankine}. 

\begin{lemma} \label{lem:tech0}
Let $f$ be locally Lipschitz continuous and let $\phi_\pm \in \R$ with $\phi_- < 0 < \phi_+$. Suppose that $f(\phi_+) = f(\phi_-)$ and the Gel'fand-Oleinik entropy condition
\begin{align}
f(z) - f(\phi_\pm) < 0, \label{Gelfand2}
\end{align}
holds for all $z \in (\phi_-,\phi_+)$. Then, for each $\kappa \in (f(0),f(\phi_\pm))$, there exists a constant $R > 0$ such that for all  $\delta \in (0,\kappa - f(0))$, there exist $\tilde{f} \in C^\infty(\R)$ and $\psi_\pm \in (\phi_-,\phi_+)$ with $\psi_- < 0 < \psi_+$ such that the following assertions hold:
\begin{itemize}
    \item[i)] For all $z \in (\phi_-,\phi_+)$ we have 
    \begin{align}
\tilde{f}(\phi_+) = \tilde{f}(\phi_-), \qquad \tilde{f}'(\phi_\pm) \neq 0, \qquad \tilde{f}(z) - \tilde{f}(\phi_\pm) < 0. \label{approx1}
\end{align}
\item[ii)] For all $z \in (\psi_-,\psi_+)$ it holds
\begin{align} \label{approx2}
\tilde{f}(\psi_+) = \kappa = \tilde{f}(\psi_-), \qquad \tilde{f}'(\psi_\pm) \neq 0, \qquad \tilde{f}(z) - \tilde{f}(\psi_\pm) < 0.
\end{align}
\item[iii)] For all $z \in [\phi_-,\phi_+]$ we have
\begin{align} \label{approx5}
\big|f(z)-\tilde{f}(z)\big| < \delta, \qquad \big|\tilde{f}'(z)\big| < R.
\end{align}
\end{itemize}
\end{lemma}
\begin{proof}
We first recall that, since $f$ is locally Lipschitz continuous, Rademacher's theorem asserts that $f$ is differentiable almost everywhere and its derivative $f'$ is essentially bounded on each bounded interval. We denote 
\begin{align*}R_1 := \sup\{|f'(u)| : u \in [\phi_-,\phi_+]\}.\end{align*}

Take $\delta \in (0,\kappa - f(0))$. Let $\Phi \colon \R \to \R$ be a mollifier with $\|\Phi\|_1 = 1$, $\Phi(x) > 0$ for $x \in (\phi_-,\phi_+)$ and $\Phi(x) = 0$ for $x \in \R \setminus (\phi_-,\phi_+)$. Set $\Phi_\eta(x) = \Phi(x/\eta)/\eta$ for $\eta > 0$. The function $g \colon \R \to \R$ given by $g(x) = \min\{f(x) + \frac{\delta}{4},f(\phi_-)\}$ is locally Lipschitz continuous. Moreover, it holds $|g'(x)| \leq |f'(x)|$ for each $x \in [\phi_-,\phi_+]$. Since $g$ is continuous, it can be approximated by the sequence $g_\eta := \Phi_\eta \ast g$ of smooth functions. That is, there exists $\eta_0 > 0$ such that
\begin{align*} |g_\eta(u) - g(u)| < \frac{\delta}{4}\end{align*}
for all $u \in [\phi_-,\phi_+]$ and $\eta \in (0,\eta_0)$. By construction we have $g_\eta(x) \leq f(\phi_\pm)$ for all $x \in \R$ and $\eta > 0$. In addition, since $g$ is constant in a neighborhood of $\phi_\pm$ and it holds $g_\eta' = \Phi_\eta \ast g'$, there exists $\eta_1 \in (0,\eta_0)$ such that $g_{\eta}(\phi_\pm) = f(\phi_\pm)$ and $|g_\eta'(u)| \leq \|\Phi_\eta\|_1 R_1 = R_1$ for all $\eta \in (0,\eta_1]$. We conclude that $\tilde{g} = g_{\eta_1} - \delta \Phi/(4\|\Phi\|_\infty)$ is a smooth function which satisfies $\tilde{g}(z) < \tilde{g}(\phi_\pm) = f(\phi_\pm)$ for $z \in (\phi_-,\phi_+)$. Moreover, it holds
\begin{align*} |\tilde{g}(u) - f(u)| \leq |g_{\eta_1}(u) - g(u)| + \frac{\delta}{4} + |g(u) - f(u)| < \frac{3\delta}{4},\end{align*}
and
\begin{align*}
|\tilde{g}'(u)| \leq R_1 + \frac{\delta\|\Phi'\|_\infty}{4\|\Phi\|_\infty},
\end{align*}
for $u \in [\phi_-,\phi_+]$.

Since we have $\tilde{g}(0) < f(0) + \delta < \kappa < f(\phi_\pm) = \tilde{g}(\phi_\pm)$, the open set $\tilde{g}^{-1}[\{z \in \R : z < \kappa\}]$ must contain an interval $(\psi_-,\psi_+)$ with $\phi_- < \psi_- < 0 < \psi_+ < \phi_+$ and $\tilde{g}(\psi_+) = \kappa = \tilde{g}(\psi_-)$. Hence, it holds $\tilde{g}(z) < \tilde{g}(\psi_\pm)$ for $z \in (\psi_-,\psi_+)$. Finally, set $d = \frac12 \min\{\phi_+-\psi_+,\psi_+,-\psi_-,\psi_--\phi_-\} > 0$ and let $\Psi \colon \R \to \R$ be an even, smooth cut-off function such that $\Psi(0) = 1$, $\|\Psi\|_\infty \leq 1$, $\Psi(x) > 0$ for all $x \in (-d,d)$ and $\Psi(x) = 0$ for all $x \in \R \setminus (-d,d)$. Recalling the properties of the function $\tilde{g}$, we conclude that for any $\kappa_\pm, \lambda_\pm \in [0,\delta/(4(\phi_+ - \phi_-))]$, the smooth function 
\begin{align*}
\tilde{f}(x) &= \tilde{g}(x) - \kappa_- (x-\phi_-) \Psi(x-\phi_-) - \lambda_- (x-\psi_-) \Psi(x-\psi_-)  + \lambda_+ (x-\psi_+) \Psi(x-\psi_+)\\ 
&\qquad + \, \kappa_+ (x-\phi_+) \Psi(x-\phi_+), 
\end{align*}
satisfies~\eqref{approx5}, it holds
\begin{align*}
\tilde{f}(\phi_+) = \tilde{f}(\phi_-), \qquad \tilde{f}'(\phi_\pm) = \tilde{g}'(\phi_\pm) \pm \kappa_\pm, \qquad \tilde{f}(z) - \tilde{f}(\phi_\pm) < 0,
\end{align*}
for all $z \in (\phi_-,\phi_+)$, and we have
\begin{align}
\tilde{f}(\psi_+) = \kappa = \tilde{f}(\psi_-), \qquad \tilde{f}'(\psi_\pm) = \tilde{g}'(\psi_\pm) \pm \lambda_\pm, \qquad \tilde{f}(z) - \tilde{f}(\psi_\pm) < 0,
\end{align}
for all $z \in (\psi_-,\psi_+)$. Hence, choosing $\kappa_\pm,\lambda_\pm \in [0,\delta/(4(\phi_+ - \phi_-))]$ in such a way that $\tilde{f}'(\phi_\pm), \tilde{f}'(\psi_\pm) \neq 0$, we find that $\tilde{f}$ satisfies~\eqref{approx1},~\eqref{approx2}, and~\eqref{approx5}.
\end{proof}

Having established a suitable approximation $\tilde{f}$ of the flux function $f$, we now provide the proof of Theorem~\ref{t:shockopposite} following the outline sketched above. 

\begin{proof}[Proof of Theorem~\ref{t:shockopposite}]
We consider the case $\phi_- < 0 < \phi_+$. The case $\phi_+ < 0 < \phi_-$ is handled analogously. Clearly, the zeros (including their multiplicities) of $u(t,\cdot)$ are the same as those of the translate $u(t,\cdot - c t)$ for any $t \geq 0$. Thus, upon replacing $f(u)$ by $f(u) + cu$ in~\eqref{modBurgers}, where $c$ is given by~\eqref{rankine}, we may assume
\begin{align*}
f(\phi_+) = f(\phi_-),
\end{align*}
so that~\eqref{Gelfand} yields
\begin{align*}
f(z) - f(\phi_+) = f(z) - f(\phi_-)  < 0,
\end{align*}
for all $z \in (\phi_-,\phi_+)$. By continuity of $f$ there exists $\eta > 0$ such that for all $z \in [\phi_-,\phi_- + \eta] \cup [\phi_+ - \eta,\phi_+]$ it holds
\begin{align} \label{approx11}
f(z) > \frac{f(\phi_\pm) + f(0)}{2}.
\end{align}
Note that since $f(0) < f(\phi_\pm)$, we must have $\eta < |\phi_\pm|$. Since $u_0$ is continuous and converges to $\phi_\pm \neq 0$ as $x \to \pm \infty$, the function $u_0 - \phi_+ + \eta$ possesses a largest root $\xi_+$ and $u_0 - \phi_- - \eta$ possesses a smallest root $\xi_-$. We set
\begin{align} \label{upperboundtime} T = \frac{\displaystyle \max\left\{\int_{-\infty}^{\xi_+} \left(u_0(x) - \phi_-\right) \de x, \int_{\xi_-}^\infty \left(\phi_+ - u_0(x)\right) \de x\right\}}{f(\phi_+) - f(0)} > 0.\end{align}

We argue by contradiction and assume that there exists $\tau > T$ such that $u(\tau,\cdot)$ has at least two distinct zeros. Then, since $u$ is continuously differentiable, there must exist a zero $x_0$ of $u(\tau,\cdot)$ with $u_x(\tau,x_0) \leq 0$. Fix $\kappa > f(0)$ such that
\begin{align} \label{approx9}
(\kappa - f(0))\tau < \left(f(\phi_-) - f(0)\right)\left(\tau - T\right), \qquad \kappa < \frac{f(0) + f(\phi_\pm)}{2}.
\end{align}
Denote by $L > 0$ the Lipschitz constant of $f$ on $[\phi_-,\phi_+]$ and let $R > 0$ be the constant from Lemma~\ref{lem:tech0} (which depends on $\kappa$). Fix $\varepsilon > 0$ such that 
\begin{align} \label{approx6}
(L+2)\varepsilon < \kappa - f(0), \qquad (R + 1) \varepsilon < \kappa - f(0), \qquad \varepsilon < \min\{M_\eta,m_\eta\}.
\end{align}
Finally, let $\delta_0 > 0$ be the constant from Lemma~\ref{p:approxcub1} (which depends on $R,\tau,\varepsilon > 0$) and take $\delta > 0$ such that 
\begin{align} \label{approx10}
\begin{split}
\delta < \min\{\delta_0,\varepsilon,\kappa - f(0)\}, &\qquad \delta < \frac{f(0) + f(\phi_\pm)}{2} - \kappa, \\ \delta \tau < \left(f(\phi_-) - f(0)\right)&\left(\tau - T\right) -  (\kappa - f(0))\tau,
\end{split}
\end{align}
which is possible by~\eqref{approx9}.

By Lemma~\ref{lem:tech0} there exist $\tilde f \in C^\infty(\R)$ and $\psi_\pm \in (\phi_-,\phi_+)$ with $\psi_- < 0 < \psi_+$ satisfying~\eqref{approx1},~\eqref{approx2}, and~\eqref{approx5}. Lemma~\ref{p:approxcub1} then yields a global classical solution~\eqref{regularity2} of~\eqref{modBurgersApprox} with initial condition $\tilde{u}(0,\cdot) = u_0$ satisfying~\eqref{approxub2}. Then, it must hold
\begin{align} \label{approx8}
\tilde{u}_x(\tau,x_0) \leq \varepsilon, \qquad |\tilde{u}(\tau,x_0)| \leq \varepsilon.
\end{align}
On the other hand, the mean value theorem implies
\begin{align*}
\kappa - \tilde f(0) = \tilde f(\psi_\pm) - \tilde f(0) \leq R |\psi_\pm|.
\end{align*}
Combining the latter with~\eqref{approx5}, and~\eqref{approx6} yields
\begin{align} \label{approx7}
|\psi_\pm| \geq \frac{\kappa - f(0) - \delta}{R} \geq \frac{\kappa - f(0) - \varepsilon}{R} > \varepsilon.
\end{align}
On the other hand~\eqref{approx2},~\eqref{approx11}, and~\eqref{approx9} imply
\begin{align} \label{approx12}
\phi_- + \eta < \psi_- < 0 < \psi_+ < \phi_+ - \eta.
\end{align}

By~\eqref{approx1} and~\eqref{approx2} there exist heteroclinic solutions $\phi(x)$ and $\psi(x)$ of the profile equations
\begin{align} \label{profile2}
0 = \phi_\xi + \tilde f(\phi) - \tilde f(\phi_\pm), \qquad 0 = \psi_\xi + \tilde f(\psi) - \tilde f(\psi_\pm),
\end{align}
respectively, converging exponentially to the asymptotic limits $\phi_\pm$ and $\psi_\pm$, respectively, as $x \to \pm \infty$. Since $u_0 - \phi_\pm$ is $L^1$-integrable on $\R_\pm$, so is $u_0 - \phi$. Therefore, Lemma~\ref{p:globalcub1} yields that $\tilde u(t,\cdot) - \phi$ is $L^1$-integrable for all $t \geq 0$. We conclude that $\tilde u(t,\cdot) - \phi_\pm$ is $L^1$-integrable on $\R_\pm$ for all $t \geq 0$. 

Using~\eqref{approx8} and~\eqref{approx7}, and the fact that $\psi(x)$ is strictly monotone and converges to $\psi_\pm$ as $x \to \pm \infty$, there must exist a translate $x_1 \in \R$ such that the point $(x_0,\tilde{u}(\tau,x_0))$ lies on the graph of $\psi(\cdot - x_1)$. Our aim is to show that the difference $v(t,\cdot) = \tilde u(t,\cdot) - \psi(\cdot - x_1)$ has only a single zero at $t = \tau$, which must lie at $x_0$. This then leads to a contradiction with~\eqref{approx9},~\eqref{approx10},~\eqref{approx6}, and~\eqref{approx8} by our choice of constants $\kappa,\varepsilon$ and $\delta$. 

Upon replacing the traveling shock wave $\psi$ by its translate $\psi(\cdot-x_1)$, we may without loss of generality assume $x_1 = 0$. We observe that
\begin{align} \label{regularity3}
v \in C\big([0,\infty),C_{\mathrm{ub}}^1(\R)\big) \cap C\big((0,\infty),C_{\mathrm{ub}}^2(\R)\big) \cap C^1\big((0,\infty),C_{\mathrm{ub}}(\R)\big),
\end{align}
is a global classical solution of the equation
\begin{align} \label{modBurgersDiff}
v_t = v_{xx} + \tilde f(v + \psi)_x - \tilde f(\psi)_x.
\end{align}
We can apply the Sturm theorem,~\cite[Theorem~B]{ANG}, upon recasting~\eqref{modBurgersDiff} as the linear parabolic equation
\begin{align} \label{linear}
 v_t = v_{xx} + b(t,x) v_x + a(t,x) v,
\end{align}
with 
\begin{align*}
b(t,x) = \tilde f'(v(t,x) + \psi(t,x)), \qquad a(t,x) = \psi_x(t,x) \frac{\tilde f'(v(t,x) + \psi(x)) - \tilde f'(\psi(x))}{v(t,x)},\end{align*} 
where we note that $a$, $b, b_x$, and $b_t$ are bounded on the strip $\R \times [0,s]$ for any $s > 0$ by~\eqref{regularity3}, and the fact that $\tilde f$ and $\psi$ and smooth. Applying~\cite[Theorem~B]{ANG} to~\eqref{linear} yields that, if it holds $v(t_0,\xi_0) = 0 = v_x(t_0,\xi_0)$ at some $(t_0,\xi_0) \in (0,\infty) \times \R$, then there exist $\theta \in (0,t_0)$ and a neighborhood $U \subset \R$ of $\xi_0$ such that for $t \in (t_0-\theta,t_0)$, there are at least two zeros of $v(t,\cdot)$ in $U$ and for $t \in (t_0,t_0+\theta)$, there is at most one zero of $v(t,\cdot)$ in $U$. Noting that $v(t,x)$ is continuously differentiable with respect to $x$ and $t$, this leads to two important observations. First, no new zeros of $v(t,\cdot)$ can form dynamically over time. Second, multiple roots are isolated in $\R \times (0,\infty)$. 

Now assume by contradiction that for all $t \in [0,\tau]$, there exist at least two zeros of $v(t,\cdot)$. A consequence of the above two observations, the regularity of $v(t,\cdot)$, and the fact that $v(t,\cdot)$ converges to $\phi_\pm - \psi_\pm$ at $\pm \infty$ with $\phi_- - \psi_- < 0 < \phi_+ - \psi_+$, is that there must be three functions $\tilde{\xi}_{1,2,3} \colon [0,T] \to \R$ which depend continuously on time such that it holds $\tilde{\xi}_1(t) < \tilde{\xi}_2(t) < \tilde{\xi}_3(t)$, $v(t,\tilde{\xi}_i(t)) = 0$ for $i = 1,2,3$, $v(t,x) > 0$ for all $x \in (\tilde{\xi}_1(t),\tilde{\xi}_2(t))$, $v(t,x) < 0$ for all $x \in (\tilde{\xi}_2(t),\tilde{\xi}_3(t))$, and $v_x(t,\tilde{\xi}_2(t)) \leq 0$ for all $t \in [0,T]$. We note that by~\eqref{approx12}, it must hold $\xi_- < \tilde{\xi}_2(0) < \xi_+$. 

Take a sequence $\{\tilde{\xi}_{2,n}\}_n$ of smooth functions converging uniformly in $C([0,\tau])$ to $\tilde{\xi}_2$ as $n \to \infty$.  Define the masses
\begin{align*}
M_{1,n}(t) &= \int_{-\infty}^{\tilde{\xi}_{2,n}(t)} (v(t,x) - \phi_- + \psi_-) \de x\\ 
&= \int_{-\infty}^{\tilde{\xi}_{2,n}(t)} (\tilde u(t,x) - \phi_-) \de x - \int_{-\infty}^{\tilde{\xi}_{2,n}(t)} (\psi(x) - \psi_-) \de x, \\ 
M_{1}(t) &= \int_{-\infty}^{\tilde{\xi}_2(t)} (v(t,x) - \phi_- + \psi_-) \de x = \int_{-\infty}^{\tilde{\xi}_2(t)} (\tilde u(t,x) - \phi_-) \de x - \int_{-\infty}^{\tilde{\xi}_2(t)} (\psi(x) - \psi_-) \de x,
\end{align*}
which are well-defined as $\tilde u(t,\cdot) - \phi_-$ and $\psi - \psi_-$ are $L^1$-integrable on $\R_-$ for all $t \geq 0$. 
Applying the Leibniz' rule, we find
\begin{align*}
M_{1,n}'(s) &= \tilde{\xi}_{2,n}'(s)\left(v(s,\tilde{\xi}_{2,n}(s)) - \phi_- + \psi_-\right)\\ 
&\qquad + \, \int_{-\infty}^{\tilde{\xi}_{2,n}(s)} \left(v_{xx}(s,x) + \partial_x\left(\tilde f(v(s,x) + \psi(x)) - \tilde f(\psi(x))\right)\right) \de x\\ 
&= \partial_s\left(\tilde{\xi}_{2,n}(s)\left(v(s,\tilde{\xi}_{2,n}(s)) - \phi_- + \psi_-\right)\right) - \tilde{\xi}_{2,n}(s)\partial_s \left(v(s,\tilde{\xi}_{2,n}(s))\right) + v_x(s,\tilde{\xi}_{2,n}(s)) \\
&\qquad + \, \tilde f\big(v(s,\tilde{\xi}_{2,n}(s)) + \psi(\tilde{\xi}_{2,n}(s))\big) - \tilde f\big(\psi(\tilde{\xi}_{2,n}(s))\big) - \tilde f(\phi_-) + \tilde f(\psi_-),
\end{align*}
for $s \in (0,\tau]$. Integrating the latter from $0$ to $t$ we obtain
\begin{align*}
M_{1,n}(t) &= M_{1,n}(0) + \tilde{\xi}_{2,n}(t)\left(v(t,\tilde{\xi}_{2,n}(t)) - \phi_- + \psi_-\right) - \tilde{\xi}_{2,n}(0)\left(v(0,\tilde{\xi}_{2,n}(0)) - \phi_- + \psi_-\right)\\ 
&\qquad + \, \int_0^t \left(\tilde{\xi}_{2,n}(s)\partial_s \left[ v(s,\tilde{\xi}_{2,n}(s)) \right] + v_x(s,\tilde{\xi}_{2,n}(s)) \right) \de s + \left(\tilde f(\psi_-) - \tilde f(\phi_-)\right) t\\
&\qquad + \, \int_0^t \left(\tilde f\big(v(s,\tilde{\xi}_{2,n}(s)) + \psi(\tilde{\xi}_{2,n}(s))\big) - \tilde f\big(\psi(\tilde{\xi}_{2,n}(s))\big)\right)\de s,
\end{align*}
for $t \in (0,\tau]$. Taking the limit $n \to \infty$, while recalling the regularity~\eqref{regularity3} of $v(t,\cdot)$ and the fact that $v_x(\tilde{\xi}_2(s),s) \leq 0$ for all $s \in [0,T]$, we arrive at
\begin{align*}
M_1(t) &= M_{1}(0) - (\tilde{\xi}_2(t) - \tilde{\xi}_2(0))\left(\phi_- - \psi_-\right) + \int_0^t v_x(s,\tilde{\xi}_2(s)) \de s + \left(\tilde f(\psi_-) - \tilde f(\phi_-)\right) t\\
&\leq M_{1}(0) + (\tilde{\xi}_2(t) - \tilde{\xi}_2(0))(\psi_- - \phi_-) + \left(\tilde f(\psi_-) - \tilde f(\phi_-)\right) t,
\end{align*}
implying
\begin{align*}
\int_{-\infty}^{\tilde{\xi}_2(t)} (\tilde u(t,x) - \phi_-) \de x &\leq \int_{-\infty}^{\tilde{\xi}_2(0)} (u_0(x) - \phi_-) \de x + \int_{\tilde{\xi}_2(0)}^{\tilde{\xi}_2(t)} \left(\psi(x) - \phi_-\right) \de x\\ 
&\qquad +\, \left(\tilde f(\psi_-) - \tilde f(\phi_-)\right) t,
\end{align*}
for $t \in [0,\tau]$. On the other hand, since $\tilde{u}(t,\cdot) - \phi_-$ is nonnegative for all $t \geq 0$ by~\eqref{approxub2}, it holds
\begin{align} \label{lowerboundmass}
\int_{\tilde{\xi}_1(t)}^{\tilde{\xi}_2(t)} \left(\psi(x) - \phi_-\right) \de x \leq \int_{-\infty}^{\tilde{\xi}_2(t)} (\tilde u(t,x) - \phi_-) \de x,
\end{align}
cf.~Figure~\ref{fig:interfaces}. Combining the latter two inequalities, while using $\tilde{\xi}_2(0) < \xi_+$, we obtain
\begin{align*}
\int_{\tilde{\xi}_1(t)}^{\tilde{\xi}_2(0)} \left(\psi(x) - \phi_-\right) \de x &\leq \int_{-\infty}^{\tilde{\xi}_2(0)} (u_0(x) - \phi_-) \de x + \left(\tilde f(\psi_-) - \tilde f(\phi_-)\right) t\\
&\leq \int_{-\infty}^{\xi_+} (u_0(x) - \phi_-) \de x + \left(f(0) - f(\phi_-)\right) t + \left(\kappa - f(0)\right) t + \delta t.
\end{align*}
Inserting $t = \tau$ in the latter, applying~\eqref{approx10}, and recalling~\eqref{upperboundtime}, we arrive at
\begin{align*}
\int_{\tilde{\xi}_1(t)}^{\tilde{\xi}_2(0)} \left(\psi(x) - \phi_-\right) \de x &\leq \left(f(0) - f(\phi_-)\right)(\tau - T) + \left(\kappa - f(0)\right) \tau + \delta \tau < 0.
\end{align*}
yielding $\tilde{\xi}_2(0) \leq \tilde{\xi}_1(\tau) < \tilde{\xi}_2(\tau)$, since we have $\psi(x) - \phi_- \geq \psi_- - \phi_- > 0$ for all $x \in \R$. 

Similarly, we establish
\begin{align*}
\int_{\tilde{\xi}_2(0)}^{\tilde{\xi}_3(t)} \left(\phi_+ - \psi(x)\right) \de x &\leq \int_{\tilde{\xi}_2(0)}^\infty (\phi_+ - u_0(x)) \de x + \left(\tilde f(\psi_+) - \tilde f(\phi_+)\right) t,
\end{align*}
yielding
\begin{align*}
\int_{\tilde{\xi}_2(0)}^{\tilde{\xi}_3(t)} \left(\phi_+ - \psi(x)\right) \de x \leq \left(f(0) - f(\phi_+)\right)(\tau - T) + \left(\kappa - f(0)\right) \tau + \delta \tau < 0,
\end{align*}
and thus, $\tilde{\xi}_2(\tau) < \tilde{\xi}_3(\tau) \leq \tilde{\xi}_2(0)$, which contradicts $\tilde{\xi}_2(0) < \tilde{\xi}_2(\tau)$. Hence, there must exist a $t \in [0,\tau]$ such that $v(t,\cdot)$ has only a single zero. Recalling that the number of zeros is non-increasing, we conclude that $v(\tau)$ has a single zero, which must be $x_0$. Since $v(t,\cdot)$ converges to $\phi_\pm - \psi_\pm$ as $x \to \pm \infty$ and we have $\phi_- < \psi_- < 0 < \psi_+ < \phi_+$, it must hold $\tilde{u}_x(\tau,x_0) - \psi'(x_0) = v_x(\tau,x_0) \geq 0$. On the other hand, using that $\psi$ solves~\eqref{profile2} and $0 = v(\tau,x_0) = \tilde{u}(\tau,x_0) - \psi(x_0)$, while recalling~\eqref{approx5} and~\eqref{approx8}, we infer
\begin{align*}
\varepsilon \geq \tilde{u}_x(\tau,x_0) &\geq \psi'(x_0) = -\tilde f(\psi(x_0)) + \tilde f(\psi_\pm) = -\tilde f(\tilde u(\tau,x_0)) + \kappa\\
&= -\tilde f(\tilde u(\tau,x_0)) + f(\tilde u(\tau,x_0)) - f(\tilde u(\tau,x_0)) + f(0) + \kappa - f(0)\\ 
&\geq -\delta - L\varepsilon + \kappa - f(0).
\end{align*}
Combining the latter with~\eqref{approx10} yields
\begin{align*}
(L+2)\varepsilon \geq \kappa - f(0),
\end{align*}
which contradicts~\eqref{approx6}. We conclude that for each $t > T$, the function $u(t,\cdot)$ possesses at most one zero.
\end{proof}

\begin{remark}
We note that the proof of Theorem~\ref{t:shockopposite} provides an explicit upper bound $T$, given by~\eqref{upperboundtime}, on the time at which all interfaces of the solution $u(t,\cdot)$ of~\eqref{modBurgers} have collapsed to a single interface. The upper bound~\eqref{upperboundtime} only depends on the flux function $f$ and the initial condition $u_0$.     
\end{remark}

\begin{remark}
We expect that it might be possible to lift the assumption that $u_0(x) \in [\min\{\phi_-,\phi_-\},\max\{\phi_-,\phi_+\}]$ for all $x \in \R$ in Theorem~\ref{t:shockopposite} by bounding $u_0(x)$ from below by a smooth function $u_-(x)$ and from above by a smooth function $u_+(x)$ satisfying 
$$
\lim_{x \to \pm \infty} u_-(x) = \min\{\phi_-,\phi_+\} \quad \mbox{\rm and} \quad \lim_{x \to \pm \infty} u_+(x) = \max\{\phi_-,\phi_+\}.
$$ 
It has been established in~\cite{FRSE} that the solutions $\tilde{u}_\pm(t,\cdot)$ of the regularized problem~\eqref{modBurgersApprox} with initial conditions $\tilde u_\pm(0,\cdot) = u_\pm$ converge in $L^1$- and $L^\infty$-norm to their asymptotic limits as $t \to \infty$. So, by the comparison principle, the area of $\tilde{u}(t,\cdot)$ under $\min\{\phi_-,\phi_+\}$ or above $\max\{\phi_-,\phi_+\}$ converges to $0$ as $t \to \infty$. We expect that using similar techniques as in the proof of Theorem~\ref{t:shocksame}, one can obtain decay estimates on this area, which are independent of the approximation $\tilde{f}$ of the flux function $f$. One would then hope to find an explicit time $T_1 > 0$, only depending on $f$ and the initial condition $u_0$, such that for $t > T_1$ this area is so small that the estimate~\eqref{lowerboundmass} is still valid and one can proceed as in the proof of Theorem~\ref{t:shockopposite}. We decided to refrain from providing this exposition, since it merely introduces additional technicalities obscuring the main ideas of the proof.
\end{remark}

\subsection{Solutions with initial data of class II} 

We prove the finite-time extinction of all interfaces of solutions with initial data of class II. That is, we consider a solution $u(t,x)$ of~\eqref{modBurgers} with initial condition $u(0,x) = u_0(x)$, which converges to nonzero asymptotic limits $\phi_\pm$ as $x \to \pm \infty$ that have the same sign. By approximating the solution $u(t,x)$ by a solution $\tilde{u}(t,x)$ to the regularized problem~\eqref{modBurgersApprox} with smooth flux function $\tilde{f}$ and bounding the initial condition $u_0$ from below or above, it suffices by the comparison principle of~\cite{PRWE,SERR} to prove the statement for a solution $\tilde{v}(t,\cdot)$ of the regularized problem~\eqref{modBurgersApprox} which possesses the same non-zero asymptotic limit $\phi_0$ at $\pm \infty$, see Figure~\ref{fig:interfaces2}. We show that all interfaces of $\tilde{v}(t,\cdot)$ go extinct within finite time by deriving an energy inequality for the difference $\tilde{v}(t,\cdot) - \phi_0$. The energy estimate relies on the Gagliardo-Nirenberg inequality and the conservation of mass.

\begin{theorem} \label{t:shocksame}
Let $f \colon \R \to \R$ be locally Lipschitz continuous and $u_0 \in C_{\mathrm{ub}}^1(\R)$. Suppose that $u_0(x)$ converges to nonzero asymptotic limits $\phi_\pm$ as $x \to \pm \infty$ such that $\phi_+$ and $\phi_-$ have the same sign. Let $u \in C\big([0,\infty),C_{\mathrm{ub}}(\R)\big)$ be the global mild solution of~\eqref{modBurgers}, established in Lemma~\ref{p:globalcub}. Then, there exists a time $T > 0$ such that for all $t > T$, the solution $u(t,\cdot)$ possesses no zeros.
\end{theorem}
\begin{proof}
Throughout the proof, $C>0$ denotes the constant appearing in the Gagliardo-Nirenberg interpolation inequality
\begin{align} \label{gagliardo}
 \|g\|_\infty \leq C\|g'\|_2^{\frac{2}{3}}\|g\|_1^{\frac{1}{3}},
\end{align}
which holds for all $g \in L^1(\R) \cap H^1(\R)$.

We consider the case $0 < \phi_- \leq \phi_+$. The cases $0 < \phi_+ \leq \phi_-$, $\phi_- \leq \phi_+ < 0$ and $\phi_+ \leq \phi_- < 0$ are handled analogously. Take any $v_0 \in C_{\mathrm{ub}}^1(\R)$ such that $v_0 - \tfrac{2}{3}\phi_-$ is $L^1$-integrable, not identically zero, and nonpositive and it holds $v_0(x) \leq u_0(x)$ for all $x \in \R$.  Set
\begin{align} \label{upperboundtime2}
T = \frac{2\phi_-^3}{9C^3\left\|v_0 - \tfrac{2}{3}\phi_-\right\|_2^2 \left\|v_0 - \tfrac{2}{3}\phi_-\right\|_1} > 0.
\end{align}

Let $\tau > T$. By Lemma~\ref{p:globalcub} there exists $\tilde{f} \in \C^\infty(\R)$ such that the global classical solution $\tilde{u}(t,\cdot)$ of the regularized problem~\eqref{modBurgersApprox} with initial condition $\tilde{u}(0,\cdot) = u_0$ satisfies~\eqref{regularity2} and
\begin{align}
\|u(\tau,\cdot) - \tilde{u}(\tau,\cdot)\|_\infty < \tfrac{1}{3}\phi_-. \label{approx00}
\end{align}
Let 
\begin{align*}
\tilde{v} \in C\big([0,\infty),C_{\mathrm{ub}}^1(\R)\big) \cap C\big((0,\infty),C_{\mathrm{ub}}^2(\R)\big) \cap C^1\big((0,\infty),C_{\mathrm{ub}}(\R)\big),
\end{align*}
be the solution of~\eqref{modBurgersApprox} with initial condition $\tilde{v}(0,\cdot) = v_0$, cf.~Lemma~\ref{p:globalsmooth}. By the comparison principle, cf.~\cite{PRWE,SERR}, it holds 
\begin{align} \label{comparison} 
\tilde{v}(t,x) \leq \tilde{u}(t,x), \qquad \tilde{v}(t,x) \leq \tfrac{2}{3}\phi_-,
\end{align}
for all $t \geq 0$ and $x \in \R$. Our aim is to show that we have $\tilde{v}(\tau,x) \geq \tfrac{1}{3}\phi_-$ for all $x \in \R$, which together with~\eqref{approx00} and~\eqref{comparison} yields the desired result that $u(\tau,\cdot)$ does not posses any zeros, cf.~Figure~\ref{fig:interfaces2}.

\begin{figure}[htb]
	\centering
 \begin{subfigure}{0.49\textwidth}
    \includegraphics[trim = 1 1 1 1, clip, scale=0.6]{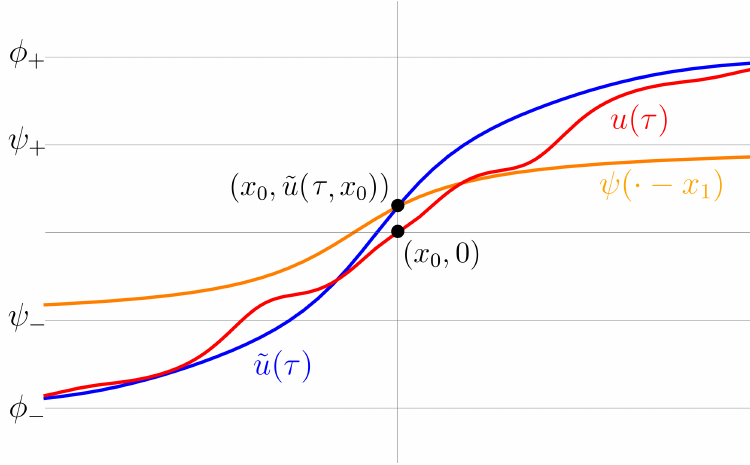}
\end{subfigure}
    \hfill
\begin{subfigure}{0.49\textwidth}
    \includegraphics[trim = 1 1 1 1, clip, scale=0.6]{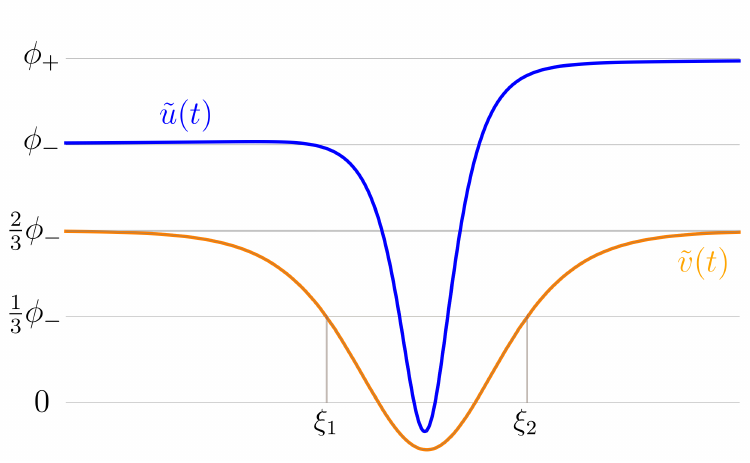}
\end{subfigure}
	\caption{Left: the shock wave $u(t,\cdot)$ and its approximation $\tilde{u}(t,\cdot)$ from the proof of Theorem~\ref{t:shockopposite} at time $t = \tau$. The shock wave $u(\tau,\cdot)$ possesses the asymptotic limits $\phi_\pm$ at $\pm \infty$ and has an interface at $x_0$. The translate $\psi(\cdot - x_1)$ of the traveling shock wave $\psi$, connecting the asymptotic states $\psi_\pm$, passes through the point $(x_0,\tilde{u}(\tau,x_0))$. Right: the approximate solution $\tilde{u}(t,\cdot)$ connecting the asymptotic end states $\phi_\pm$ and its subsolution $\tilde{v}(t,\cdot)$ possessing the asymptotic limit $\tfrac23 \phi_-$ at $\pm \infty$. In the proof of Theorem~\ref{t:shocksame} we approximate the energy of $\tilde{v}(t,\cdot) - \tfrac23 \phi_-$ at a point $t = \tau$ from below by $(\xi_2 - \xi_1) \tfrac{1}{9} \phi_-^2$.}
	\label{fig:interfaces2}
\end{figure}

We argue by contradiction and assume that there exist $\xi_1, \xi_2 \in \R$ with $\xi_1 < \xi_2$ such that $\tilde{v}(\tau,\xi_1) = \tfrac{1}{3}\phi_- = \tilde{v}(\tau,\xi_2)$. First, we observe that $(\tau,\xi_1)$ is a root of $z(t,x) = \tilde{v}(t,x) - \tfrac{1}{3}\phi_-$, which satisfies the linear equation
\begin{align} \label{e:linearized}
z_t = z_{xx} + b(t,x) z_x,
\end{align}
where the spatial and temporal derivative of $b(t,x) = \tilde f'(\tilde{v}(t,x))$ are bounded on the strip $\R \times [0,s]$ for any $s > 0$. Applying the Sturm Theorem~\cite[Theorem~B]{ANG} to~\eqref{e:linearized} yields that $z(t,\cdot)$ must have a zero for all $t \in [0,\tau]$. That is, it holds
\begin{align} \label{linftybound}
\left\|\tilde{v}(t,\cdot) - \tfrac23 \phi_- \right\|_\infty \geq \tfrac{1}{3} \phi_-,
\end{align}
for all $t \in [0,\tau]$.

Next, we observe that the mass
\begin{align*} M(t) = \int_{\R} \left(\tilde{v}(t,x) - \tfrac{2}{3}\phi_-\right) \de x, \qquad t \geq 0,\end{align*}
is conserved. Indeed, it holds
\begin{align*}
M'(t) = \int_{\R} \left(\tilde{v}_{xx}(t,x) + \partial_x \big(\tilde{f}(\tilde{v}(t,x))\big)\right) \de x 
= 0,
\end{align*}
and thus, we have $M(t) = M(0)$ for all $t \geq 0$. Second, we establish an estimate for the energy 
\begin{align*}E(t) = \left\|\tilde{v}(t,\cdot) - \tfrac{2}{3}\phi_-\right\|_2^2.\end{align*} We compute using integration by parts
\begin{align*}
E'(t) &= 2 \int_{\R} \left(\tilde{v}(t,x) - \tfrac{2}{3}\phi_-\right)\left(\tilde{v}_{xx}(t,x) + \partial_x \big(\tilde{f}(\tilde{v}(t,x))\big)\right) \de x \\
&= -2 \int_{\R} \tilde{v}_x(t,x)\left(\tilde{v}_{x}(t,x) + \tilde{f}(\tilde{v}(t,x))\right) \de x \\
&= -2\|\tilde{v}_x(t,\cdot)\|_2^2,
\end{align*}
for $t \geq 0$. Therefore, using the Gagliardo-Nirenberg inequality~\eqref{gagliardo}, the bound~\eqref{comparison} and the fact that the (nonzero) mass $M(t)$ is conserved, we obtain the energy estimate
\begin{align*}
E'(t) \leq -\frac{2}{C^3 |M(t)|} \left\|\tilde{v}(t,\cdot) - \tfrac{2}{3}\phi_-\right\|_\infty^3 = -\frac{2}{C^3 |M(0)|} \left\|\tilde{v}(t,\cdot) - \tfrac{2}{3}\phi_-\right\|_\infty^3,
\end{align*}
for $t \geq 0$. Integrating the latter from $0$ to $\tau$, while using~\eqref{linftybound} and $\tau > T$, we obtain
\begin{align*} (\xi_2 - \xi_1)\frac{\phi_-^2}{9} &\leq E(\tau) \leq E(0)-\frac{2}{C^3 |M(0)|} \int_0^\tau \left\|\tilde{v}(t,\cdot) - \tfrac{2}{3}\phi_-\right\|_\infty^3 \de t\\ 
&\leq E(0)-\frac{2 \phi_-^3}{9C^3  |M(0)|} \tau = E(0)\left(1-\frac{\tau}{T}\right) < 0,
\end{align*}
which contradicts $\xi_1 < \xi_2$, see Figure~\ref{fig:interfaces2}. Therefore, $\tilde{v}(\tau,\cdot) - \tfrac13 \phi_-$ can possess at most one single zero, which together with estimates~\eqref{approx00} and~\eqref{comparison} and the fact that $\tilde{v}(\tau,x)$ converges to $\tfrac23 \phi_-$ as $x \to \pm \infty$, implies that $u(\tau,\cdot)$ cannot have any zeros. \end{proof}

\begin{remark} \label{explicittime}
Assume that the initial condition $u_0$ in Theorem~\ref{t:shocksame} possesses an interface and it holds $0 < \phi_- \leq \phi_+$. By mollifying the compactly supported, nonpositive, nonzero function $u_1(x) = \min\{\tfrac23 \phi_-,u_0(x)\} - \tfrac23 \phi_-$, one readily finds a sequence $\{z_n\}_n$ of nonpositive, nonzero, smooth, and compactly supported functions such that $z_n$ converges in $L^p(\R)$ to $u_1$ as $n \to \infty$ for $p = 1,2$. Thus, $w_n = z_n + \tfrac23 \phi_-$ is a smooth function such that $w_n - \tfrac23 \phi_-$ is $L^1$-integrable, not identically zero, and nonpositive such that $w_n(x) \leq u_0(x)$ for all $n \in \mathbb N$. Hence, $w_n$ satisfies the criteria for the function $v_0$ in the proof of Theorem~\ref{t:shocksame} for any $n \in \mathbb N$. That is, we find that the upper bound~\eqref{upperboundtime2} on the time at which all interfaces of the solution $u(t,\cdot)$ have gone extinct, could be taken equal to
\begin{align*}
T = \frac{2\phi_-^3}{9C^3 \|u_1\|_2^2 \|u_1\|_1}.
\end{align*}
We stress that $T$ only depends on the initial condition $u_0$ of the solution $u(t,\cdot)$ and the positive constant $C$ from the Gagliardo-Nirenberg inequality~\eqref{gagliardo}.
\end{remark}

\subsection{Solutions with initial data of class III} \label{s:antishocks}

In Theorem~\ref{t:shockopposite}, we proved finite-time coalescence of interfaces for shock waves converging to asymptotic limits of opposite signs. This prompts the question of whether anti-shock waves converging to asymptotic limits of opposite signs also exhibit finite-time coalescence of interfaces. One readily observes that the proof of Theorem~\ref{t:shockopposite} strongly relies on the Gel'fand-Oleinik entropy inequality~\eqref{Gelfand} to bound the mass. It cannot be expected that the same strategy applies to the case of anti-shock waves  that violate~\eqref{Gelfand}. Therefore, the question of whether finite-time coalescence of interfaces can be established for solutions with initial data of class III remains open. 

Nevertheless, we can study the interface dynamics of solutions with initial data of class III in the framework of the modular Burgers' equation
\begin{equation}
\label{Burgers-modular-sec3}
u_t = u_{xx} + \lvert u \rvert_x,
\end{equation}
which corresponds to the scalar viscous conservation law~\eqref{modBurgers} with the modular flux function $f(u) = |u|$. Our upcoming analysis establishes finite-time coalescence of interfaces for anti-shock waves converging to asymptotic limits $\mp \phi_*$ as $x \to \pm \infty$ with $\phi_* > 0$. We make the following assumption on the regularity of solutions 
to the modular Burgers' equation~(\ref{Burgers-modular-sec3}).

\begin{assumption}
	\label{assumption-modular}
For every $u_0 \in C_{\mathrm{ub}}^1(\R)$ converging to nonzero asymptotic limits $u_\pm$ at $\pm \infty$, the global mild solution $u \in C([0,\infty),C_{\mathrm{ub}}(\R))$ of~\eqref{Burgers-modular-sec3}, established in Lemma~\ref{p:globalcub}, with initial condition $u(0,\cdot) = u_0$ satisfies $u \in C^1((0,\infty) \times \R,\R)$ 
such that $u(t,\cdot)$, $t \geq 0$ is piecewise $C^2$ with the finite jump condition 
\begin{equation}
\label{interface-con}
u_{xx}(t,\xi(t)^+) - u_{xx}(t,\xi(t)^-) = -2 |u_x(t,\xi(t))|,
\end{equation}
across any interface $x = \xi(t) \in \mathbb{R}$.
\end{assumption}

Assumption~\ref{assumption-modular} was proven in~\cite{PELI} for the class of solutions to~(\ref{Burgers-modular-sec3}) with a single interface in a local neighborhood of a traveling shock wave. In a more general setting, the validity of Assumption~\ref{assumption-modular} is an open question. 


We expect that Assumption~\ref{assumption-modular} can be proven in a general case by using approximation by solutions of the regularized equation as in Theorems~\ref{t:shockopposite} and~\ref{t:shocksame}. However, since our main goal is to illustrate the finite-time coalescence of interfaces of solutions of~\eqref{modBurgers} with initial data of class III rather than proving a general well-posedness result for piecewise smooth flux functions, we refrain from doing so. 

The following lemma establishes that the odd parity of initial data is preserved in the time evolution of the modular Burgers' equation~(\ref{Burgers-modular-sec3}).

\begin{lemma}
	\label{lemma-initial}
Let $u_0 \in C^1_{\rm ub}(\mathbb{R})$ satisfy $u_0(-x) = -u_0(x)$ for every $x \in \mathbb{R}$. Then, the mild solution $u \in C([0,\infty),C_{\rm ub}(\mathbb{R}))$ of~(\ref{Burgers-modular-sec3}), established in Lemma~\ref{p:globalcub}, satisfies $u(t,-x) = -u(t,x)$ for every $t \geq 0$ and $x \in \mathbb{R}$.
\end{lemma}

\begin{proof}
First, observe that, if $z \in C_{\mathrm{ub}}(\R)$ is odd, then
\begin{align*}
\re^{\partial_x^2 t} z = \int_{\R} \frac{\re^{-\frac{y^2}{4t}}}{\sqrt{4\pi t}} z(x-y) \de y
\end{align*}
is also odd, which follows by the substitution $y \mapsto -y$. Now the mild solution
$u(t,\cdot)$ of~\eqref{Burgers-modular-sec3} is given by
\begin{align*}
u(t,\cdot) = \re^{\partial_x^2 t} u_0 + \partial_x \int_0^t  \re^{\partial_x^2(t-s)} \left|u(s,\cdot)\right| \de s,
\end{align*}
for $t \geq 0$. Since $u_0$ is odd, so is $\re^{\partial_x^2 t} u_0$. Hence, using again the substitution $y \mapsto y$, we obtain 
\begin{align*}
u(t,\cdot) + u(t,-\,\cdot) &= \partial_x \int_0^t \re^{\partial_x^2(t-s)} \left(\left|u(s,\cdot)\right| - \left|u(s,-\,\cdot)\right|\right) \de s,
\end{align*}
for $t \geq 0$. Taking norms in the latter and recalling the well-known fact that there exists a constant $C > 0$ such that
$$
\big\| \partial_x e^{\partial_x^2 t} v \big\|_{L^\infty} \leq C t^{-1/2} \| v \|_{L^\infty}
$$
for $t > 0$ and $v \in L^\infty(\R)$, yields
\begin{align*}
\|u(t,\cdot) + u(t,-\,\cdot)\|_{L^\infty} &\leq C \int_0^t \frac{1}{\sqrt{t-s}} \||u(s,\cdot)| - |u(s,-\,\cdot)|\|_{L^\infty} \de s \\
&\leq C \int_0^t \frac{1}{\sqrt{t-s}} \|u(s,\cdot) + u(s,-\,\cdot)\|_{L^\infty} \de s,
\end{align*}
for $t \geq 0$. Therefore, Gr\"onwall's inequality, cf.~\cite[Lemma~7.0.3]{LUN}, implies that 
$$
\|u(t,\cdot) + u(t,-\,\cdot)\|_{L^\infty} = 0
$$ 
for all $t \geq 0$, which finishes the proof.
\end{proof}

The main result of this section is the following theorem. 

\begin{theorem}
Suppose Assumption~\ref{assumption-modular} holds. Take $\phi_* > 0$ and $x_{0,1} \in \R$ with $0 < x_1 - x_0 < \frac16$. Let $\phi \colon \R \to \R$ be the odd function given by 
\begin{align*}
\phi(x) = \phi_*(\re^{-x} - 1), 
\end{align*}
for $x \geq 0$. Consider $u_0 \in C_{\rm ub}^1(\R)$ satisfying
\begin{align} \label{compphi}
\phi(x-x_0) \leq u_0(x) \leq \phi(x-x_1),
\end{align}
for all $x \in \R$. Let $u \in C([0,\infty),C_{\mathrm{ub}}(\R))$ be the mild solution of~\eqref{Burgers-modular-sec3} established in Lemma~\ref{p:globalcub}. Then, $u(t,\cdot)$ cannot posses two consecutive simple zeros $\xi_1(t),\xi_2(t)$ that exist for all $t \geq 0$. 
\end{theorem}
\begin{proof}
Our analysis relies on comparison with an explicit reference solution $u_{\rm ref}(t,x)$ of~\eqref{Burgers-modular-sec3} with odd initial condition $u_{\rm ref}(0,\cdot) = \phi \in C_{\rm ub}^1(\R)$. By Lemma~\ref{lemma-initial}, the solution $u_{\rm ref} \in C([0,\infty),C_{\rm ub}(\mathbb{R}))$ is spatially odd. It satisfies the following diffusion-advection boundary-value problem:
\begin{align*}
\left\{ \begin{array}{lll} u_t = u_{xx} - u_x, & t > 0, & x > 0,\\ 
u(t,0) = 0, & t \geq 0, & \\
u(0,x) = \phi(x), & & x \geq 0, \end{array} \right.
\end{align*}
whose solution is explicitly given by
\begin{align*}
u_{\rm ref}(t,x) = \int_0^\infty G(t,x,y) \phi(y) \de y,
\end{align*}
for $t \geq 0$ and $x \in \R$, where $G(t,x,y)$ is the Green's function used in~\cite{PELI}:
\begin{align*}
G(t,x,y) = \frac{1}{\sqrt{4\pi t}} \left(\re^{-\frac{(x-y -t)^2}{4t}} - \re^{-y} \re^{-\frac{(x+y-t)^2}{4t}}\right).
\end{align*}
Evaluating the integral we find
\begin{align*}
 u_{\rm ref}(t,x) &= \frac{\phi_*}{2} \left(\re^x \text{erfc}\left(\frac{t+x}{2 \sqrt{t}}\right)-\text{erfc}\left(\frac{t-x}{2 \sqrt{t}}\right)\right.\\
 &\qquad \qquad \qquad \left. - \re^{2 t-x} \left(\text{erfc}\left(\frac{x-3 t}{2 \sqrt{t}}\right)+\re^{3 x} \text{erfc}\left(\frac{3 t+x}{2\sqrt{t}}\right)-2\right)\right),
\end{align*}
for $t \geq 0$ and $x \in \R$.

By the comparison principle, cf.~\cite[Corollary~3.1]{ENDJ}, and~\eqref{compphi} it holds
\begin{align*}
u_-(t,x) \leq u(t,x) \leq u_+(t,x)
\end{align*}
for $x \in \R$ and $t \geq 0$, where $u_-(t,x) = u_{\rm ref}(t,x - x_0)$ and $u_+(t,x) = u_{\rm ref}(t,x - x_1)$ are translates of the reference solution $u_{\rm ref}(t,x)$ of~\eqref{Burgers-modular-sec3}, see Figure~\ref{fig:antishockanalysis}. Note that $u_-(t,\cdot)$ and $u_+(t,\cdot)$ possess an odd symmetry with respect to the points $x = x_0$ and $x = x_1$, respectively. In particular, it holds $u_-(t,x_0) = 0 = u_+(t,x_1)$. 

We argue by contradiction and assume that $u(t,x)$ possesses zeros $\xi_1(t),\xi_2(t),\xi_3(t)$ for all $t \geq 0$ such that $\xi_1(t) < \xi_2(t) < \xi_3(t)$, $u(t,x) < 0$ for $x \in (\xi_1(t),\xi_2(t))$, $u(t,x) > 0$ for all $x \in (\xi_2(t),\xi_3(t))$, and $u_x(t,\xi_2(t)) > 0$ for all $t \geq 0$. Since $u_\pm(t,\cdot)$ are monotone, it holds $\xi_i(t) \in (x_0,x_1)$ for all $t \geq 0$ and $i = 1,2,3$, see Figure~\ref{fig:antishockanalysis}. By translational invariance, we may assume without loss of generality that $x_0 = 0$. 

As in the proof of Theorem~\ref{t:shockopposite}, we derive differential inequalities for the masses
\begin{align*}
M_1(t) = \int_{-\infty}^{\xi_2(t)} (\phi_* - u(t,x)) \de x, \qquad M_2(t) = \int_{\xi_2(t)}^\infty (u(t,x) + \phi_*) \de x.
\end{align*}
However, in contrast to the proof of Theorem~\ref{t:shockopposite}, we cannot employ the Gel'fand-Oleinik entropy inequality to bound $M_1(t)$ and $M_2(t)$. Instead, we use explicit expressions of the reference solutions $u_\pm(t,\cdot)$ to bound $M_1(0)$ and $M_2(0)$ from above and $M_1(t)$ and $M_2(t)$ from below.

Recalling $u_x(\xi_2(t),0) > 0$, the implicit function theorem implies that $\xi_2(t)$ is differentiable with respect to $t$. We apply Leibniz' rule to compute
\begin{align*}
M_2'(t) &= -\xi_2'(t)\phi_* + \int_{\xi_2(t)}^\infty (u_{xx}(t,x) + |u|_x(t,x)) \de x\\ 
&= -\xi_2'(t)\phi_* + \phi_* - u_x(t,\xi_2(t))\\ 
&< \phi_*\left(1-\xi_2'(t)\right).
\end{align*}
Integrating this inequality we arrive at
\begin{align*}
M_2(t) \leq M_2(0) + \phi_*t - \left(\xi_2(t) - \xi_2(0)\right)\phi_*.
\end{align*}
On the other hand, since $u(t,\cdot) - \phi_*$ and $u(t,\cdot) - u_-(t,\cdot)$ are nonnegative by the comparison principle, it holds
\begin{align*}
\left(\xi_3(t) - \xi_2(t)\right)\phi_* + \int_{x_1}^\infty \left(u_-(t,x) + \phi_*\right) \de x \leq M_2(t),
\end{align*}
see also Figure~\ref{fig:antishockanalysis}. Finally, since $u_+(0,\cdot) - u(0,\cdot)$ is nonnegative, we arrive at
\begin{align*}
 M_2(0) \leq 2x_1 \phi_* + \int_{x_1}^\infty (u_+(0,x) + \phi_*) \de x = 2x_1 \phi_* + \int_0^\infty \left(\phi(x) + \phi_*\right) \de x = \phi_*(2x_1 + 1).
\end{align*}
We compute
\begin{align*}
F(t) &:= \int_{x_1}^\infty \left(\frac{u_-(t,x)}{\phi_*} + 1\right) \de x - t\\ 
&= \frac{1}{4} \left(-(2 t+3) \text{erf}\left(\frac{x_1-t}{2 \sqrt{t}}\right)+2 x_1 \text{erfc}\left(\frac{t-x_1}{2 \sqrt{t}}\right)-2 \re^{2 t-x_1} \text{erfc}\left(\frac{x_1-3 t}{2 \sqrt{t}}\right)\right.\\
&\qquad \qquad \left. - \, 2 \re^{x_1} \text{erfc}\left(\frac{x_1+t}{2 \sqrt{t}}\right) + \re^{2 (x_1+t)} \text{erfc}\left(\frac{x_1+3 t}{2 \sqrt{t}}\right) +4 \re^{2 t-x_1}\right.\\
&\qquad \qquad \left. + \, \frac{4 \sqrt{t}}{\sqrt{\pi}} \re^{-\frac{(x_1-t)^2}{4 t}} - 4x_1-2 t+3\right),
\end{align*}
and obtain
\begin{align*}
\lim_{t \to \infty} F(t) = \frac{3}{2} - x_1.
\end{align*}

All in all, we have established
\begin{align*}
\phi_*\left(\xi_3(t) - \xi_2(t) + F(t) + t\right) \leq M_2(t) \leq \phi_*\left(2x_1 + 1 + t - \left(\xi_2(t) - \xi_2(0)\right)\right),
\end{align*}
yielding
\begin{align*}
\xi_3(t) - \xi_2(0) \leq 2x_1 + 1 - F(t) \to 3x_1 - \frac{1}{2} \quad \mbox{\rm as} \;\; t \to \infty.
\end{align*}
Consequently, as $x_1 < \frac{1}{6}$ there exists $t_2 > 0$ such that $\xi_2(t) < \xi_3(t) \leq \xi_2(0)$ for all $t \geq t_2$.

Similarly, by bounding the integral $M_1(t)$, one finds $t_1 > 0$ such that $\xi_2(0) \leq \xi_1(t) < \xi_2(t)$ for all $t \geq t_1$, which contradicts the fact that $\xi_2(t) < \xi_2(0)$ for all $t \geq t_2$. Hence, the interfaces $\xi_1(t), \xi_2(t)$ and $\xi_3(t)$ must coalesce within finite time. 
\end{proof}

\begin{figure}[htb]
	\centering
	\includegraphics[trim = 1 1 1 1, clip, scale=0.7]{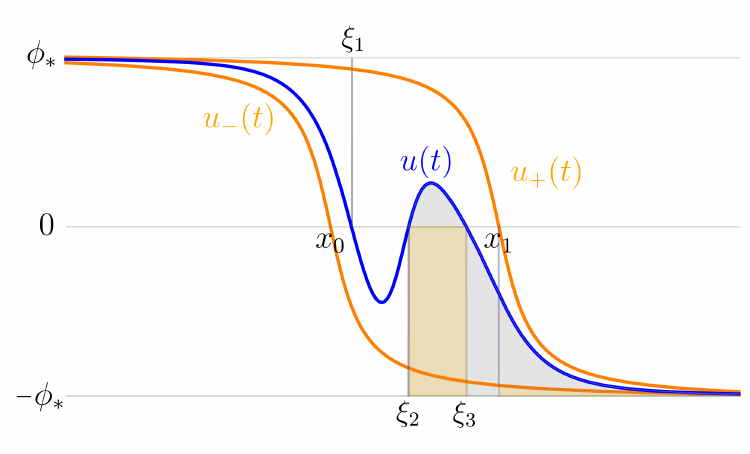}
	\caption{The anti-shock wave $u(t,\cdot)$ connecting the asymptotic end states $\mp \phi_*$ at $\pm \infty$, the odd subsolution $u_-(t,\cdot)$ with zero $x_0$, the odd supersolution $u_+(t,\cdot)$ with zero $x_1$ and the interfaces $\xi_1(t)$, $\xi_2(t)$ and $\xi_3(t)$ of $u(t,\cdot)$ (we suppressed the $t$-dependency of the interfaces). We bound the shaded area below the graph of $u(t,\cdot)$ from below by the orange subareas.}
	\label{fig:antishockanalysis}
\end{figure}

\section{Dynamics near a coalescence event for smooth flux functions} \label{s:interfacedynamics}

Let us consider the initial-value problem for the viscous conservation law:
\begin{equation}
\label{Burgers}
\left\{ \begin{array}{lll} 	u_t = u_{xx} + f'(u) u_x, \quad & t > 0, \quad & x \in \mathbb{R}, \\
u(0,x) = u_0(x), \quad & \quad & x \in \R, \end{array} \right.
\end{equation}
where $f \in C^{\infty}(\R)$ satisfies $f'(0) = 0$. We assume that the initial condition $u_0 \in C^{\infty}(\R)$ is bounded and has bounded derivatives.

From the well-posedness of the viscous conservation law in the class of smooth data, cf. Lemma~\ref{p:globalsmooth}, we know that there exists a smooth solution $u \in C^{\infty}((0,\infty) \times \R,\R)$ to the initial-value problem~(\ref{Burgers}). A zero $x = \xi(t)$ of $u(t,\cdot)$ on $\R$ is a $C^1$-function of $t$ as long as $u_x(t,\xi(t)) \neq 0$ by the implicit function theorem. 

Here we classify the first two bifurcations for which 
the function $t \to \xi(t)$ exists for $t$ in some interval $[0,t_0]$ with $t_0 > 0$ such that $u(t,\xi(t)) = 0$ for $t \in [0,t_0]$ and $u_x(t,\xi(t)) \neq 0$ for $t \in [0,t_0)$ but may fail to exist for $t > t_0$ because we have
$u_x(t_0,\xi_0) = 0$ at $\xi_0 = \xi(t_0)$. 

\subsection{Fold bifurcation} 

The main result is given by the following proposition.

\begin{proposition}
	\label{prop-fold}
Assume that there exists $(t_0,\xi_0) \in (0,\infty) \times \R$ such that 
$$
u_x(t_0,\xi_0) = 0 \quad \mbox{\rm and} \quad u_{xx}(t_0,\xi_0) \neq 0.
$$ 
Then, there exist two roots of $u(t,\cdot)$ near $\xi_0$ for $t < t_0$ near $t_0$, denoted by $\xi_{1,2}(t)$, such that 
\begin{equation}
\label{claim-xi-fold}
\xi_{1,2}(t) - \xi_0 = \pm \sqrt{2(t_0 - t)}  + \mathcal{O}(t_0-t) \quad \mbox{\rm as} \quad t \to t_0^-
\end{equation}
and
\begin{equation}
\label{claim-der-fold}
u_x(t,\xi_{1,2}(t)) = \pm \sqrt{2(t_0 - t)} u_{xx}(t_0,\xi_0) + \mathcal{O}(t_0-t) \quad \mbox{\rm as} \quad t \to t_0^-. 
\end{equation}
No roots of $u(t,\cdot)$ near $\xi_0$ exist for $t > t_0$ near $t_0$.
\end{proposition}

\begin{proof}
By using the equation of motion in~(\ref{Burgers}), we have 
$$
u_t(t_0,\xi_0) = u_{xx}(t_0,\xi_0) \neq 0.
$$
Moreover, using Taylor expansions for smooth solutions, we obtain 
for any root $\xi(t)$ of $u(t,\cdot)$ near $\xi_0$:
\begin{align*}
0 & = u(t,\xi(t)) \\
&= \underbrace{u(t_0,\xi_0)}_{=0} + (t-t_0) \underbrace{u_t(t_0,\xi_0)}_{\neq 0} 
+ (\xi(t) - \xi_0)  \underbrace{u_x(t_0,\xi_0)}_{=0} \\
& \quad + \frac{1}{2} (t-t_0)^2 u_{tt}(t_0,\xi_0) 
+ (t-t_0) (\xi(t) - \xi_0)  u_{tx}(t_0,\xi_0) 
+ \frac{1}{2} (\xi(t) - \xi_0) ^2  \underbrace{u_{xx}(t_0,\xi_0)}_{\neq 0} + \mathcal{O}(3).
\end{align*}
It follows from the Newton's polygon in Figure~\ref{fig:Newton} that this expansion defines two roots for $\xi(t)$, denoted by $\xi_{1,2}(t)$, which are given by 
the expansion
\begin{align*}
\xi_{1,2}(t) - \xi_0 &= \pm \sqrt{\frac{2 u_{t}(t_0,\xi_0)}{u_{xx}(t_0,\xi_0)} (t_0 -t)} + \mathcal{O}(t_0-t) \\
&= \pm \sqrt{2 (t_0 -t)} + \mathcal{O}(t_0-t),
\end{align*}
which exist for $t < t_0$ near $t_0$, coalesce at $t = t_0$ and disappear for $t > t_0$. 
We also obtain 
\begin{align*}
u_x(t,\xi_{1,2}(t)) &= \underbrace{u_x(t_0,\xi_0)}_{=0} + (t-t_0) u_{tx}(t_0,\xi_0) 
+ (\xi_{1,2}(t) - \xi_0)   \underbrace{u_{xx}(t_0,\xi_0)}_{\neq 0} + \mathcal{O}(2) \\
&= \pm \sqrt{2(t_0-t)} u_{xx}(t_0,\xi_0) + \mathcal{O}(t_0 - t).
\end{align*}
Both expansions prove the validity of~(\ref{claim-xi-fold}) and~(\ref{claim-der-fold}).
\end{proof}

\subsection{Pitchfork bifurcation}

The main result is given by the following proposition. 

\begin{proposition}
	\label{prop-pitchfork}
	Assume that there exists $(t_0,\xi_0) \in (0,\infty) \times \R$ such that 
	$$
u_x(t_0,\xi_0) = 0, \quad u_{xx}(t_0,\xi_0) = 0, \quad \mbox{\rm and} \quad u_{xxx}(t_0,\xi_0) \neq 0.
	$$ 
Then, there exist three roots of $u(t,\cdot)$ near $\xi_0$ for $t < t_0$ near $t_0$ and one root near $\xi_0$ for $t > t_0$ near $t_0$. Two roots, denoted by $\xi_{1,2}(t)$, are not continued for $t > t_0$ and satisfy 
\begin{equation}
\label{claim-xi}
\xi_{1,2}(t) - \xi_0 = \pm \sqrt{6(t_0 - t)} + \mathcal{O}(t_0-t) 
\quad \mbox{\rm as} \quad t \to t_0^-, 
\end{equation}
whereas the third root, denoted by $\xi(t)$, is continued for $t > t_0$ and  satisfies 
\begin{equation}
\label{claim-xi-single}
\xi(t) - \xi_0 = \frac{u_{tt}(t_0,\xi_0)}{2 u_{xxx}(t_0,\xi_0)} (t_0 - t) + \mathcal{O}((t_0-t)^2) 
\quad \mbox{\rm as} \quad t \to t_0. 
\end{equation}
We also have
\begin{equation}
\label{claim-der}
u_x(t,\xi_{1,2}(t)) = 2 u_{xxx}(t_0,\xi_0) (t_0 - t) + \mathcal{O}((t_0-t)^{3/2}) \quad \mbox{\rm as} \;\; t \to t_0^- 
\end{equation}
and
\begin{equation}
\label{claim-der-single}
u_x(t,\xi(t)) = u_{xxx}(t_0,\xi_0) (t - t_0) + \mathcal{O}((t_0-t)^2) \quad \mbox{\rm as} \;\; t \to t_0.
\end{equation}
\end{proposition}

\begin{remark}
The scaling laws~(\ref{claim-xi}) and~(\ref{claim-der}) were conjectured in~\cite{PEDR} based on numerical simulations of spatially odd 
solutions of the modular Burgers' equation. 
Proposition~\ref{prop-pitchfork} shows that this behavior holds for every scalar viscous conservation law~\eqref{modBurgers} with smooth nonlinearity and smooth initial data. 
\end{remark}

\begin{remark}
The asymptotic expansions~(\ref{claim-xi}) and~(\ref{claim-der}) imply 
\begin{equation}
\label{claim-der-second}
u_{xx}(t,\xi_{1,2}(t)) = \pm u_{xxx}(t_0,\xi_0) \sqrt{6(t_0 - t)} + \mathcal{O}(t_0-t) \quad \mbox{\rm as} \quad t \to t_0^-,
\end{equation}
which was also conjectured in~\cite{PEDR}. Indeed, if we differentiate $u(t,\xi(t)) = 0$ with the chain rule for the smooth solutions $u \in C^{\infty}((0,\infty) \times \R,\R)$ for $t \in (0,t_0)$, while assuming that $u_x(t,\xi(t)) \neq 0$, $\xi(t)$ of $u(t,\cdot)$, then we obtain from~(\ref{Burgers}) with $f'(0) = 0$:
$$
u_t(t,\xi(t)) + \xi'(t) u_x(t,\xi(t)) = 0 \quad \Rightarrow \quad u_{xx}(t,\xi(t)) = -\xi'(t) u_x(t,\xi(t)),
$$
for $t \in (0,t_0)$. Hence,~(\ref{claim-xi}) and~(\ref{claim-der}) imply~(\ref{claim-der-second}). Similarly, we can derive from~(\ref{claim-xi-single}) and~(\ref{claim-der-single}):
\begin{equation}
\label{claim-der-second-single}
u_{xx}(t,\xi(t)) = \frac{1}{2} u_{tt}(t_0,\xi_0) (t-t_0) + \mathcal{O}((t_0-t)^2) \quad \mbox{\rm as} \quad t \to t_0,
\end{equation}
for the third root $\xi(t)$ which exists for all $t$ near $t_0$.
\end{remark}

\begin{remark}
It follows from~(\ref{claim-xi}) and~(\ref{claim-xi-single}) that the three interfaces satisfy the natural ordering for the pitchfork bifurcation 
	$$
	\xi_1(t) < \xi(t) < \xi_2(t),
	$$
for $t < t_0$ near $t_0$. It follows from~(\ref{claim-der}) and~(\ref{claim-der-single}) that 
the sign of the first partial derivative of $u(t,x)$ in $x$ at $x = \xi(t)$ is opposite to the one at $x = \xi_{1,2}(t)$ for $t < t_0$ near $t_0$.  
\end{remark}

\begin{remark}
	\label{remark-odd}
	If $u_0(-x) = -u_0(x)$ and $f'(-z) = -f'(z)$ for $z \in \R$ in~(\ref{Burgers}), then $u(t,-x) = -u(t,x)$ for every $t > 0$ and $x \in \R$. In this case of odd symmetry, if the assumptions of Proposition~\ref{prop-pitchfork} are satisfied and $\xi_0 = 0$, then $\xi(t) = 0$ for all $t \geq 0$. Consequently, we have
	$$
	u(t,0) = u_{xx}(t,0) = 0,
	$$
 for all $t \geq 0$.
\end{remark}

\begin{proof}[Proof of Proposition~\ref{prop-pitchfork}.]
By using the equation of motion in~(\ref{Burgers}), we have 
$$
u_t(t_0,\xi_0) = 0 \quad \mbox{\rm and } \quad 
u_{tx}(t_0,\xi_0) = u_{xxx}(t_0,\xi_0) \neq 0.
$$
Moreover, using Taylor expansions for smooth solutions, we obtain 
for any root $\xi(t)$ of $u(t,\cdot)$ near $\xi_0$:
\begin{align*}
0 & = u(t,\xi(t)) \\
&= \underbrace{u(t_0,\xi_0)}_{=0} + (t-t_0) \underbrace{u_t(t_0,\xi_0)}_{=0} 
+ (\xi(t) - \xi_0) \underbrace{u_x(t_0,\xi_0)}_{=0} \\
& \quad + \frac{1}{2} (t-t_0)^2 u_{tt}(t_0,\xi_0) 
+ (t-t_0) (\xi(t) - \xi_0) \underbrace{u_{tx}(t_0,\xi_0)}_{\neq 0} 
	+ \frac{1}{2} (\xi(t) - \xi_0)^2  \underbrace{u_{xx}(t_0,\xi_0)}_{=0} \\
	& \quad + \frac{1}{6} (t-t_0)^3 u_{ttt}(t_0,\xi_0)
	+ \frac{1}{2} (t-t_0)^2 (\xi(t) - \xi_0) u_{ttx}(t_0,\xi_0) \\
	& \quad + \frac{1}{2} (t-t_0) (\xi(t) - \xi_0)^2  u_{txx}(t_0,\xi_0)+ \frac{1}{6} (\xi(t) - \xi_0)^3 \underbrace{u_{xxx}(t_0,\xi_0)}_{\neq 0} + \mathcal{O}(4).
\end{align*}
It follows from the Newton's polygon in Figure~\ref{fig:Newton} that this expansion defines two sets of roots. One set appears at the balance of $(t-t_0)(\xi(t)-\xi_0)$ and $(\xi(t)-\xi_0)^3$ terms and the other set appears at the balance between $(t-t_0)^2$ and $(t-t_0)(\xi(t)-\xi_0)$ terms. 

The former set is represented by two roots denoted as $\xi_{1,2}(t)$ 
which satisfy the expansion 
\begin{align*}
\xi_{1,2}(t) - \xi_0 &= \pm \sqrt{\frac{6 u_{tx}(t_0,\xi_0)}{u_{xxx}(t_0,\xi_0)} (t_0 -t)} + \mathcal{O}(t_0-t) \\
&= \pm \sqrt{6 (t_0 -t)} + \mathcal{O}(t_0-t).
\end{align*}
The two roots exist for $t < t_0$ near $t_0$, coealesce at $t = t_0$ 
and disappear for $t > t_0$. We also obtain 
\begin{align*}
u_x(t,\xi_{1,2}(t)) & = \underbrace{u_x(t_0,\xi_0)}_{=0} 
+ (t-t_0) \underbrace{u_{tx}(t_0,\xi_0)}_{\neq 0}
+ (\xi_{1,2}(t) - \xi_0) \underbrace{u_{xx}(t_0,\xi_0)}_{=0} \\
& \quad + \frac{1}{2} (t-t_0)^2 u_{ttx}(t_0,\xi_0) 
+ (t-t_0) (\xi_{1,2}(t) - \xi_0)   u_{txx}(t_0,\xi_0)\\
& \quad + \frac{1}{2} (\xi_{1,2}(t) - \xi_0)^2 \underbrace{u_{xxx}(t_0,\xi_0)}_{\neq 0} + \mathcal{O}(3),
\end{align*}
which implies 
\begin{equation*}
u_x(t,\xi_{1,2}(t)) = 2 u_{xxx}(t_0,\xi_0) (t_0 - t) + \mathcal{O}((t_0-t)^{3/2}) \quad \mbox{\rm as} \;\; t \to t_0^-.
\end{equation*}
These expansions prove the validity of~(\ref{claim-xi}) and~(\ref{claim-der}). 

The latter set is represented by one root denoted by $\xi(t)$ which satisfies the expansion 
\begin{align*}
\xi(t) - \xi_0 &= \frac{u_{tt}(t_0,\xi_0)}{2 u_{tx}(t_0,\xi_0)} (t_0 -t) + \mathcal{O}((t_0-t)^2) \\
&= \frac{u_{tt}(t_0,\xi_0)}{2 u_{xxx}(t_0,\xi_0)} (t_0 -t) + \mathcal{O}((t_0-t)^2).
\end{align*}
The root $\xi(t)$ persists for all $t$ near $t_0$. We also obtain 
\begin{align*}
u_x(t,\xi(t)) & = \underbrace{u_x(t_0,\xi_0)}_{=0} 
+ (t-t_0) \underbrace{u_{tx}(t_0,\xi_0)}_{\neq 0}
+ (\xi(t) - \xi_0) \underbrace{u_{xx}(t_0,\xi_0)}_{=0} \\
& \quad + \frac{1}{2} (t-t_0)^2 u_{ttx}(t_0,\xi_0) 
+ (t-t_0) (\xi(t) - \xi_0)   u_{txx}(t_0,\xi_0) \\
& \quad + \frac{1}{2} (\xi(t) - \xi_0)^2 \underbrace{u_{xxx}(t_0,\xi_0)}_{\neq 0} + \mathcal{O}(3),
\end{align*}
which implies 
\begin{equation*}
u_x(t,\xi(t)) = u_{xxx}(t_0,\xi_0) (t - t_0) + \mathcal{O}((t_0-t)^2) \quad \mbox{\rm as} \;\; t \to t_0.
\end{equation*}
These expansions prove the validity of~(\ref{claim-xi-single}) and~(\ref{claim-der}). 
\end{proof}

\begin{figure}[htb]
	\centering
	\begin{subfigure}{0.49\textwidth}
	\includegraphics[scale=0.42]{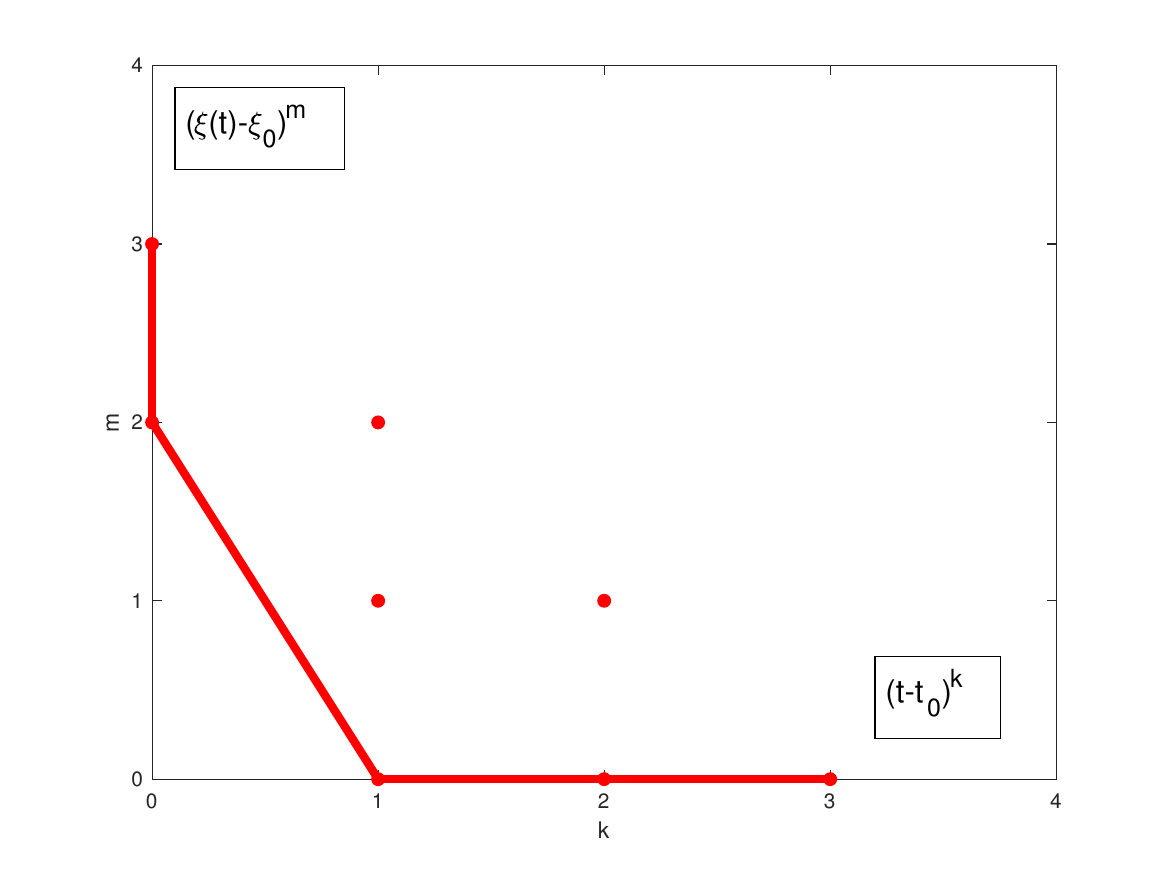}
    \end{subfigure}
	\hfill
	\begin{subfigure}{0.49\textwidth}
	\includegraphics[scale=0.42]{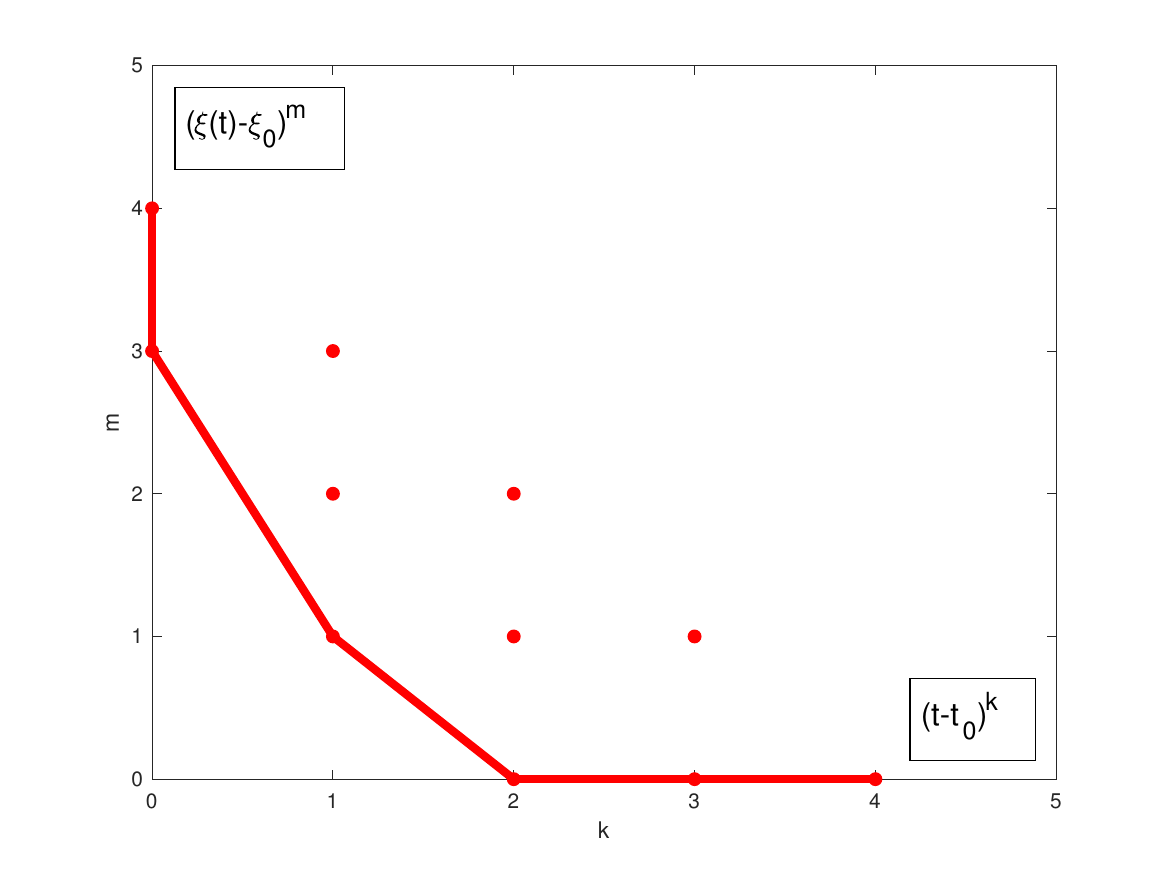}
	\end{subfigure}
	\caption{Newton's polygons used in the proofs of Proposition~\ref{prop-fold} (left) and Proposition~\ref{prop-pitchfork} (right).}
	\label{fig:Newton}
\end{figure}

\subsection{Bifurcations of higher order}
By continuing the analysis from the previous two subsections, one can characterize coalescence of roots of $u(t,\cdot)$ in the non-generic case when there exists an integer 
$m \geq 4$ and $(t_0,\xi_0) \in (0,\infty) \times \R$ such that 
all partial derivatives of $u(t,x)$ in $x$ at $(t_0,\xi_0)$ are zero 
up to the $m$-th order and the $m$-th partial derivative 
of $u(t,x)$ in $x$ at $(t_0,\xi_0)$ is nonzero.

\subsection{Viscous Burgers' equation with quadratic nonlinearity}

We give a precise description of a class of solutions to the viscous Burgers' equation whose zeros undergo a pitchfork bifurcation. Thus, we take $f(u) = u^2$ in~\eqref{Burgers} and consider the initial value problem for the Burgers' equation
\begin{equation}
\label{Burgers-quad}
\left\{ \begin{array}{l} u_t = u_{xx} + 2 u u_x, \quad t > 0, \\
u(0,x) = u_0(x), \quad x \in \R. \end{array} \right.
\end{equation}
As is well-known,~(\ref{Burgers-quad}) can be solved explicitly using the Cole-Hopf transformation (see Section 3.6 in~\cite{MILL}). We will use the decomposition near the stationary shock-wave solution $\phi(x) = \tanh(x)$ of~\eqref{Burgers-quad} to show that the pitchfork bifurcation of Proposition~\ref{prop-pitchfork} does happen within finite time for all solutions of~\eqref{Burgers-quad} with  spatially odd initial data $u_0$ having a single zero on $(0,\infty)$. The main result is given by the following proposition. 

\begin{proposition}
	\label{prop-Burgers}
	Let $u_0 \in C^{\infty}(\R)$ satisfy 
\begin{itemize}
	\item $u_0 \mp 1$, $u_0'$ and $u_0''$ are $L^2$-integrable on $\R_\pm$,
	\item $u_0(-x) = -u_0(x)$ for $x \in \mathbb{R}$,
	\item for some $x_0 \in \mathbb{R}_+$, we have $u_0(x) < 0$ for $x \in (0,x_0)$ and $u_0(x) > 0$ for $x \in (x_0,\infty)$.
\end{itemize}
Then, there exist a time $t_0 \in (0,\infty)$ and $\xi \in C^{\infty}((0,t_0),\mathbb{R}_+)$ such that 
	the solution $u \in C^{\infty}((0,\infty) \times \R,\R)$ to the initial-value problem~(\ref{Burgers-quad}) satisfies:
\begin{itemize}
	\item[(i)] $\lim_{x \to \pm \infty} u(t,x) = \pm 1$ for $t \geq 0$,
	\item[(ii)] $u(t,-x) = -u(t,x)$ for $t \geq 0$ and $x \in \mathbb{R}$, 
	\item[(iii)] $u(t,x) < 0$ for $x \in (0,\xi(t))$ and $u(t,x) > 0$ for $x \in (\xi(t),\infty)$ if $t \in [0,t_0)$, 
	\item[(iv)] $u(t,x) > 0$ for $x \in (0,\infty)$ if $t \geq t_0$.
\end{itemize}
Moreover, we have $u_x(t_0,0) = 0$, $u_{xx}(t_0,0) = 0$, and $u_{xxx}(t_0,0) > 0$.
\end{proposition}

\begin{remark}
	For $u \in C^{\infty}((0,\infty) \times \R,\R)$ and $\xi \in C^{\infty}((0,t_0),\mathbb{R}_+)$ in Proposition~\ref{prop-Burgers} we obtain the identities~\eqref{claim-xi},~\eqref{claim-xi-single},~\eqref{claim-der} and~\eqref{claim-der-single}, since the assumptions of Proposition~\ref{prop-pitchfork} are satisfied.
\end{remark}

\begin{proof}[Proof of Proposition~\ref{prop-Burgers}.]
We use the decomposition of $u$ at the stationary shock-wave solution $x \mapsto \tanh(x)$ of $0 = 2 u u_x + u_{xx}$ and write
\begin{equation}
\label{decomp-u}
u(t,x) = \tanh(x) + v(t,x).
\end{equation}
The perturbation $v$ (which is not necessarily small) satisfies 
\begin{equation}
\label{v-eq}
v_t = v_{xx} + 2 v v_x + 2(\tanh(x) v)_x.
\end{equation}
This nonlinear equation can be linearized with the Cole-Hopf transformation 
\begin{equation}
\label{decomp-v}
v(t,x) = \partial_x \log \psi(t,x).
\end{equation}
By substituting~(\ref{decomp-v}) into~(\ref{v-eq}), we obtain the following linear advection-diffusion equation 
\begin{equation}
\label{adv-diff}
\psi_t = \psi_{xx} + 2 \tanh(x) \psi_x.
\end{equation}
We are looking for a solution of~(\ref{adv-diff}) which is bounded away from zero by a positive constant. Without loss of generality, this constant can be normalized to unity, so that we can look for a solution of the form 
\begin{equation}
\label{decomp-psi}
\psi(t,x) = 1 + \hat{\psi}(t,x), \quad \hat{\psi}(t,\cdot) \in H^2(\R), \quad t \in \mathbb{R}_+.
\end{equation}
To obtain the exact solution of~(\ref{adv-diff}), we write 
\begin{equation}
\label{decomp-psi-hat}
\hat{\psi}(t,x) = {\rm sech}(x) \chi(t,x) 
\end{equation}
and obtain the linear diffusion equation with constant dissipation for $\chi$:
\begin{equation}
\label{eq-chi}
\chi_t = \chi_{xx} - \chi.
\end{equation}
The solutions of this linear equation are given by
\begin{equation}\label{sol-chi}
\chi(t,x) = \frac{\re^{-t}}{\sqrt{4\pi t}} \int_{\R} \chi_0(y) \re^{-\frac{(x-y)^2}{4t}} \de y,
\end{equation}
where $\chi_0 := \chi(0,\cdot)$ denotes the initial condition. The associated solution of the Burgers' equation~(\ref{adv-diff}) is then obtained from~(\ref{decomp-u}),~(\ref{decomp-v}), 
(\ref{decomp-psi}), and~(\ref{decomp-psi-hat}) in the form:
\begin{equation}
\label{sol-u}
u(t,x) = \frac{\sinh(x) + \chi_x(t,x)}{\cosh(x) + \chi(t,x)},
\end{equation}
where $\chi(t,x)$ is given by~(\ref{sol-chi}).

If $\chi_0 \in C^{\infty}(\R) $ satisfies $1 + \chi_0''(0) < 0$ and $\chi_0(0) > 0$, then 
$$
u_0'(0) = (1 + \chi_0''(0))/(1+\chi_0(0)) < 0,
$$ 
so that there exists a root $x_0 \in \R_+$ of $u_0$. The positive root $x_0$ must be unique by the assumptions on $u_0$. Thus, we find by~\eqref{decomp-u},~\eqref{decomp-v},~\eqref{decomp-psi} and~\eqref{decomp-psi-hat} that the assumptions on $u_0$ are in one-to-one correspondence with the class of even functions $\chi_0 \in C^{\infty}(\R)$ such that ${\rm sech}(\cdot) \chi_0 \in H^2(\R)$ and 
\begin{itemize}
	\item $\chi_0(x) > 0$ for all $x \in \R$, 
	\item $x \mapsto \cosh(x) + \chi_0''(x)$ is monotonically increasing on $\R_+$ with $1 + \chi_0''(0) < 0$.
\end{itemize}
Now take such $\chi_0 \in C^\infty(\R)$. Then, $\cosh(x) + \chi_0(x) > 0$ for all $x \in \R$ and 
$\sinh(x) + \chi_0'(x)$ has a single root $x_0 \in (0,\infty)$. 
Since $\chi_0$ is even, so is $\chi \in C^{\infty}((0,\infty) \times \R,\R)$, which implies that $u(t,\cdot) \in C^{\infty}((0,\infty) \times \R,\R)$ is spatially odd, so that (ii) holds. Furthermore, 
${\rm sech}(\cdot) \chi_0 \in H^2(\R)$ ensures by~\eqref{sol-chi} that 
${\rm sech}(\cdot) \chi(t,\cdot) \in H^2(\R)$ for all $t \geq 0$. 
Since $\hat{\psi}(t,\cdot) \in H^2(\R)$ for all $t \geq 0$,  we have from~(\ref{decomp-u}),~(\ref{decomp-v}), and~(\ref{decomp-psi}) that 
$\lim_{x \to \pm \infty} v(t,x) = 0$ and $\lim_{x \to \pm \infty} u(t,x) = \pm 1$, so that (i) holds. 

It follows from the exact solution~(\ref{sol-chi}) that for every $t \geq 0$, we have $\chi(t,x) > 0$ for all $x \in \R$ and $x \mapsto \cosh(x) + \chi_{xx}(t,x)$ is monotonically increasing on $(0,\infty)$. 
Hence, $\cosh(x) + \chi(t,x) > 0$ for all $x \in \R$ and $\sinh(x) + \chi_x(t,x)$ has a single root $\xi(t) \in (0,\infty)$ for $t \in [0,t_0)$ 
as long as $1 + \chi_{xx}(t,0) < 0$. Since 
$$
\chi_{xx}(t,0) = \frac{\re^{-t}}{\sqrt{4\pi t}} \int_{\R} \chi_0''(y) \re^{-\frac{y^2}{4t}} \de y
$$
and ${\rm sech}(\cdot) \chi_0 \in H^2(\R)$, the mapping $t \mapsto \chi_{xx}(t,0)$ is monotonically increasing from a negative value 
$\chi_0''(0) < -1$ towards $0$ as $t \to +\infty$. Hence, there exists a unique time $t_0 \in \R_+$ such that $1 + \chi_{xx}(t,0)$ crosses $0$ at $t = t_0$ and becomes positive for $t > t_0$ so that (iii) and (iv) hold. 

Let us now show the non-degeneracy assumption at $t = t_0$ for which 
$u_x(t_0,0) = 0$. Since the solution is smooth and spatially odd, we also 
have $u_{xx}(t_0,0) = 0$. Since the mapping $t \mapsto \chi_{xx}(t,0)$ is monotonically increasing and $t \mapsto \chi(t,0)$ is monotonically decreasing, then $t \mapsto u_x(t,0)$ is monotonically increasing,
where 
$$
u_x(t,0) = \frac{1 + \chi_{xx}(t,0)}{1 + \chi(t,0)}.
$$
Thus, $u_{tx}(t_0,0) > 0$ and the Burgers' equation in~(\ref{Burgers-quad}) implies $u_{xxx}(t_0,0) > 0$.
\end{proof}

\begin{figure}[htb]
	\centering
	\includegraphics[height=6cm,width=10cm]{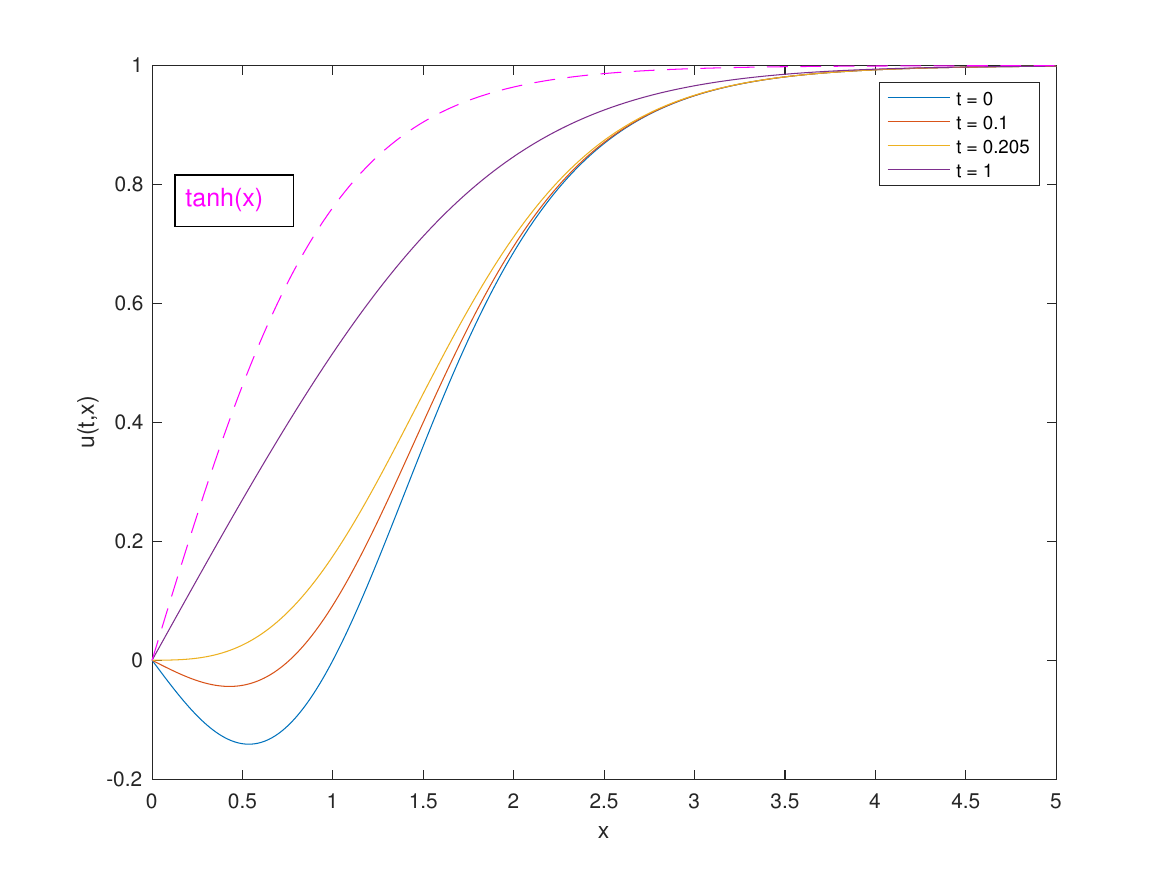}
	\caption{An illustration of the exact solution to the Burgers' equation~(\ref{Burgers-quad}) obtained by means of~(\ref{sol-chi}) and~(\ref{sol-u}). The initial condition for~(\ref{sol-chi}) 
is set as $\chi_0(x) := \cosh^2(1) {\rm sech}(x)$ so that 
the initial condition $u_0$ for~(\ref{Burgers-quad}) has a positive zero at $x = 1$. The integration 
of the exact solution in~(\ref{sol-chi}) was executed by using a numerical integration package. The root $\xi(t)$ of $u(t,\cdot)$ on $(0,\infty)$ exists for $t \in [0,t_0)$, coalesces at $0$ at $t = t_0$ and disappears for $t > t_0$, where $t_0 \approx 0.205$. The solution $u(t,x)$ approaches the stationary shock wave $x \mapsto \tanh(x)$ as $t \to \infty$, which is represented by the dashed line.}
	\label{fig:Burgers}
\end{figure}

\section{Numerical simulations in the modular Burgers' equation} \label{sec:numerics}
 
Here we report on numerical simulations in the viscous Burgers' equation with modular nonlinearity. The associated initial value problem reads
\begin{equation}
\label{Burgers-modular}
\left\{ \begin{array}{lll} 	u_t = u_{xx} + \lvert u \rvert_x, \quad & t > 0, \quad & x \in \R, \\
u(0,x) = u_0(x), \quad & \quad & x \in \R. \end{array} \right.
\end{equation}
Numerical computations in~\cite{PEDR} implemented the finite-difference 
method for spatially odd solutions of~(\ref{Burgers-modular}), see Lemma~\ref{lemma-initial}, 
for which the initial-value problem~(\ref{Burgers-modular}) can be closed on the half-line $[0,\infty)$ subject to a Dirichlet boundary condition at $x=0$. The jump condition ~(\ref{interface-con}) was used at $x = 0$ as well as at $x = \pm \xi(t)$. The three interfaces were transformed to  time-independent grid points after a scaling transformation. 

We will confirm the scaling law~(\ref{law-pitchfork1}) of the finite-time extinction of multiple interfaces in the initial-value problem~(\ref{Burgers-modular}). Compared to the previous numerical simulations in~\cite{PEDR}, we use a regularization for the modular nonlinearity, for which 
the finite-difference method can be implemented without any additional equations for the interface dynamics. The numerical data is extracted from zeros of the solution $u(t,\cdot)$ on $(0,\infty)$ to determine the power of the scaling law of the interface coalescence. 

\subsection{Regularization}

The modular Burgers' equation can be rewritten as 
\begin{equation}
\label{Burgers-modular-2}
u_t =  u_{xx} + {\rm sgn}(u) u_x,
\end{equation}
where ${\rm sgn}(u)$ has a jump discontinuity at $u=0$. To smoothen out the jump, we define the following smooth nonlinearity for $\varepsilon > 0$,
\begin{equation*}
f'_\varepsilon(u):=\frac{u}{\sqrt{\varepsilon^2 + u^2}}.
\end{equation*}
We have $f'_\varepsilon(u) \to {\rm sgn}(u)$ as $\varepsilon \to 0$ for all $u\in \mathbb{R}$, i.e.~$f_\varepsilon'(u)$ converges pointwise to $\mathrm{sgn}(u)$. This yields the regularized equation
\begin{equation}
\label{Burgers-regularized}
u_t = u_{xx} + \frac{u}{\sqrt{\varepsilon^2 + u^2}} u_x.
\end{equation}
We consider initial data $u(0,x) = u_0(x)$ for shock and anti-shock waves with the boundary condition $u_0(x) \to u_{\pm}$ as $x \to \pm \infty$, where $u_\pm$ have opposite signs. The case of $u_- < 0 < u_+$ includes a monotone, steadily traveling shock wave, to which the evolution of small exponentially decaying perturbations converges~\cite{PELI}. The anti-shock case of $u_- > 0 > u_+$ does not admit any steadily traveling shock-wave solutions.

For the simulation of shock-wave solutions with the normalized asymptotic limits $u_{\pm} = \pm 1$, we take the following initial condition:
\begin{equation}
\label{phi_2}
u_0(x) = \tanh(x) \left( 1-\re^{\alpha(1-x^2)} \right),
\end{equation}
where $\alpha > 0$ is a free parameter. The parameter $\alpha >0$ can be used to construct slopes of the initial data at $x=1$. For the simulation of anti-shock wave solutions with the normalized asymptotic limits $u_{\pm} = \mp 1$, we take the negative version of~(\ref{phi_2}), that is, 
\begin{equation}
\label{-phi_2}
u_0(x)=-\tanh(x) \left( 1-\re^{\alpha(1-x^2)} \right).
\end{equation}
Both in ~(\ref{phi_2}) and~(\ref{-phi_2}), the convergence 
of $u_0(x) \to u_{\pm}$ as $x \to \pm \infty$ is exponentially fast.

\subsection{Finite-difference method}

We rewrite the regularized Burgers' equation~(\ref{Burgers-regularized}) 
in the equivalent form,
\begin{equation}
\label{Regularized-Burgers}
u_t = u_{xx} + f_{\varepsilon}(u)_x, 
\end{equation}
with $f_{\varepsilon}(u) = \sqrt{\varepsilon^2 + u_\varepsilon^2}-\varepsilon$.

We will use the Crank-Nicolson method based on the trapezoidal rule to set up our numerical simulations for the equation~(\ref{Regularized-Burgers}). For the numerical discretization, we first define the spatial domain $[0,L]$ partitioned into $(N+1)$ grid points with spatial step $h$ and the time domain $[0,T]$ partitioned into $M$ grid points with time step $\tau$. We let $x_n$ for $0 \le n \le N$ be the spatial grid point and $t_m$ for $0 \le m \le M$ be the temporal grid point. We impose a Dirichlet condition at $x=0$ which yields $u_0^m = 0$ and a Neumann condition at $x=L$. By using the virtual grid point $x_{N+1} > L$, the Neumann condition reads $u_{N+1}^m = u_{N-1}^m$.

The Crank-Nicolson method is based on the discretization rule,
\begin{multline}
u^{m+1} = u^m + \frac{\tau}{2h^2} \left[ u_{n+1}^m -2u_n^m + u_{n-1}^m + u_{n+1}^{m+1} - 2u_n^{m+1} + u_{n-1}^{m+1} \right] \\
+ \frac{\tau}{4h} \left[ f_\varepsilon (u_{n+1}^m) - f_\varepsilon (u_{n-1}^m) + f_\varepsilon (u_{n+1}^{m+1}) - f_\varepsilon (u_{n-1}^{m+1}) \right]. \notag
\end{multline}
We need to solve $N$ equations for $N$ unknowns $\{ u_n^{m+1} \}_{n=1}^N$ at each $0 \le m \le M-1$. Hence, we rearrange the discretization scheme to get the unknown variables on the left and the known variables on the right as
\begin{align}
    \begin{split}
&u_n^{m+1} + \frac{\tau}{h^2}u_n^{m+1} - \frac{\tau}{2h^2} \left( u_{n+1}^{m+1} + u_{n-1}^{m+1} \right) - \frac{\tau}{4h} \left[ f_\varepsilon (u_{n+1}^{m+1}) - f_\varepsilon (u_{n-1}^{m+1}) \right] = 
\\ 
&\qquad \qquad \qquad \qquad u_n^m + \frac{\tau}{2h^2} \left( u_{n+1}^m + u_{n-1}^{m} \right) - \frac{\tau}{h^2}u_n^m + \frac{\tau}{4h} \left[ f_\varepsilon (u_{n+1}^m) - f_\varepsilon (u_{n-1}^m) \right]. 
 \end{split} \label{Crank-Nicolson}
\end{align}
To simplify the expression, we use a predictor-corrector method (also known as Heun's method). The idea is to use the solution at an initial point, $u^m$, and to calculate an initial guess value of the next point $(u^*)^{m+1}$. Heun's method then improves this initial guess value using the trapezoidal rule to determine a better estimate of the next term $u^{m+1}$.

To represent the predictor-corrector method, we introduce two matrices:
\begin{equation*}
A_\pm =
\begin{bmatrix}
    1 \pm \frac{\tau}{h^2} & \mp \frac{\tau}{2h^2} & 0 & \cdots & 0 \\
    \mp \frac{\tau}{2h^2} & 1 \pm \frac{\tau}{h^2} & \mp \frac{\tau}{2h^2} & \cdots & \vdots \\
    \vdots & \mp \frac{\tau}{2h^2} & \ddots & \ddots\\
    \vdots & \vdots & \ddots & \ddots & \ddots \\
    \vdots & \vdots & \mp \frac{\tau}{2h^2} & 1 \pm \frac{\tau}{h^2} & \mp \frac{\tau}{2h^2} \\
    0 & 0 & \cdots & \mp \frac{\tau}{h^2} & 1 \pm \frac{\tau}{h^2}
\end{bmatrix},
\end{equation*}
where the elements of $A_{\pm}$ at the $(N,N-1)$ entry are doubled due to the Neumann condition $u_{N+1}^m=u_{N-1}^m$. We also represent the regularized terms in matrix vector notion,
\begin{equation*}
b(u^m) =
\begin{bmatrix}
    f_\varepsilon (u_2^m) \\
    f_\varepsilon (u_3^m) - f_\varepsilon (u_1^m) \\
    f_\varepsilon (u_4^m) - f_\varepsilon (u_2^m) \\
    \vdots \\
    f_\varepsilon(u_N^m) - f_\varepsilon (u_{N-2}^m) \\
    0
\end{bmatrix},
\end{equation*}
where we note that $f_\varepsilon (0)=0$ by construction of $f_\varepsilon$ and the Dirichlet condition, and $f_\varepsilon (u_{N+1}^m) - f_\varepsilon (u_{N-1}^m) = 0$ by the Neumann condition $u_{N+1}^m=u_{N-1}^m$. The correction step is computed 
from~(\ref{Crank-Nicolson}) by Euler's method as
\begin{equation}
\label{Euler-Predictor}
(u^*)^{m+1} = A_+^{-1} \left( A_- u^m + \frac{\tau}{2h}b(u^m) \right).
\end{equation}
The prediction step is computed from~(\ref{Crank-Nicolson}) by Heun's method as
\begin{equation}
\label{Heun-Corrector}
u^{m+1} = A_+^{-1} \left( A_- u^m + \frac{\tau}{4h}b(u^m) + \frac{\tau}{4h}b((u^*)^{m+1}) \right).
\end{equation}
We now extract the interface position $\xi(t_m)$ from $u^m$ at $t = t_m$ by finding the two adjacent grid points $x_n$ and $x_{n+1}$, where $u_n$ and $u_{n+1}$ are of opposite signs. By the straight line interpolation between $(x_n,u_n)$ and $(x_{n+1},u_{n+1})$, we obtain 
\begin{equation*}
u(x) = \left( \frac{u_{n+1} - u_n}{x_{n+1}-x_n} \right) (x - x_n) + u_n.
\end{equation*}
The value of $\xi(t_m)$ is obtained by finding the root of $u$ as 
\begin{equation}
\label{straight-line-interpolation}
\xi(t_m) = \frac{u_nx_{n+1} - u_{n+1}x_n}{u_n - u_{n+1}}.
\end{equation}

\subsection{Numerical simulations for shock waves}

We have performed iterations on the domain $[0,L]$ discretized with the grid size $h=0.01$. The time step was chosen to be $\tau = 0.0005$. Moreover, we took $\varepsilon = 10^{-16}$. 

\begin{figure}[htb!]
	\begin{subfigure}{0.49\textwidth}
		\includegraphics[scale=0.42]{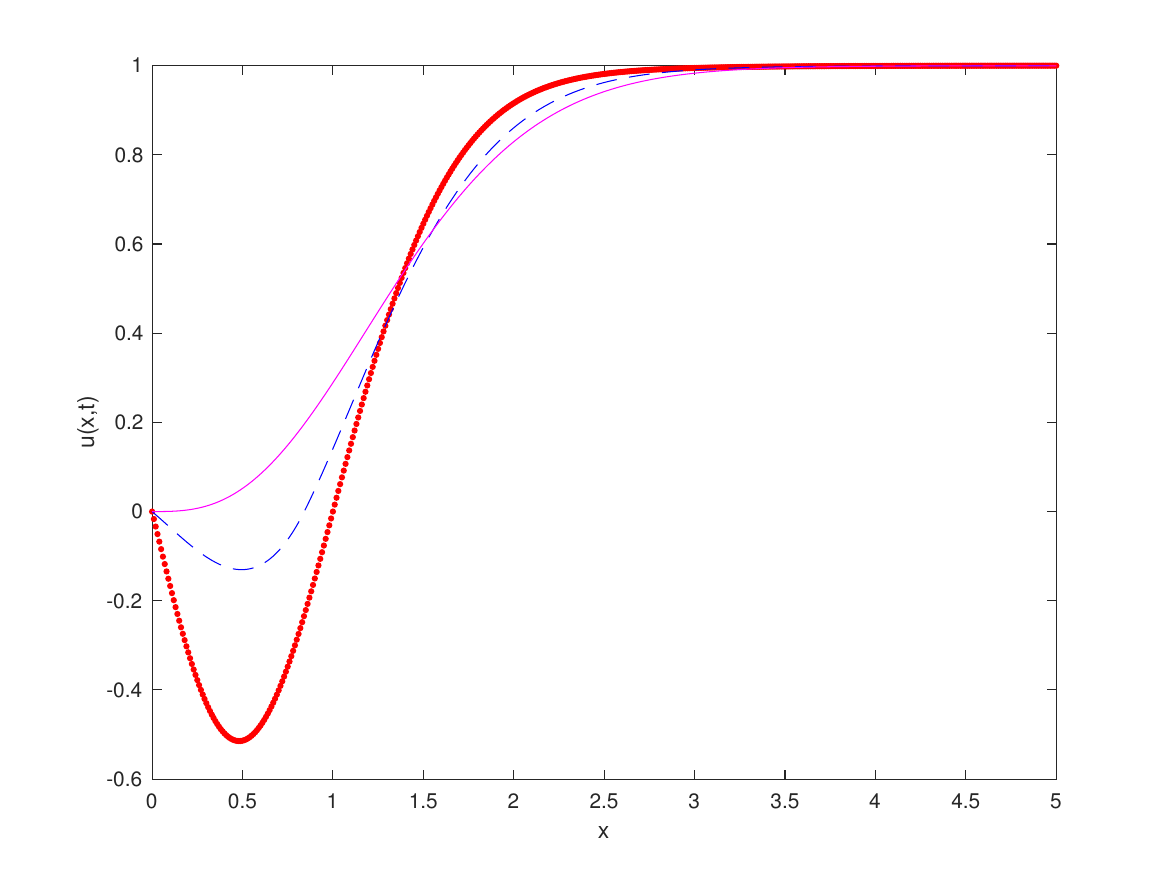}
	\end{subfigure}
	\hfill
	\begin{subfigure}{0.49\textwidth}
		\includegraphics[scale=0.42]{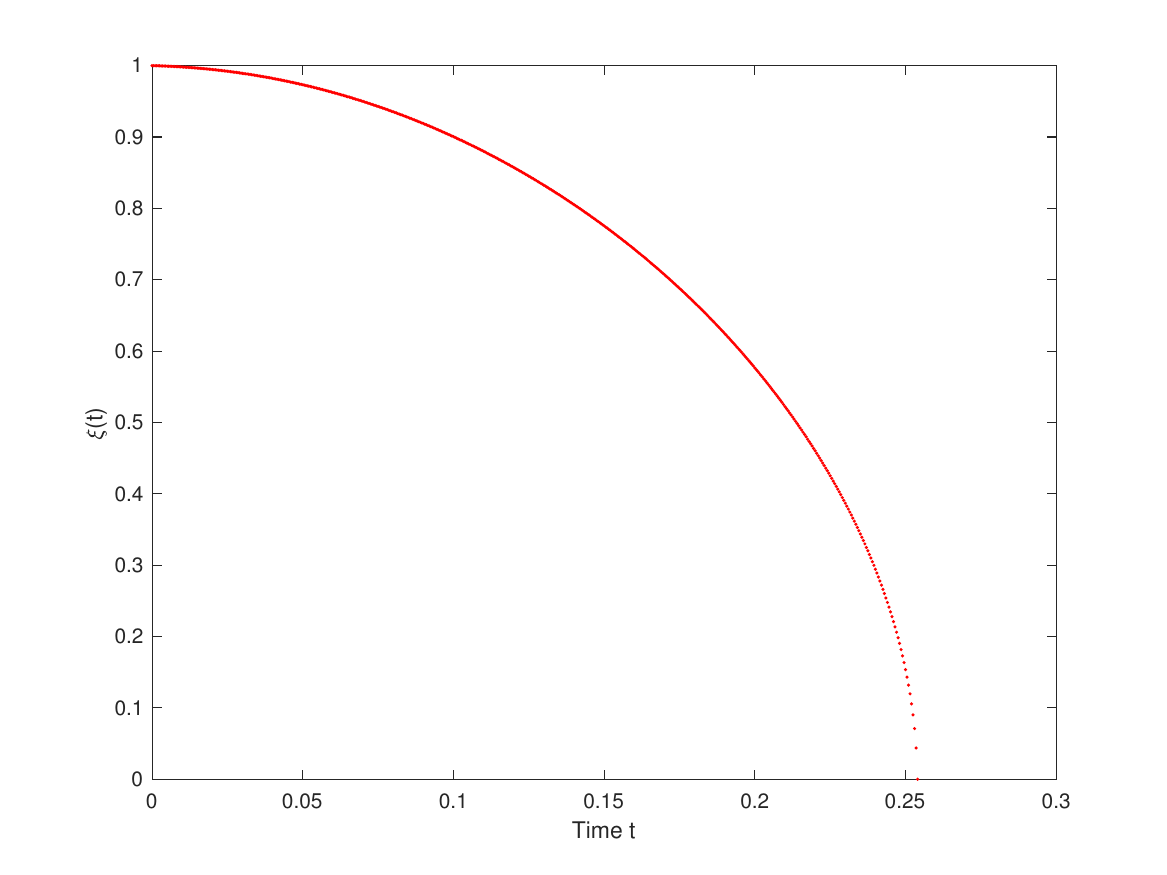}
	\end{subfigure}
	\caption{Evolution of~(\ref{Regularized-Burgers}) for the initial data~(\ref{phi_2}) with $\alpha = 1$. Left: $u(t,x)$ versus $x$ for times $t=0$, $t=0.126$, and $t=0.2535$. Right: evolution of $\xi(t)$ versus $t$.}
	\label{fig:phi_2}
\end{figure}

Figure~\ref{fig:phi_2} depicts the outcome of numerical simulations of the regularized approximation~(\ref{Regularized-Burgers}) of the modular Burgers' equation~\eqref{Burgers-modular-2} for the initial condition~(\ref{phi_2}) with $\alpha = 1$ for which we take $L = 5$. It is observed that $\xi(t)$ indeed goes to $0$ in finite time after which numerical computations can be continued. Yet, we stop them since we are only interested in the dynamics up to coalescence.  

We have also performed numerical simulations for the initial condition~(\ref{phi_2}) with $\alpha = 4$ shown in Figure~\ref{fig:parameter_4}. For these simulations, we have taken $L = 10$ to avoid the boundary effects from the Neumann boundary condition at $x = L$. With smaller values of $L$, the solution decays below $1$ at $x = L$ before the interface reaches $0$. Although the initial condition $u_0$ has larger negative parts on $[0,1]$, we observe that $\xi (t)$ still goes to $0$ in a finite time. Compared to Figure~\ref{fig:phi_2}, $\xi(t)$ is non-monotone as it first expands before it converges to $0$.

\begin{figure}[htb!]
\begin{subfigure}{0.49\textwidth}
    \includegraphics[scale=0.42]{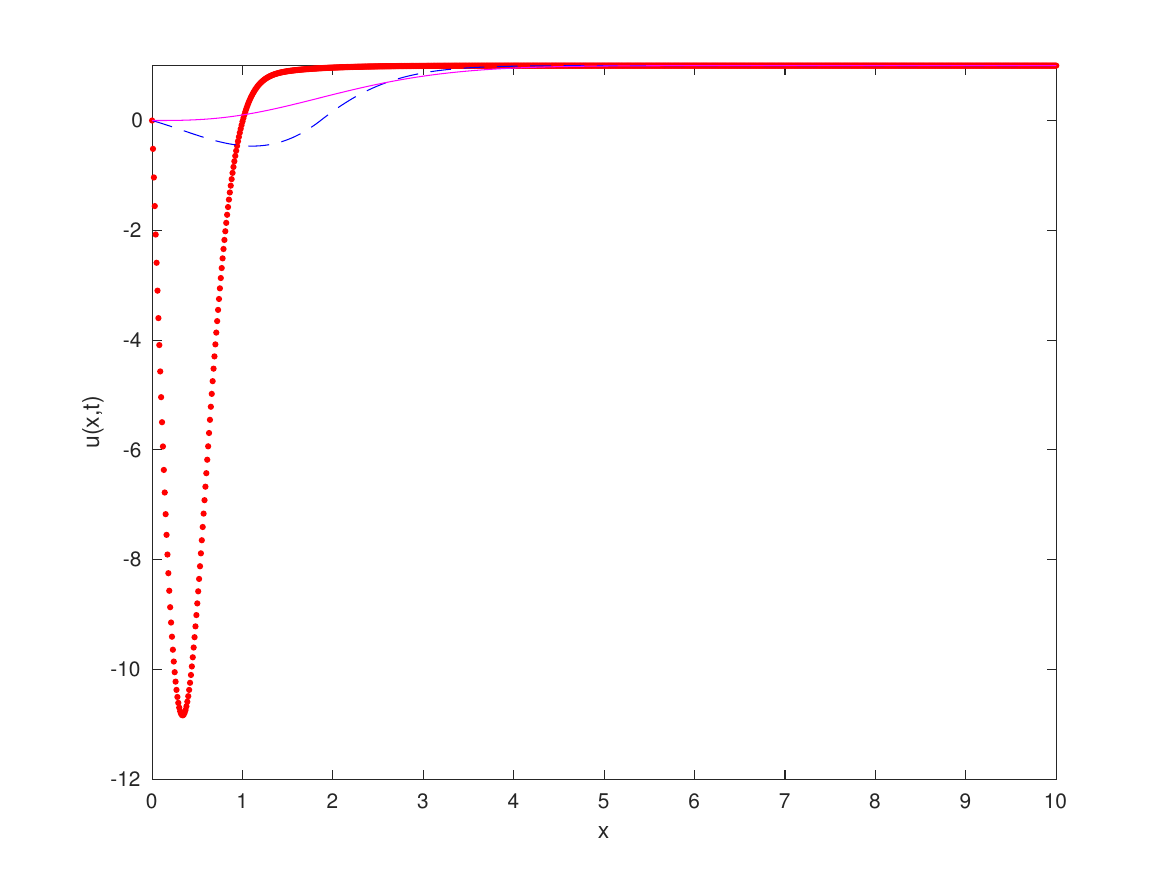}
\end{subfigure}
    \hfill
\begin{subfigure}{0.49\textwidth}
    \includegraphics[scale=0.42]{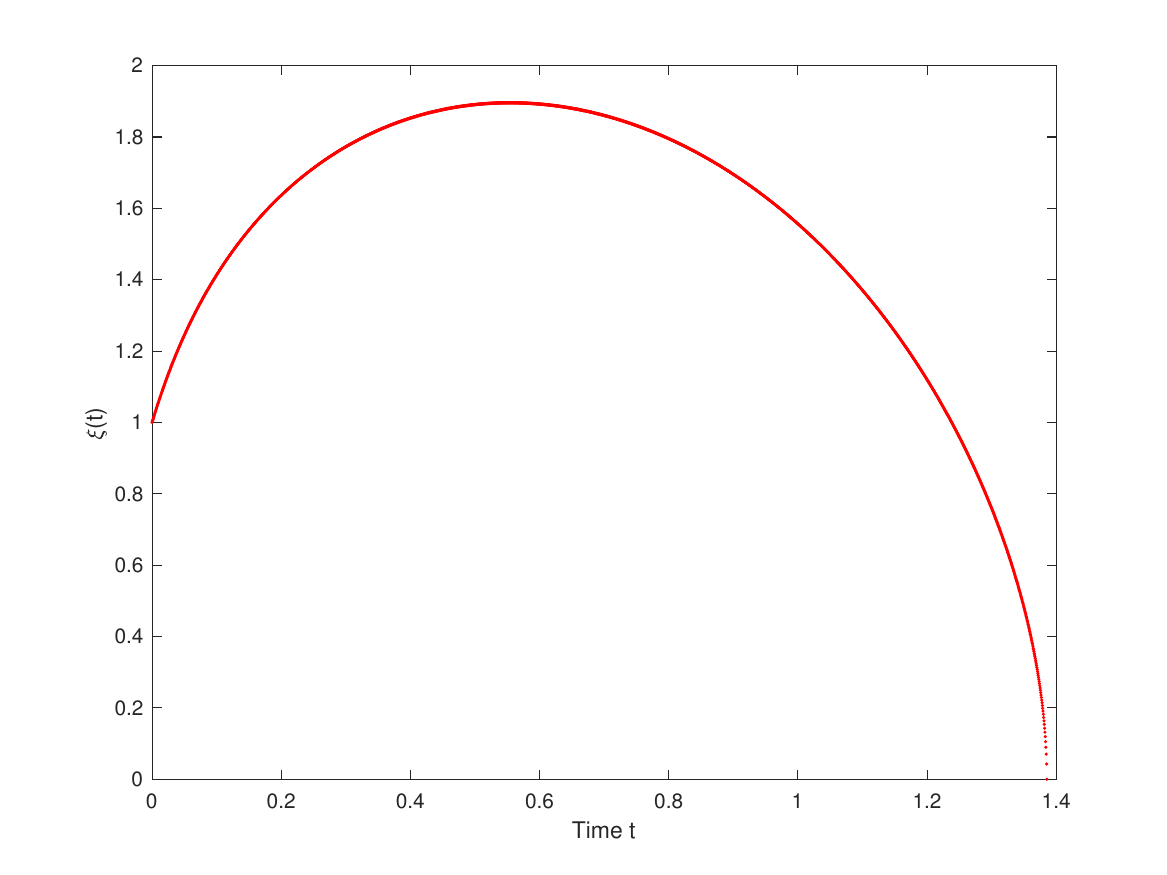}
\end{subfigure}
\caption{The same as in Figure~\ref{fig:phi_2} but with $\alpha =4$ and for times $t=0$, $t=0.692$, and $t=1.3285$. }
\label{fig:parameter_4}
\end{figure}

To confirm the scaling law~(\ref{law-pitchfork1}) of the interface coalescence,  we use linear regression in the log-log variable to approximate the associated power. That is, we consider
\begin{equation}
\label{scaling-law}
\log{\xi(t)} \text{ versus } c_1\log{(t_0 - t)} + c_2,
\end{equation}
where the coefficient $c_1$ represents the power of the scaling law. Note that the regression~(\ref{scaling-law}) depends on the unknown time $t_0$ of the interface coalescence. Thus, we first conduct computations for $t_0$ defined on a numerical grid and obtain the best fit by minimizing the approximation error.
 
 \begin{figure}[htb!]
 	\begin{subfigure}{0.49\textwidth}
 		\includegraphics[scale=0.42]{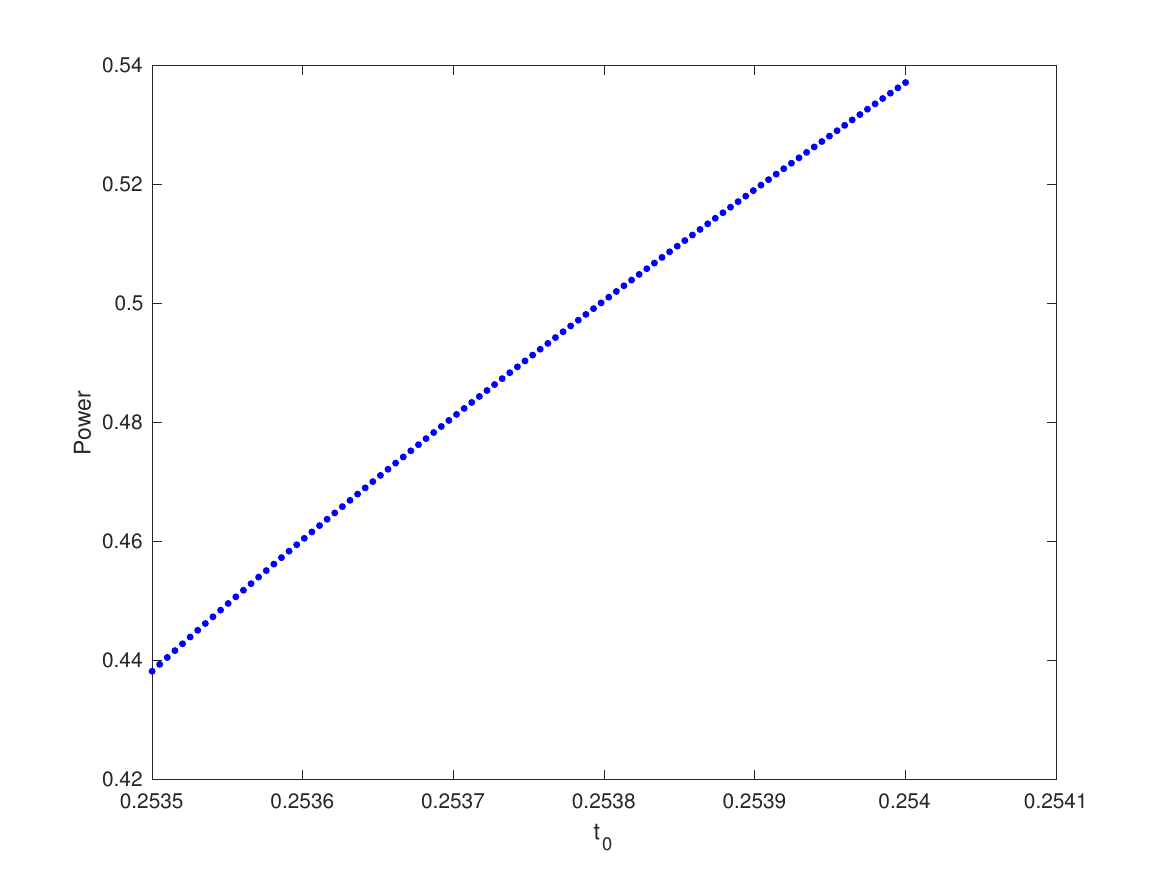}
 	\end{subfigure}
 	\hfill
 	\begin{subfigure}{0.49\textwidth}
 		\includegraphics[scale=0.42]{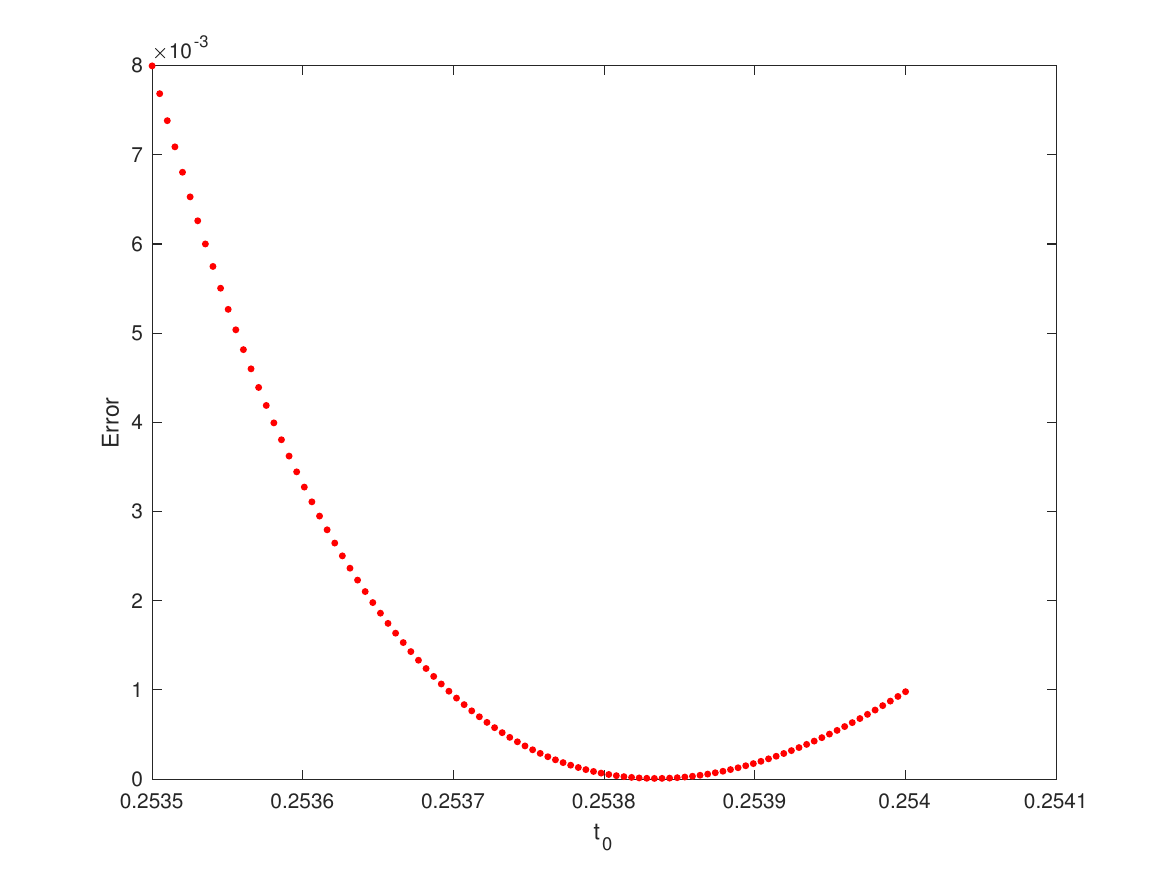}
 	\end{subfigure}
 	\caption{Left: power of the linear regression for Figure~\ref{fig:phi_2}. Right: approximation error versus $t_0$.}
 	\label{fig:xi_2}
 \end{figure}
 
 \begin{figure}[htb!]
 	\begin{subfigure}{0.49\textwidth}
 		\includegraphics[scale=0.42]{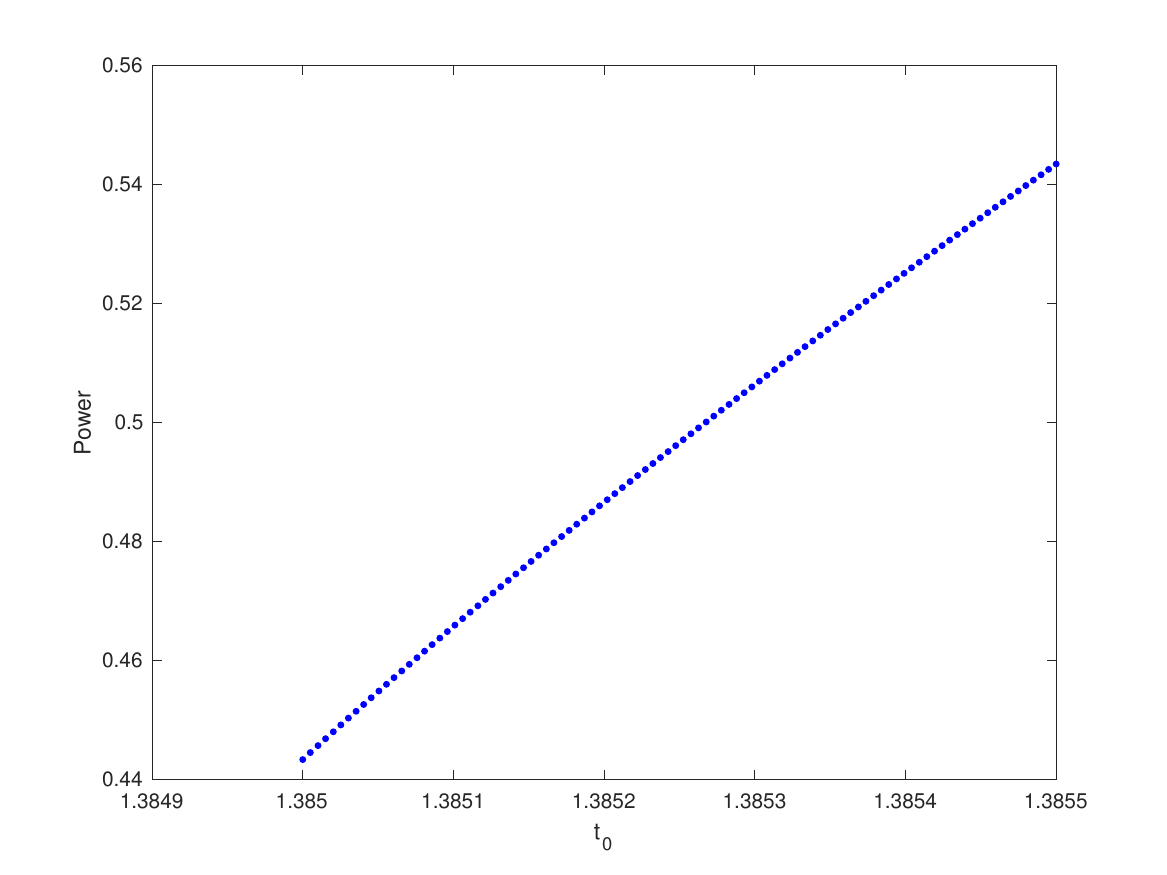}
 	\end{subfigure}
 	\hfill
 	\begin{subfigure}{0.49\textwidth}
 		\includegraphics[scale=0.42]{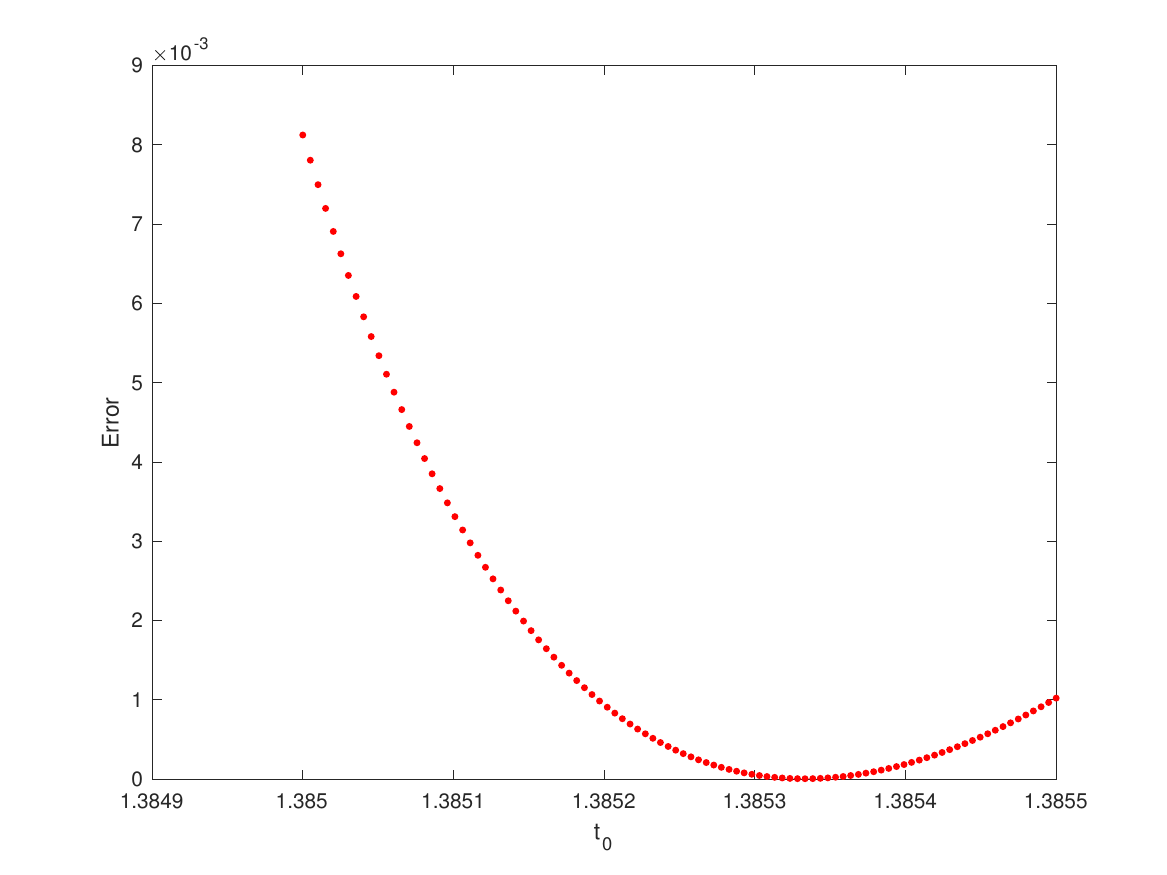}
 	\end{subfigure}
 	\caption{The same as in Figure~\ref{fig:xi_2} but for the data in Figure~\ref{fig:parameter_4}.}
 	\label{fig:parameter_4_xi}
 \end{figure}

The outcomes of these computations are depicted in Figures~\ref{fig:xi_2} and~\ref{fig:parameter_4_xi} for the approximations shown in Figures~\ref{fig:phi_2} and~\ref{fig:parameter_4}. The left panel shows the power versus $t_0$ and the right panel shows the corresponding approximation error versus $t_0$. The minimal error for $\alpha = 1$ is attained at $t_0=0.2538$ and this value of $t_0$ corresponds to $c_1=0.5068$. The minimal error for $\alpha = 4$ is attained at $t_0=1.3853$ and this value of $t_0$ corresponds to $c_1=0.5127$. In both cases, the power is close to the claimed value of $0.5$. We note that the time $t_0$ of extinction is larger for $\alpha =4$ than for $\alpha = 1$. 

\subsection{Numerical simulations for anti-shock waves}

We have also simulated~(\ref{Regularized-Burgers}) for the anti-shock wave initial condition~(\ref{-phi_2}). Figures~\ref{fig:-phi_2} and~\ref{fig:anti_parameter_4} depict the outcomes of numerical simulations for $\alpha = 1$ and $\alpha = 4$ respectively. For $\alpha = 1$, the interface position $\xi(t)$ goes to $0$ monotonically, similar to the computations in Figure~\ref{fig:phi_2}. For $\alpha = 4$, $\xi(t)$ first expands and then reduces towards $0$, similar 
to Figure~\ref{fig:parameter_4}.

\begin{figure}[htb!]
	\begin{subfigure}{0.49\textwidth}
		\includegraphics[scale=0.42]{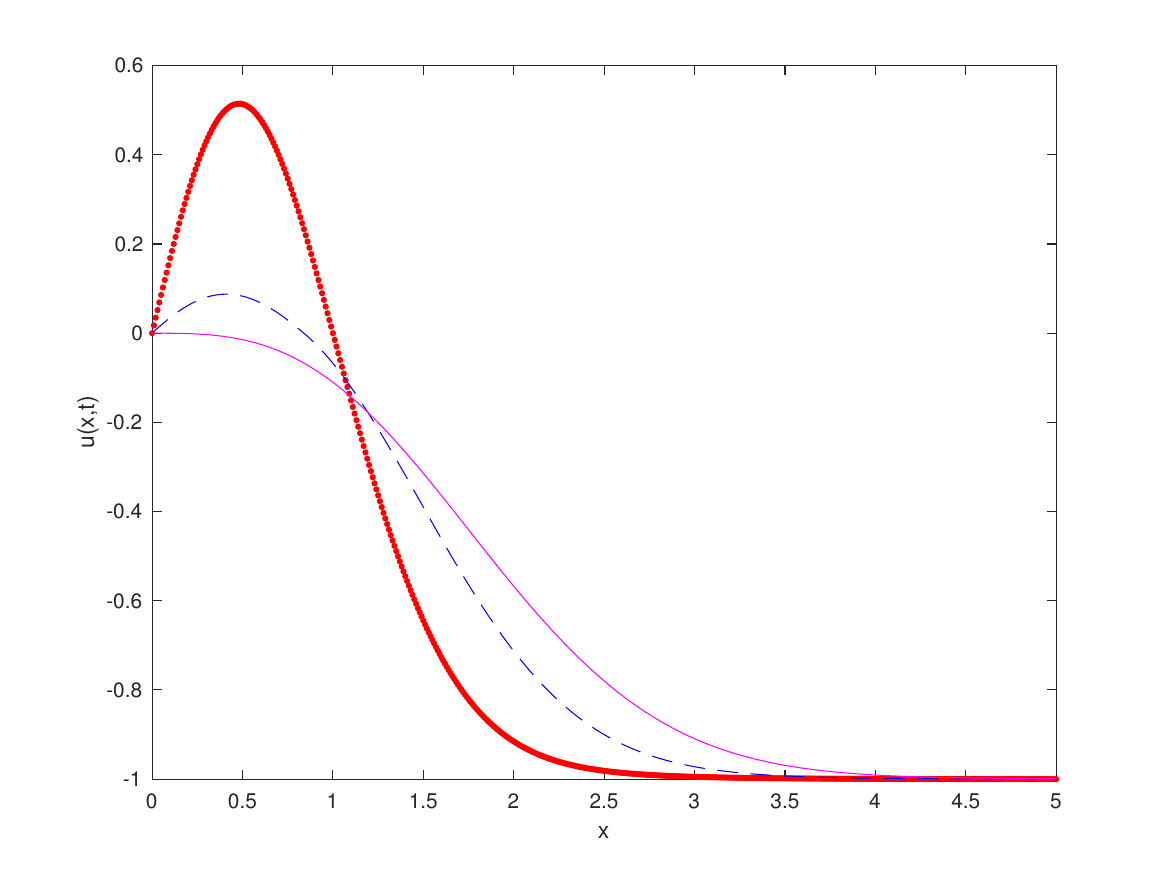}
	\end{subfigure}
	\hfill
	\begin{subfigure}{0.49\textwidth}
		\includegraphics[scale=0.42]{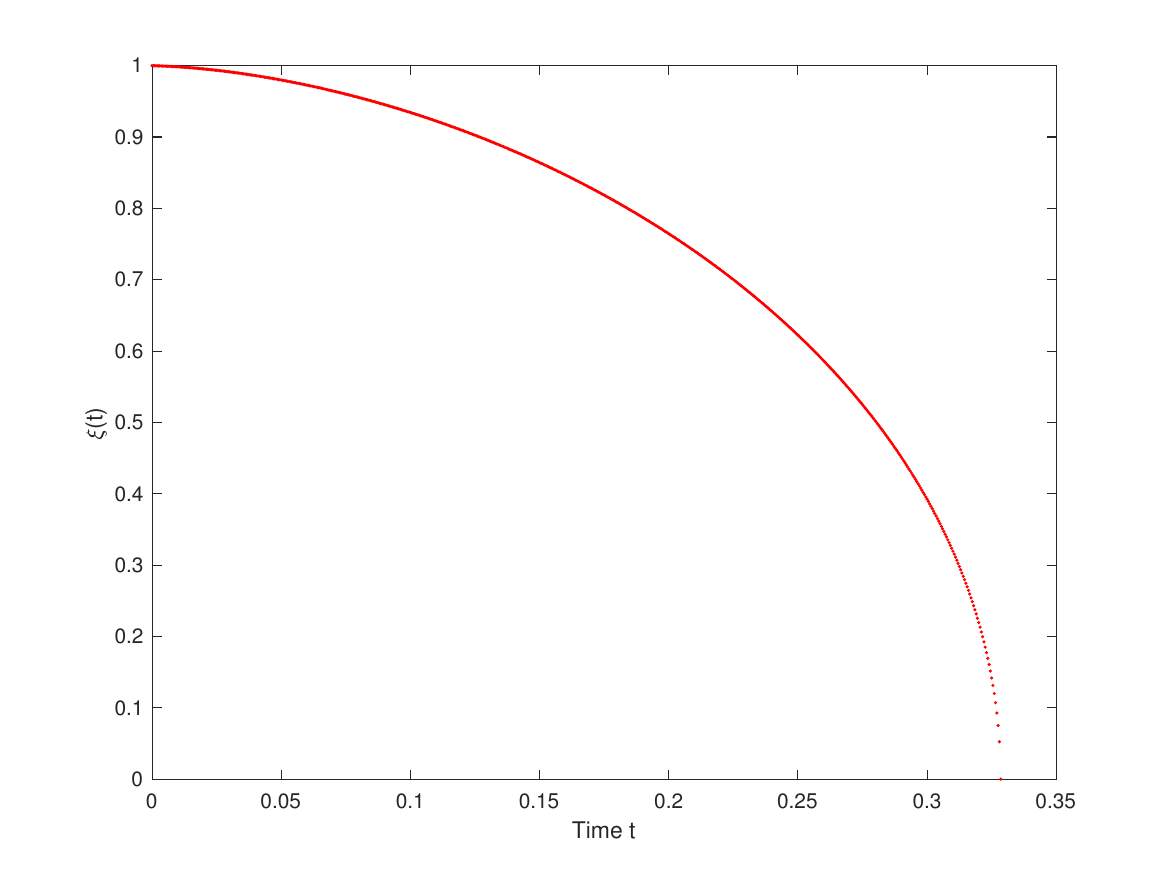}
	\end{subfigure}
	\caption{Evolution of~(\ref{Regularized-Burgers}) for the initial data~(\ref{-phi_2}) with $\alpha = 1$. Left: $u(t,x)$ versus $x$ for times $t=0$, $t=0.1635$, and $t=0.328$. Right: evolution of $\xi(t)$ versus $t$.}
	\label{fig:-phi_2}
\end{figure}
\begin{figure}[htb!]
	\begin{subfigure}{0.49\textwidth}
		\includegraphics[scale=0.42]{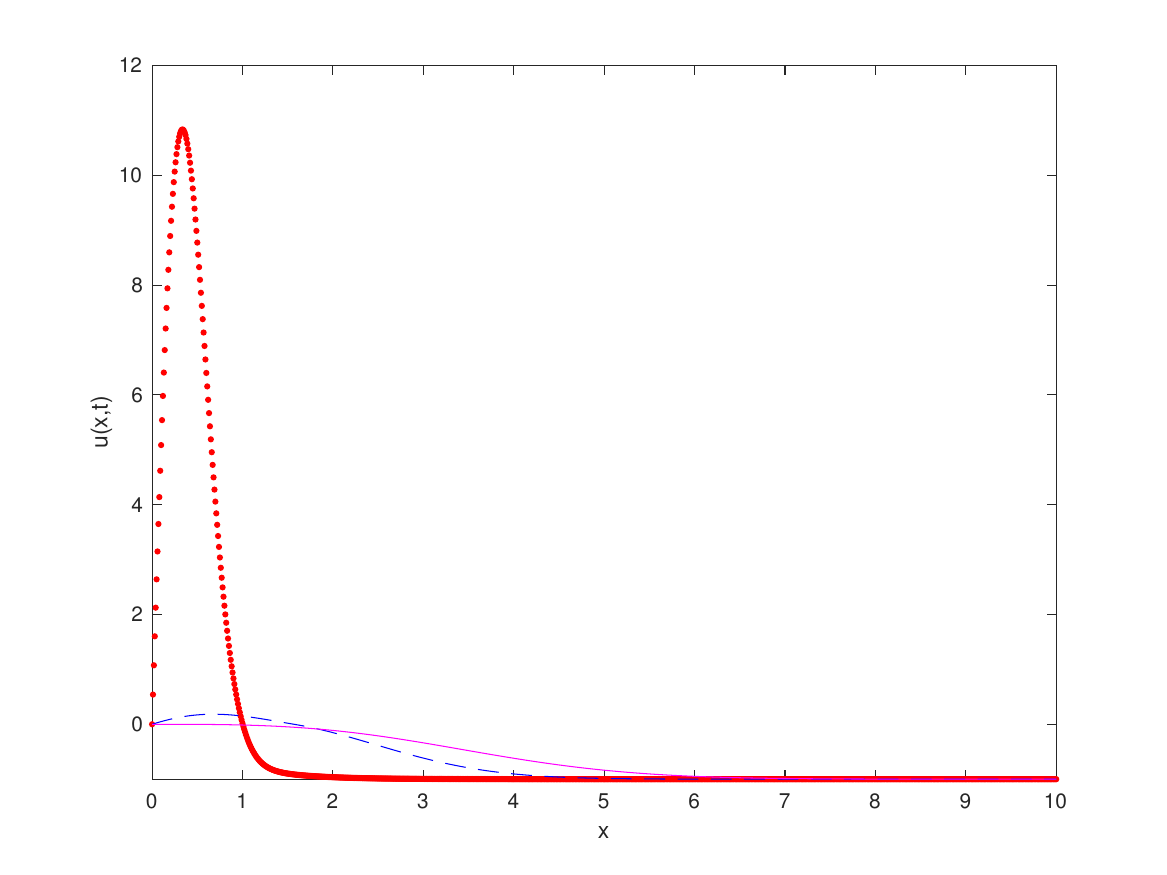}
	\end{subfigure}
	\hfill
	\begin{subfigure}{0.49\textwidth}
		\includegraphics[scale=0.42]{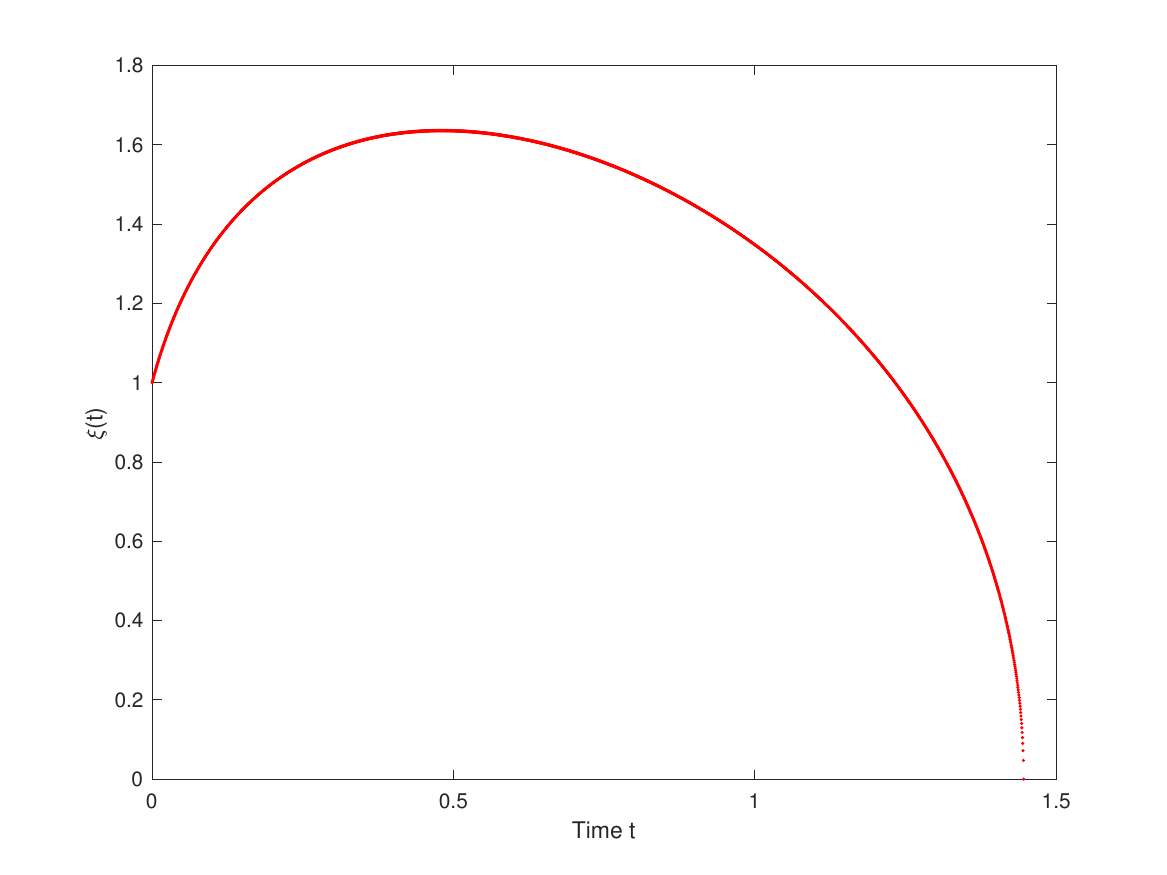}
	\end{subfigure}
	\caption{Tne same as in Figure~\ref{fig:-phi_2} but with $\alpha = 4$ and for times $t=0$, $t=0.7225$, and $t=1.4455$.}
	\label{fig:anti_parameter_4}
\end{figure}

Figures~\ref{fig:-xi_2} and~\ref{fig:anti_parameter_4_xi} show the approximate power of the scaling law and the approximation error versus $t_0$ for the simulations shown in Figures~\ref{fig:-phi_2} and~\ref{fig:anti_parameter_4}. The minimum error for $\alpha =1$ is attained at $t_0=0.3284$ and this value of $t_0$ corresponds to the power $c_1=0.4846$. The minimum error for $\alpha = 4$ is attained at $t_0=1.4459$ and this value of $t_0$ corresponds to $c_1=0.4884$. In both cases, the power is close to $0.5$ and thus, the scaling law~(\ref{law-pitchfork1}) is shown numerically to hold for anti-shock wave solutions considered here. However, the finite time of extinction is slightly larger for the anti-shock waves compared to that of the shock waves both for $\alpha = 1$ and $\alpha = 4$.

\begin{figure}[htb!]
\begin{subfigure}{0.49\textwidth}
    \includegraphics[scale=0.42]{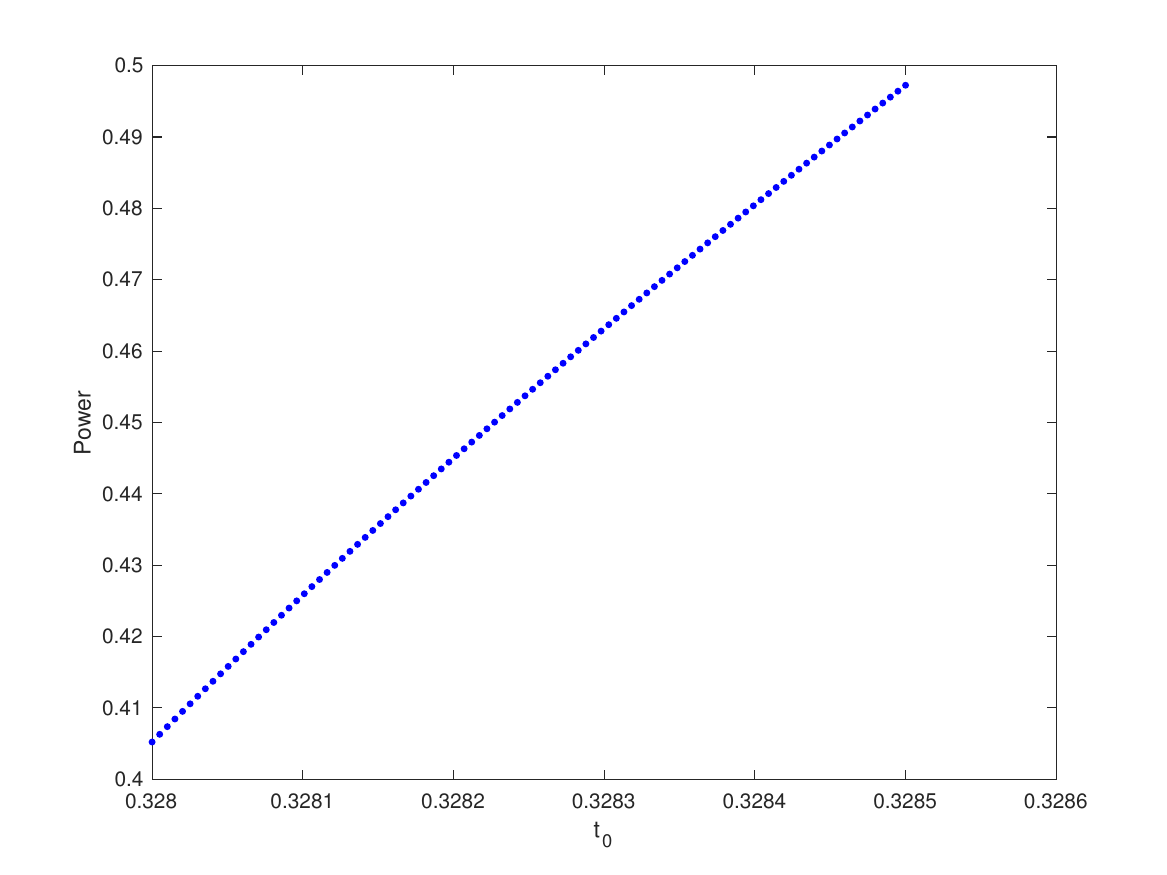}
\end{subfigure}
    \hfill
\begin{subfigure}{0.49\textwidth}
    \includegraphics[scale=0.42]{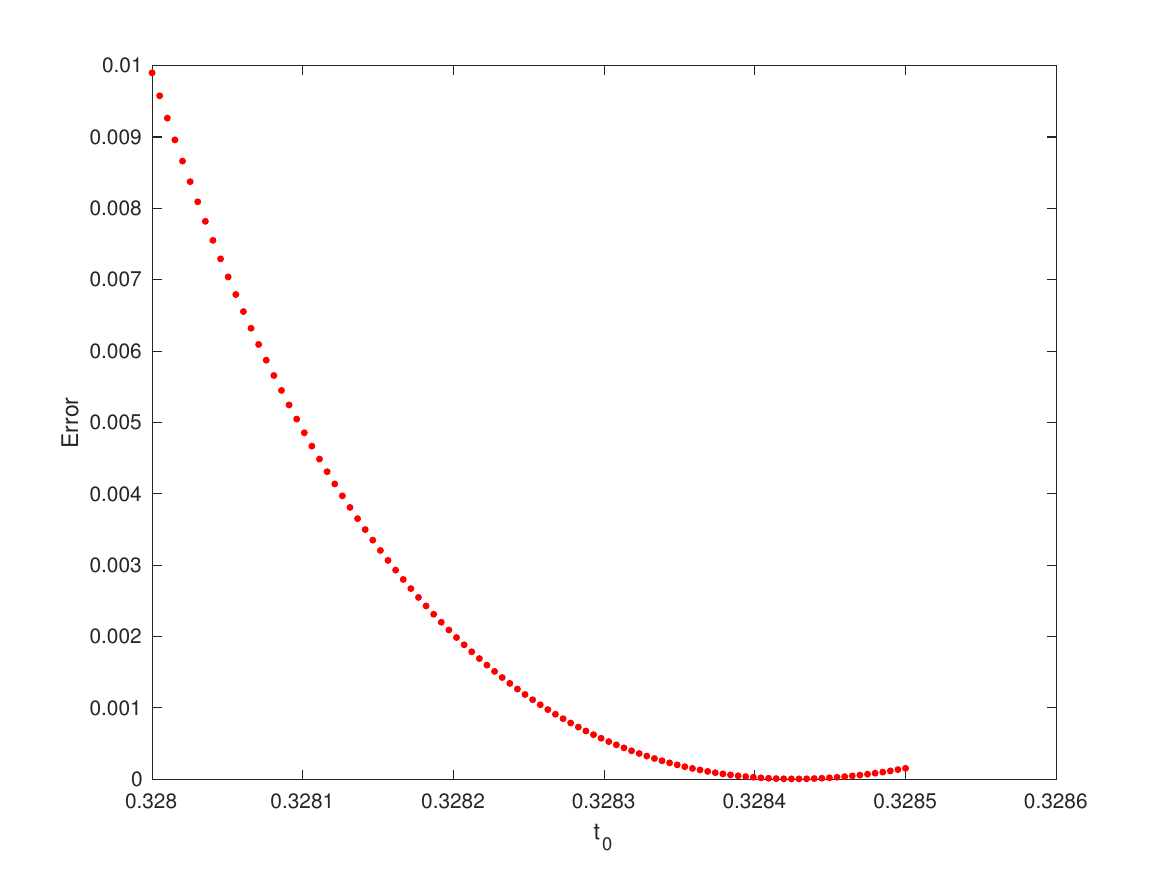}
\end{subfigure}
\caption{Left: power of the linear regression for Figure~\ref{fig:-phi_2}. Right: approximation error versus $t_0$.}
\label{fig:-xi_2}
\end{figure}

\begin{figure}[htb!]
\begin{subfigure}{0.49\textwidth}
    \includegraphics[scale=0.42]{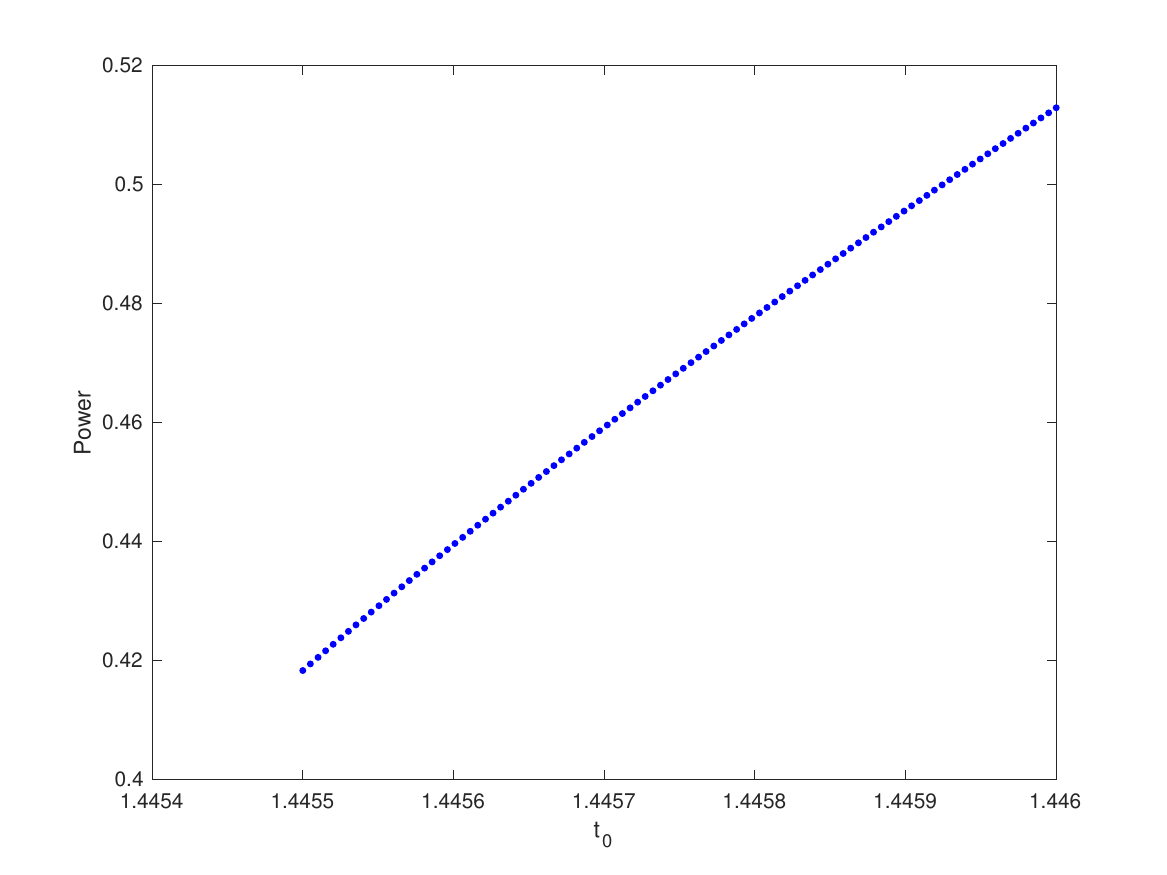}
\end{subfigure}
    \hfill
\begin{subfigure}{0.49\textwidth}
    \includegraphics[scale=0.42]{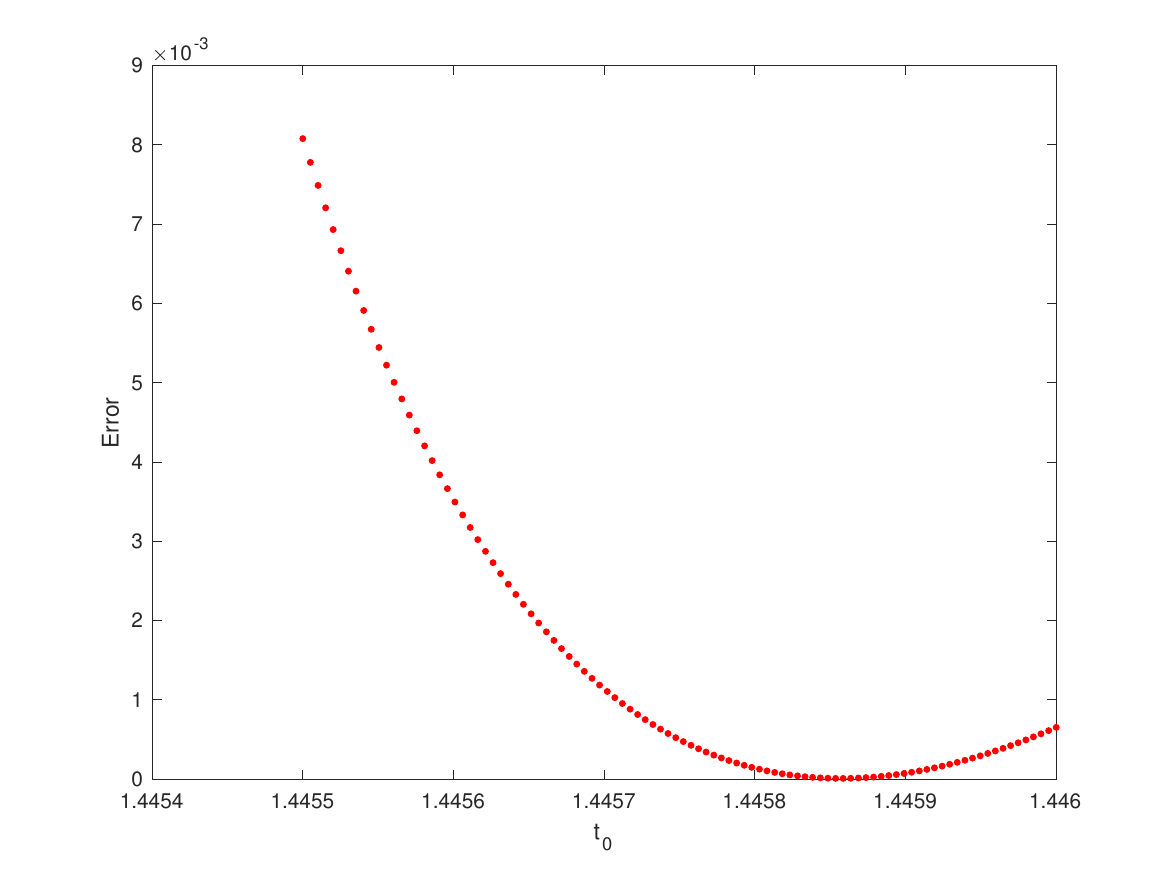}
\end{subfigure}
\caption{The same as in Figure~\ref{fig:-xi_2} but for the data in  Figure~\ref{fig:anti_parameter_4}.}
\label{fig:anti_parameter_4_xi}
\end{figure}

\appendix

\section{Proofs of well-posedness and approximation results} \label{app:A}

Here we provide proofs of the well-posedness and approximation results stated in~\S\ref{s:wellposedness}. Local well-posedness of the scalar viscous conservation law~\eqref{modBurgers}, as well as approximation by solutions of the regularized equation~\eqref{modBurgersApprox}, follows from standard theory for semilinear parabolic equations, cf.~\cite{LUN}, whereas global well-posedness relies on the comparison principle, cf.~\cite{PRWE,SERR}.

\begin{proof}[Proof of Lemma~\ref{p:globalsmooth}]
First, it is well-known that $\partial_x^2$ is a sectorial operator on $C_{\mathrm{ub}}(\R)$ with domain $C_{\mathrm{ub}}^2(\R)$ and there exists a constant $C>0$ such that 
\begin{align} \label{heatbounds} \|\partial_x^m \re^{\partial_x^2 t} u\|_\infty \leq Ct^{-\frac{m}{2}} \|u\|_\infty,\end{align} 
for $m = 0,1,2$, $t > 0$ and $u \in C_{\mathrm{ub}}(\R)$. Second, the map $N \colon C_{\mathrm{ub}}^1(\R) \to C_{\mathrm{ub}}(\R)$ given by $N(u) = f'(u) u_x$ is locally Lipschitz continuous since $f$ is smooth. Third, $C_{\mathrm{ub}}^1(\R)$ is an intermediate space of class $J_{1/2}$ between $C_{\mathrm{ub}}(\R)$ and $C_{\mathrm{ub}}^2(\R)$. Hence, it follows from standard analytic semigroup theory, cf.~\cite{LUN}, that there exist a maximal time $T \in (0,\infty]$ and a unique classical solution
\begin{align*}
u \in C\big([0,T),C_{\mathrm{ub}}^1(\R)\big) \cap C\big((0,T),C_{\mathrm{ub}}^2(\R)\big) \cap C^1\big((0,T),C_{\mathrm{ub}}(\R)\big),
\end{align*}
of~\eqref{modBurgers} with initial condition $u(0,\cdot) = u_0 \in C^1_{\mathrm{ub}}(\R)$. Moreover, if we have $T < \infty$, then it holds $\limsup_{t \to T^-} \|u(t,\cdot)\|_{W^{1,\infty}} = \infty$. A standard bootstrapping argument, using the fact that $f \in C^\infty(\R)$, then yields $\partial_t^k u(t,\cdot) \in C_{\mathrm{ub}}^l(\R)$ for any $k,l \in \mathbb{N}_0$ and $t \in [0,T)$ implying $u \in C^\infty\big((0,T) \times \R, \R\!\big)$.

It is well-known~\cite{PRWE,SERR} that the scalar conservation law~\eqref{modBurgers} obeys a comparison principle yielding $m_0 \leq u(t,\cdot) \leq M_0$ for all $t \in [0,T)$ upon comparison with the constant solutions $u \equiv m_0$ and $u \equiv M_0$ of~\eqref{modBurgers}. Differentiating the mild formulation of~\eqref{modBurgers}, we obtain 
\begin{align} 
\notag
u_x(t,\cdot) &= \re^{\partial_x^2 t} u_0' + \int_0^{t(1-\delta)} \partial_x^2 \re^{\partial_x^2(t-s)} f(u(s,\cdot)) \de s \\
\label{mildform6}
& \quad + \int_{t(1-\delta)}^t \partial_x \re^{\partial_x^2(t-s)} f'(u(s,\cdot)) u_x(s,\cdot) \de s ,
\end{align}
for $t \in [0,T)$, where $\delta \in (0,1)$ will be fixed a posteriori. Let $R \geq 1$ be such that 
$$
\sup\{|f(v)| + |f'(v)| : v \in [m_0,M_0]\} \leq R.
$$ 
Fix some $\tau \in [0,T)$. Taking norms in~\eqref{mildform6}, while using~\eqref{heatbounds} and the fact that $m_0 \leq u(t,\cdot) \leq M_0$, we establish
\begin{align*}
\|u_x(t,\cdot)\|_\infty &\leq C \|u_0\|_{W^{1,\infty}} + \int_0^{t(1-\delta)} \frac{CR}{t-s} \de s + \int_{t(1-\delta)}^t \frac{CR \sup\{\|u_x(s,\cdot)\|_\infty : s \in [0,\tau]\}}{\sqrt{t-s}} \de s \\
&\leq C\left(\|u_0\|_{W^{1,\infty}} + R |\log(\delta)| + 2 R \sqrt{\delta t} \sup\{\|u_x(s,\cdot)\|_\infty : s \in [0,\tau]\}\right)
\end{align*}
for all $t \in [0,\tau]$. Thus, setting $\delta = \frac{1}{16 C^2 R^2 \max\{1,\tau\}} \in (0,1)$ and taking suprema in the latter inequality, we arrive at
\begin{align*}
 \sup\{\|u_x(s,\cdot)\|_\infty : s \in [0,\tau]\} \leq 2C\left(\|u_0\|_{W^{1,\infty}} + R \log\left(16 C^2 R^2 \max\{1,\tau\}\right)\right),
\end{align*}
for all $\tau \in [0,T)$. We conclude that $\limsup_{t \to T^-} \|u(t,\cdot)\|_{W^{1,\infty}} = \infty$ cannot occur, implying that $T = \infty$ and the classical solution is global.
\end{proof}

\begin{proof}[Proof of Lemma~\ref{p:globalcub}]
Recall that $\partial_x^2$ is a sectorial operator on $C_{\mathrm{ub}}(\R)$ satisfying~\eqref{heatbounds}. In addition, the flux function $f \colon C_{\mathrm{ub}}(\R) \to C_{\mathrm{ub}}(\R)$ is locally Lipschitz continuous. Hence, by a standard fixed point argument as in the proofs of~\cite[Theorem~7.1.2 and Proposition~7.2.1]{LUN} there exist a maximal time $T \in (0,\infty]$ and a unique solution $u \in C\big([0,T),C_{\mathrm{ub}}(\R)\big)$ of~\eqref{mildform2}. Moreover, if $T < \infty$, then it holds $\limsup_{t \to T^-} \|u(t,\cdot)\|_\infty = \infty$. 

Let $\tilde{f} \in C^\infty(\R)$ be a function satisfying
\begin{align*}
 \sup\left\{\big|f(v)-\tilde{f}(v)\big| : v \in \left[-m_0,M_0\right]\right\} < \delta,
\end{align*}
for some $\delta > 0$. By Lemma~\ref{p:globalsmooth} there exists a unique global classical solution
$$
\tilde u \in C\big([0,\infty),C_{\mathrm{ub}}^1(\R)\big) \cap C\big((0,\infty),C_{\mathrm{ub}}^2(\R)\big) \cap C^1\big((0,\infty),C_{\mathrm{ub}}(\R)\big)
$$
of the integral equation 
\begin{align} \label{mildform4}
\tilde u(t,\cdot) = \re^{\partial_x^2 t} u_0 + \int_0^t \partial_x \re^{\partial_x^2(t-s)} \tilde f(\tilde u(s,\cdot)) \de s.
\end{align}
satisfying $m_0 \leq \tilde u(t,\cdot) \leq M_0$ for all $t \geq 0$. 
From~(\ref{mildform2}) and~(\ref{mildform4}), we obtain
\begin{align} \label{mildform3}
u(t,\cdot) - \tilde u(t,\cdot) = \int_0^t \partial_x \re^{\partial_x^2(t-s)} \left(f(u(s,\cdot)) - f(\tilde u(s,\cdot)) + f(\tilde u(s,\cdot)) - \tilde f(\tilde u(s,\cdot))\right) \de s,
\end{align}
for all $t \in [0,T)$. Denote by $L > 0$ the Lipschitz constant of $f$ on $[m_0-1,M_0+1]$. Taking norms in~\eqref{mildform3} we arrive at
\begin{align*}
\|u(t,\cdot) - \tilde u(t,\cdot)\|_\infty \leq C \int_0^t \frac{L\|u(s,\cdot) - \tilde u(s,\cdot)\|_\infty + \delta}{\sqrt{t-s}} \de s,
\end{align*}
for any $t \in [0,T)$ with $\sup\{\|u(s,\cdot) - \tilde{u}(s,\cdot)\|_\infty : s \in [0,t]\} \leq 1$. Hence, Gr\"onwall's Lemma~\cite[Lemma~7.0.3]{LUN} yields a constant $M > 0$, depending only on $C$ and $L$, such that
\begin{align} \label{deltaest}
\|u(t,\cdot) - \tilde u(t,\cdot)\|_\infty \leq M \delta \sqrt{t},
\end{align}
for all $t \in [0,T)$ with $\sup\{\|u(s,\cdot) - \tilde{u}(s,\cdot)\|_\infty : s \in [0,t]\} \leq 1$. 

We argue by contradiction and assume $T < \infty$. Take 
$$
0 < \delta \leq \frac{1}{2M\sqrt{T}}.
$$ 
If 
$$
\sup\{\|u(s,\cdot) - \tilde{u}(s,\cdot)\|_\infty : s \in [0,T)\} > 1,
$$ 
then by continuity, there must exist $t \in [0,T)$ with 
$$
\sup\{\|u(s,\cdot) - \tilde{u}(s,\cdot)\|_\infty : s \in [0,t]\} = 1.
$$ 
However,~\eqref{deltaest} then implies 
$$
\|u(s,\cdot) - \tilde u(s,\cdot)\|_\infty < \frac12
$$ 
for any $s \in [0,t]$, which yields a contradiction. Hence, we have 
$$
\sup\{\|u(s,\cdot) - \tilde{u}(s,\cdot)\|_\infty : s \in [0,T)\} \leq 1
$$ 
and~\eqref{deltaest} is satisfied for all $t \in [0,T)$. So, we must have $T = \infty$ and $u(t,\cdot)$ is global. 

Since it holds $m_0 \leq \tilde u(t,\cdot) \leq M_0$ for all $t \geq 0$ and, in addition, $\delta > 0$ can be chosen arbitrarily small, it follows $m_0 \leq u(t,\cdot) \leq M_0$ for all $t \geq 0$ by~\eqref{deltaest}, which concludes the proof of~(\ref{approxub}).
\end{proof}

\begin{proof}[Proof of Lemma~\ref{p:approxcub1}]
First, Lemma~\ref{p:globalcub} implies that $m_0 \leq u(t,\cdot) \leq M_0$ for all $t \geq 0$. Second, there exists by Lemma~\ref{p:globalcub} constants $\smash{\widetilde{M}}, \tilde{\delta}_0 > 0$ such that if we take $\delta \in (0,\tilde{\delta}_0)$, then there exists a unique global classical solution~\eqref{regularity2} of~\eqref{modBurgersApprox} satisfying $m_0 \leq \tilde{u}(t,\cdot) \leq M_0$ and $\|u(t,\cdot) - \tilde{u}(t,\cdot)\|_\infty \leq \smash{\widetilde{M}} \delta \sqrt{t}$ for all $t \geq 0$. Thus, $\tilde u(t,\cdot)$ solves the mild formulation~\eqref{mildform4}. Subtracting~\eqref{mildform4} from~\eqref{mildform2} and differentiating we obtain
\begin{align} \label{mildform5}
\begin{split}
u_x(t,\cdot) - \tilde u_x(t,\cdot) &= \int_0^{t(1-\delta)} \partial_x^2 \re^{\partial_x^2(t-s)} \left(f(u(s,\cdot))  - f(\tilde u(s,\cdot))\right) \de s\\
&\qquad + \, \int_0^{t(1-\delta)} \partial_x^2 \re^{\partial_x^2(t-s)} \left(f(\tilde u(s,\cdot))  -  \tilde f(\tilde u(s,\cdot))\right) \de s\\
&\qquad + \, \int_{t(1-\delta)}^t \partial_x \re^{\partial_x^2(t-s)} \left(f'(u(s,\cdot)) - \tilde{f}'(\tilde u(s,\cdot)) \right)u_x(s,\cdot) \de s \\
&\qquad + \, \int_{t(1-\delta)}^t \partial_x \re^{\partial_x^2(t-s)} \tilde{f}'(\tilde u(s,\cdot))\left(u_x(s,\cdot) - \tilde{u}_x(s,\cdot)\right) \de s,
\end{split}
\end{align}
for all $t \geq 0$. Denote by $L > 0$ the Lipschitz constant of $f$ on $[m_0,M_0]$, and set $K = \sup\{\|u_x(s,\cdot)\|_\infty : 0 \leq s \leq \tau\}$ and $R_1 = \sup\{|f'(v)| : v \in [m_0,M_0]\}$. Thus, taking norms in~\eqref{mildform5}, while using~\eqref{heatbounds}, we arrive at
\begin{align*}
\left\|u_x(t,\cdot) - \tilde u_x(t,\cdot)\right\|_\infty &\leq C\int_{t(1-\delta)}^t \frac{K\left(R + R_1\right)}{\sqrt{t-s}} \de s + C\int_0^t \frac{R\|u_x(s,\cdot) - \tilde u_x(s,\cdot)\|_\infty}{\sqrt{t-s}} \de s \\
&\qquad + \, C  \int_0^{t(1-\delta)} \frac{\delta\left(1 + L\widetilde{M} \sqrt{t}\right)}{t-s} \de s,
\end{align*}
for all $t \in [0,\tau]$. Hence, Gr\"onwall's Lemma~\cite[Lemma~7.0.3]{LUN} yields a constant $M > 0$, independent of $\delta$, such that
\begin{align*}
\|u_x(t,\cdot) - \tilde u_x(t,\cdot)\|_\infty \leq M \sqrt{\delta t},
\end{align*}
for all $t \geq 0$. Thus, taking $\delta_0 < \min\{\tilde{\delta}_0,\varepsilon^2/(M^2 \tau),\varepsilon/(\widetilde{M} \sqrt{\tau})\}$ we establish~\eqref{approxub2}. 
\end{proof}

\begin{proof}[Proof of Lemma~\ref{p:globalcub1}]
We switch to the co-moving frame $\xi = x-ct$, in which equation~\eqref{modBurgers} reads
\begin{align}
w_t = w_{\xi\xi} + cw_\xi + f(w)_\xi. \label{modBurgerscom} 
\end{align}
If $w(t,\cdot)$ is a mild solution of~\eqref{modBurgerscom} with initial condition $w(0,\cdot) = u_0$, then the difference $z = w - \phi$ is a mild solution of
\begin{align}
z_t = \left(z_\xi + cz + f(z + \phi(\xi)) - f(\phi(\xi))\right)_\xi  \label{perteq}
\end{align}
and has initial condition $z_0 = u_0 - \phi \in C_{\mathrm{ub}}^1(\R) \cap L^1(\R)$. The integrated version of equation~\eqref{perteq} reads
\begin{align}
\begin{split}
v_t &= v_{\xi\xi} + cv_\xi + f\left(v_\xi + \phi(\xi)\right) - f(\phi(\xi)), 
\end{split}
\label{pertCauchyint}
\end{align}
where the relevant solution has initial condition $v_0 \in C_{\mathrm{ub}}^2(\R)$ given by
\begin{align*} v_0(\xi) = \int_{-\infty}^\xi z_0(y) \de y. \end{align*}
First, the nonlinearity $N \colon C_{\mathrm{ub}}^1(\R) \to C_{\mathrm{ub}}(\R)$ given by $N(v) = cv_\xi + f(v_\xi + \phi) - f(\phi)$ is well-defined and locally Lipschitz continuous. Second, $\partial_\xi^2$ is a sectorial operator on $C_{\mathrm{ub}}(\R)$ with dense domain $C_{\mathrm{ub}}^2(\R)$. Third, $C_{\mathrm{ub}}^1(\R)$ is an intermediate space of class $J_{1/2}$ between $C_{\mathrm{ub}}(\R)$ and $C_{\mathrm{ub}}^2(\R)$. Therefore, standard analytic semigroup theory, cf.~\cite[Theorem~7.1.2 and Propositions~7.1.10 and~7.2.1]{LUN}, yields a maximal time $T \in (0,\infty]$ and a solution $v \in C\big([0,T),C_{\mathrm{ub}}^2(\R)\big)$ of
\begin{align}
v(t,\cdot) &= \re^{\partial_\xi^2 t} v_0 + \int_0^t \re^{\partial_\xi^2(t-s)}  \left(cv_\xi(s,\cdot) + f(v_\xi(s,\cdot) + \phi) - f(\phi)\right)\de s. \label{mildform0}
\end{align}
Moreover, if $T < \infty$, then we must have $\limsup_{t \to T^-} \|v(t,\cdot)\|_{W^{1,\infty}} = \infty$. Differentiating~\eqref{mildform0} with respect to $\xi$ and setting $z = v_\xi$, we obtain 
\begin{align}
z(t,\cdot) &= \re^{\partial_\xi^2 t} z_0 + \int_0^t \partial_\xi \re^{\partial_\xi^2(t-s)}  \left(cz(s,\cdot) + f(z(s,\cdot) + \phi) - f(\phi)\right)\de s. \label{mildform}
\end{align}
Hence, $z \in C\big([0,T),C_{\mathrm{ub}}^1(\R)\big)$ is a mild solution of~\eqref{perteq} with initial condition $z_0$. Thus, we have 
$$
v_\xi(t,\xi) = z(t,\xi) = w(t,\xi) - \phi = u(t,\xi + ct) - \phi(\xi), 
$$
where $u \in C\big([0,\infty),C_{\mathrm{ub}}(\R)\big)$ is the global mild solution of~\eqref{modBurgerscom}, established in Lemma~\ref{p:globalcub}, satisfying $\|u(t,\cdot)\|_\infty \leq \|u_0\|_\infty$ for $t \geq 0$. So, it holds 
$$
\|v_\xi(t,\cdot)\|_\infty = \|z(t,\cdot)\|_\infty \leq \|u_0\|_\infty + \|\phi\|_\infty
$$ 
for all $t \geq 0$. Taking norms in~\eqref{mildform0} and using~\eqref{heatbounds} we arrive at
\begin{align} \label{blowupineq}
\|v(t,\cdot)\|_\infty \leq C\left(\|v_0\|_\infty + t \sup_{0 \leq s \leq t} \|c z(s,\cdot) + f(z(s,\cdot) + \phi) - f(\phi)\|_\infty\right).
\end{align}  
Clearly, the right-hand side of~\eqref{blowupineq} does not blow up as $t \to T^-$ yielding $T = \infty$. Thus, we have obtained a global solution $u \in C\big([0,\infty),C_{\mathrm{ub}}^1(\R)\big)$ of~\eqref{mildform2}. 

Finally, we establish $L^1$-integrability of $u(t,\cdot) - \phi$ for all $t \geq 0$. Since $\phi$ is bounded and $f$ is locally Lipschitz continuous, we observe that the nonlinearity $G \colon L^1(\R) \cap C_{\mathrm{ub}}(\R) \to L^1(\R) \cap C_{\mathrm{ub}}(\R)$ given by $G(z) = cz + f(z + \phi) - f(\phi)$ is well-defined and locally Lipschitz continuous. On the other hand, $\partial_\xi^2$ is a sectorial operator on $C_{\mathrm{ub}}(\R) \cap L^1(\R)$ and there exists a constant $C > 0$ such that
\begin{align} 
\label{heatbounds2} 
\|\partial_\xi^m \re^{\partial_\xi^2 t} g\|_p \leq Ct^{-\frac{m}2} \|g\|_p, \end{align}
for $p = 1,\infty$, $m = 0,1$, and $g \in L^p(\R)$. Hence, by a standard fixed point argument as in the proofs of~\cite[Theorem~7.1.2 and Proposition~7.2.1]{LUN}, there exist a maximal time $\tau \in (0,\infty]$ and a unique solution $z \in C\big([0,\tau),C_{\mathrm{ub}}(\R) \cap L^1(\R)\big)$ of~\eqref{mildform} such that, if $\tau < \infty$, we have 
$\limsup_{t \to \tau^-} \|z(t,\cdot)\|_{L^1 \cap L^\infty} = \infty$. Let $L > 0$ be the Lipschitz constant of $f$ on $[-\|u_0\|-\|\phi\|_\infty,\|u_0\|_\infty + \|\phi\|_\infty]$. Taking norms in~\eqref{mildform} and using~\eqref{heatbounds2} we arrive at
\begin{align*} 
\|z(t,\cdot)\|_1 \leq C\left(\|z_0\|_1 + \int_0^t \frac{(|c| + L)\|z(s,\cdot)\|_1}{\sqrt{t-s}} \de s\right),
\end{align*}  
for $t \in [0,\tau)$. Hence, Gr\"onwall's Lemma~\cite[Lemma~7.0.3]{LUN} yields a constant $M > 0$, depending only on $C, |c|$, and $L$, such that
\begin{align*} \|z(t,\cdot)\|_1 \leq M \|z_0\|_1,\end{align*}
for all $t \in [0,\tau)$. Combining the latter with $\|z(t,\cdot)\|_\infty \leq \|u_0\|_\infty + \|\phi\|_\infty$ for all $t \in [0,\tau)$ yields $\tau = \infty$. We conclude that $z(t,\cdot) = w(t,\cdot) - \phi = u(t,\cdot + ct) - \phi$, and thus $u(t,\cdot) - \phi$ itself, is $L^1$-integrable for all $t \geq 0$. 
\end{proof}

\bibliographystyle{abbrv}

\bibliography{mybib}

\end{document}